\documentclass[11pt, oneside]{article}   	
\usepackage{expl3}
\usepackage{geometry}     
\usepackage{color} 
\usepackage{esvect}
\usepackage{bbm}
\usepackage{romannum}
\usepackage{ dsfont }

\usepackage{graphicx}				
\usepackage{amssymb}
\usepackage[backgroundcolor=white, bordercolor = yellow, textsize=tiny]{todonotes}
\usepackage[parfill]{parskip}   
\usepackage{hyperref}

\usepackage[english]{babel}
\usepackage{amsmath,amssymb,enumerate,bbm,latexsym,theorem}
\usepackage{parskip}
\usepackage{mathrsfs}

\usepackage{hyperref}
\hypersetup{
	colorlinks,
	citecolor=blue,
	filecolor=blue,
	linkcolor=blue,
	urlcolor=black}

\usepackage{xcolor}
\usepackage{sectsty}
\sectionfont{\color{blue}} 
\subsectionfont{\color{blue}} 
\subsubsectionfont{\color{blue}} 

\newcommand{\support}[1]{\text{supp}(#1)}
\newcommand{\restrict}[2]{#1|_{#2}}

\newcommand{\sobolevbochner}[4]{#1^{#2}(#3, #4)}

\newcommand{\expec}[1]{\mathbb{E}\left[#1\right]}

\newcommand{\CC}{\mathcal{C}}

\newcommand{\DD}{\mathcal D}
\newcommand{\FF}{\mathcal F}

\newcommand{\power}[2]{#1^{#2}}

\newcommand{\elinfty}{L^\infty}
\newcommand{\elp}{L^p}

\newcommand{\eltwo}{L^2}

\newcommand{\holdercont}[2]{C_{#1}^{#2}}

\newcommand{\assign}{:=}

\newcommand{\mathd}{\mathrm{d}}

\newcommand{\tmdummy}{$\mbox{}$}

\newcommand{\tmop}[1]{\ensuremath{\operatorname{#1}}}

\newtheorem{theorem}{Theorem}[section]

\newtheorem{proposition}[theorem]{Proposition}
\newtheorem{definition}[theorem]{Definition}
\newtheorem{lemma}[theorem]{Lemma}
\newtheorem{corollary}[theorem]{Corollary}
\newtheorem{remark}[theorem]{Remark}

\catcode`\>=\active \def>{
\fontencoding{T1}\selectfont\symbol{62}\fontencoding{\encodingdefault}}

\newcommand{\ED}{\domain{\sqrt{-A}}}
\newcommand{\ssp}{\mathscr{H}}

\newcommand{\R}{{\mathbb R}}

\newenvironment{proof}{\noindent\textbf{Proof\ }}{\hspace*{\fill}$\Box$\medskip}

{\theorembodyfont{\rmfamily}}

\newcommand{\lap}{\Delta}
\newcommand{\sobolev}[2]{\mathcal{#1}^{#2}}
\newcommand{\timeDer}{\partial_t}

\newcommand{\norm}[2]{\left\|#1\right\|_{#2}}

\newcommand{\japanbrac}{\langle x \rangle}

\newcommand{\sobolevplain}[2]{#1^{#2}}
\newcommand{\sobolew}[3]{\mathcal{#1}^{#2}(\langle x \rangle^{#3})}
\newcommand{\sobolewb}[4]{\mathcal{B}_{#2,#3}^{#1}(\langle x \rangle^{#4})}
\newcommand{\sobolevb}[3]{\mathcal{B}_{#2,#3}^{#1}}

\newcommand{\innerprod}[2]{\langle #1, #2 \rangle}

\newcommand{\domain}[1]{\mathcal{D}(#1)}

\newcommand{\we}{w_\varepsilon}

\newcommand{\ue}{u_\varepsilon}

\begin{document}
\pagenumbering{arabic}	
		\title{Anderson Hamiltonian and associated Nonlinear Stochastic Wave and Schr\"odinger equations in the full space}
	\author{Baris Evren Ugurcan}
	
	\maketitle
	\begin{abstract}
		In this article, we study the Anderson Hamiltonian in the full space and prove wellposedness of nonlinear stochastic wave equation and NLS with polynomial nonlinearities.
			\end{abstract}
	{\tableofcontents\pagenumbering{arabic}}

\section{Introduction}

Our first purpose in this paper is to prove the essential self-adjointness of the Anderson Hamiltonian
\begin{equation}\label{equ:hoperator}
	H = \Delta + Y(x)-\infty, \ x\in \mathbb{R}^2
\end{equation}
where $Y(x)$ denotes the Gaussian white noise in the full space. Secondly, we  study the wellposedness of the associated stochastic  nonlinear Schrödiger (NLS) and  nonlinear wave equations (NLW) with polynomial nonlinearities
\begin{align*}
	i\partial_{t} u(t,x) = Hu - u|u|^2\\
	\partial_{t}^2 u(t,x) = Hu - u^3.
\end{align*}

 The ``infinity" that appears in \eqref{equ:hoperator} is to indicate that a ``renormalization" is needed to define the operator $H.$  That is, due to low Besov-H\"older regularity of the Gaussian white noise, for a smooth function $\varphi \in \holdercont{c}{\infty}$ we have $\Delta \varphi \in {L}^{2}$ but $\varphi \xi \notin {L}^{2}.$  Therefore, the classical way of constructing the operator first on the smooth functions and then using some extension procedure does not work. We will use paracontrolled distributions \cite{gubinelli2015paracontrolled} to construct a domain for this operator and this will formally correspond to adding an $\infty$ to the operator as noted in \eqref{equ:hoperator}.  We will give a precise meaning to this addition of infinity by also considering the following regularized versions 
\begin{align}
	H_{\varepsilon} = \Delta + Y_{\varepsilon} - c_{\varepsilon}(x), \label{eqn:2dHepsdef} \end{align}
for smooth approximations $Y_{\varepsilon}$ of the noise and noting that $c_{\varepsilon}(x) \rightarrow \infty$ for all $x \in \mathbb{R}^2.$ 

Heuristically, the domain constructed in the framework paracontrolled calculus consists of elements of the form 
\begin{equation}\label{equ:ansatzVeryBasic}
u = u_0 + u^\sharp
\end{equation}
where $u_0$ is the non-smooth component which locally behaves like smoothed noise and $u^\sharp$ is a sufficiently smooth function.  This is the very general idea, though in order to define the operator in $L^2$ one needs to include more correction terms in the ``ansatz \cite{gubinelli2015paracontrolled}"  \eqref{equ:ansatzVeryBasic}, in which case one would still have $u_0$ as  the most irregular term and $u^\sharp$ as the smoothest term with intermediate terms with increasing regularities.

Renormalization for the Anderson Hamiltonian and parabolic Anderson model  have been explored in various cases and different settings until now, to name a few incomplete list  of papers: regularity structures \cite{HL18, labbe_2018}, parabolic Anderson model \cite{gubinelli2015paracontrolled, HL15, HL18, KPW20, BDH19}, eigenvalue asymptotics of the Anderson Hamiltonian \cite{allez_continuous_2015, labbe_2018, ChoukWvZ19, M20}, functional inequalities, embedding results and related stochastic PDEs \cite{DW18, DM19, GUZ20, Z19, allez_continuous_2015, TzVisc20} and lattices, scaling limits and relation to stochastic (super) processes, \cite{CGP17, PR19, PM19} higher order paracontrolled calculus and manifold cases \cite{BB16, M20}. For more references, we refer the reader to \cite{GUZ20}  and bibliography therein.

 On the stochastic PDEs side, initially both regularity structures \cite{hairer_theory_2014} and the parallel development of paracontrolled distributions \cite{gubinelli2015paracontrolled} were applied to the case of parabolic stochastic PDEs but recently there has also been activity to study the non-parabolic SPDEs as well.   Up until now, a fairly general class of parabolic type SPDEs have been studied by these methods.  The non-parabolic cases such as Schr\"odinger and wave equations presents both additional difficulties and new tools.  The additional difficulty comes from the fact that the semigroup associated to these PDEs have worse or no smoothing property at all.  The additional tools are the conservation laws, namely the $L^2$ and energy conservation.  But in order to use the energy conservation, one should show the semi-boundedness of the operator as was done in \cite{GUZ20} for the case of bounded domain.  In this article, we also show this property for the case of the ``bounded part" (that we define below) of the Anderson Hamiltonian in the full space case.

On the (non-parabolic) SPDE side, in \cite{DW18} the authors managed to remove the most irregular term by an exponential transformation inspired from \cite{HL15} and proved an existence and uniqueness result for cubic NLS in $d=2$ on the torus.  In \cite{GUZ20}, also on torus, the authors made sense of the Anderson Hamiltonian, as recalled above and provided a well-posedness theory in $d=2,3$ both for the stochastic NLS and nonlinear wave equations. Later in \cite{DM19}, the authors also solved the NLS with nonlinearity, weaker than the cubic nonlinearity, by adopting the exponential transform method employed in \cite{DW18} to the full space case in $d=2$ by using weights.  There have also been works in other settings until now that treats the nonlinear stochastic wave equations with an additive white noise such as \cite{gubinelli_renormalization_2017,GHO18, GKOT20} to name a few.

In the renormalization of the operator \eqref{equ:hoperator} in the full space, as opposed to the bounded domain,  the main difficulty comes from the fact that the Gaussian white noise $Y$ is only locally in the Besov space $\sobolevb{\alpha}{\infty}{\infty}$ where $\alpha= -\frac{d}{2} - \varepsilon$ almost surely for any positive $\varepsilon >0$.   Therefore, one is faced with a growing (random) potential which makes it a difficult problem to define the Anderson Hamiltonian as a self-adjoint operator and also to study the well-posedness of nonlinear stochastic PDEs.

In order to to overcome this difficulty, our starting point in Section \ref{sec:anderson} is to use the decomposition from \cite{gubinelliHofmanova2019global}  of the Gaussian white noise
$ Y = \xi+ \eta.$
It was further shown in \cite{gubinelliHofmanova2019global} that the partial noise (or so called the ''bounded part'') $\xi$ lies globally in a distributional Besov space and the remaining part $\eta$ (or so called the ''unbounded part'') is a potential with a mild growth at infinity. Our overall strategy, in Section \ref{sec:anderson}, is to first consider the following operator whose potential term is the partial noise $\xi$
\begin{equation}\label{equ:operatorA}
	A:= \Delta + \xi -\infty
\end{equation} 
and construct a domain for this operator in the setting of paracontrolled distributions \cite{gubinelli2015paracontrolled} and realize $A$ as a semibounded self-adjoint operator (up to addition of a constant). Our approach in the construction of the domain for the operator $A$ is inspired from the paper \cite{GUZ20} where we, together with M. Gubinelli and I. Zachhuber, developed methods building on the work \cite{allez_continuous_2015} to construct the Anderson Hamiltonian on two and three dimensional torus with Gaussian white noise.  In a nutshell, we generalize the methods of  \cite{GUZ20}  to the operator $A$ as defined above, namely to the full space case with the (partial) noise $\xi$ (as opposed to Gaussian white noise on the torus \cite{GUZ20, allez_continuous_2015}) and in the setting of weighted Besov spaces.  One major difficulty and an important part of the renormalization in our setting is showing the convergence of the so called enhanced noise 
\begin{equation}\label{equ:introEnhanced}
	(\xi_\varepsilon, \xi_\varepsilon \circ (1 - \Delta)^{- 1} \xi_\varepsilon-  c_\varepsilon(x)) \rightarrow (\xi, \Xi_{2})
\end{equation} 
where formally $\Xi_{2} =  \xi \circ (1 - \Delta)^{- 1} \xi- \infty$ is the iterated noise which plays a similar role here that the ``second order process" plays in rough path theory.  This was  proved for the Gaussian white noise on the torus in \cite{allez_continuous_2015} (in which case $c_\varepsilon(x)$ are constants) and in this paper we prove this result in the case of the full space $\mathbb{R}^2$ and the partial noise $\xi.$  We  show the norm resolvent convergence of the regularized operators $A_\varepsilon$ to $A$ which makes the appearance of the infinity in \eqref{equ:hoperator} rigorous.   We also obtain full space and partial noise $\xi$-counterparts of the $L^p$-embedding theorems for the domain and the form domain and  functional inequalities which was obtained on the torus in \cite{GUZ20}, such as the well known Brezis-Gallouet inequality. All these results are also essential when we later consider well-posedness of SPDEs in Section \ref{sec:spde}. In particular  the norm resolvent convergence acts as an approximation tool, since the domain of $A$ does not contain any smooth functions.  

After the construction of the operator $A$, we also add on the growing potential $\eta$ (the unbounded part of the noise) and set out to construct the full Anderson Hamiltonian 
$$H= A + \eta $$ 
and prove essential self-adjointness on a well defined domain by using the Faris-Lavine Theorem \cite{FL94}.   Heuristically speaking, this theorem is typically used to prove essential self-adjointness of Hamiltonians whose potential term does not go below $ -|x|^2.$   The self-adjointness is then proved by establishing a commutator estimate with a suitably constructed auxiliary operator $N$ and the Hamiltonian $H$.    Namely, with the choice of the operator  $ N := H+ c|x|^2,$ that we prove to be self-adjoint and positive, we prove the following  commutator estimate 
\begin{align*}
	\pm i \left(  \innerprod{N f}{Hf}  -\innerprod{ H f}{Nf} \right) \leq C \innerprod{Nf}{f}.
\end{align*}
over a well-chosen dense subspace of the domain $\domain{A}.$  In the proof of the commutator estimate we also benefit from our construction of the operator $A$ specifically in the weighted setting, in that both the dense subspace that the commutator estimate holds and the domain of the self-adjoint operator $N$ are weighted spaces. In conclusion, we establish that the full Anderson Hamiltonian $H$ can be defined as an essentially self-adjoint operator on the full space $\mathbb{R}^2$ over a well defined domain.

In Section \ref{sec:spde}, we move onto the well-posedness of stochastic PDEs in the full space $\mathbb{R}^2$, namely the stochastic NLS and NLW equations with multiplicative noise.  Our general strategy here is to think $H = A+ \eta$, namely to treat the unbounded term as a potential. We study the solutions of the following random Cauchy problems on the full space in 2d
\begin{align*}
i\partial_{t} u(t,x) = Hu - u|u|^p\\
\partial_{t}^2 u(t,x) = Hu - u^3
\end{align*}
as a limit of the solutions to the associated regularizations
\begin{align*}
i\partial_{t} \ue(t,x) = H_\varepsilon \ue - \ue|\ue|^p\\
\partial_{t}^2 \ue(t,x) = H_\varepsilon \ue - \ue^3
\end{align*}
for general power nonlinearities that includes the case $p \geq 2,$ though for simplicity we write our proofs for the classical cubic case.   Initially, we are able to make sense of these PDEs and their respective conserved quantities by using our construction of the operator $A$ as a self-adjoint semibounded operator. We show existence of solutions by first establishing a priori estimates by intensively using the Sobolev space estimates, semi-boundedness and Brezis-Gallouet inequality that we show for the domain and form domain of operator $A$ in the full space, that we undertake in the Subsection \ref{sec:functionIneq} of the paper.  We are able to get wellposedness for general power nonlinerities with  initial data both from the form domain $\domain{\sqrt{-A}}$ and the domain $\domain{A}$ and the solutions satisfy energy conservation.  While doing so, we are faced with two important problems in our setting as opposed to the bounded domain case \cite{GUZ20}:  the first one is the unboundedness of the space so that the compactness arguments become much more difficult. The second  is the unboundedness of the potential $\eta$ which forces us to take advantage of some form of ''localization'' property of the respective PDE.  In the case of stochastic NLS,  similar in spirit to \cite{DM19} where stochastic NLS with $p<2$ was treated, we observe that a localized initial data stay localized.   Namely, we use our construction of the operator $A$ in weighted spaces to quantify the localization of the initial data and the solution. In conclusion, we show the (weak) existence of $\domain{A}$-solution in Theorem \ref{thm:weakExist} and then strong $\domain{\sqrt{-A}}$-solutions in Theorem \ref{thm:mainwellposed} for the stochastic NLS.

For the stochastic NLW equation, we need more refined  properties of the Anderson Hamiltonian $H$.  As known, when defined through the paracontrolled distributions the domain of $H$ does not contain  smooth and/or compactly supported functions.  However, for the wave equation in the full space we also have to exploit the ``finite speed of propagation" property, for which we need to be able to form compactly supported initial data.   In the first part of the paper, more precisely in Section \ref{sec:anderson},  we overcome this difficulty and formulate a technique and prove a ``product-type formula" for the renormalized Anderson Hamiltonian in our setting.  Later in Section \ref{sec:spde}, this paves the way to our treatment of the cubic and also possibly higher order nonlinearities for the stochastic wave equation.   Namely, by using this formula, we first prepare compactly supported initial data from a domain element and prove the convergence of relevant functionals of this initial data in Theorem \ref{prop:convInitialData}.  Later on by combining this with  finite speed of propagation and compactness methods along with the convergence of regularized operators $H_\varepsilon$,  we finally prove the existence, uniqueness and energy conservation in Theorem \ref{corr:convWaveSecDer}.  

\textbf{Notation:}

Throughout the paper for the Hilbert space $L^{2}(\mathbb{R}^2)$ we simply use $L^2.$ For all the other (weighted) spaces  we use calligraphic notation, for example $\sobolew{L}{2}{2}, \sobolew{H}{2}{2}$ and $\sobolewb{\alpha}{p}{q}{\delta},$ where $\japanbrac$ denotes the Japanese bracket $\japanbrac := \sqrt{1+ |x|^2}.$    The Sobolev spaces and Besov-H\"older spaces are respectively defined as $\sobolew{H}{s}{\delta}:= \sobolewb{s}{2}{2}{\delta}$ and $\sobolew{C}{\alpha}{\delta}:= \sobolewb{\alpha}{\infty}{\infty}{\delta}, 0<\alpha<1.$  We use $\domain{A}$ to denote the domain of an unbounded operator $A.$  The notation $x\sim y$ indicates the relationship
\[
c x \leq y \leq Cx
\]
for fixed constants $C, c >0.$ See also Appendix for more details on definitions and notation. In particular, for the notation and background information on function spaces, Fourier transform,   paraproducts (e.g. ``$\prec$", ``$\succ$", ``$\circ$", ``$\Delta_n$"), Littlewood-Paley theory and commutator estimates we refer the reader to the Section \ref{sec:appBasicHarm}, of the Appendix.

The capital letters $A, T, H, N$ are reserved for operators throughout the paper. These notations are uniform only with respect to trivial operations (addition of a constant and multiplication by $-1$) and re-defined freely in different sections,  to prevent the introduction of overly decorated notations for straightforward versions of the same operator.  For instance, $H$ denotes the negative of the same operator  (namely $-H$) in Section \ref{sec:anderson}, as re-defined there.  In Section \ref{sec:selfAdjoint} the notation $T$ is used instead of $A$ and in the end a constant is added to $T$ to make it into $A,$ so that its negative (i.e. $-A$) is a positive operator. After this change, in the rest of the paper and in the PDE section, namely in Section \ref{sec:spde}, the notation $A$ is used.

\textbf{Acknowledgments:} The author gratefully thanks Professor M. Gubinelli for many discussions and acknowledges partial  support (until August 2019) from CRC 1060.

\section{Anderson Hamiltonian on the full space $\mathbb{R}^2$}  \label{sec:anderson}
We recall a few definitions and constructions which are used in the rest of the section, explain the main ideas and state the main results of this section.  We start with the definition of the Gaussian white noise $Y$ in the full space and state the Wiener-Ito representation \cite{gubinelli2017kpz}.
\begin{definition}
	White noise is a family $\{ Y(\phi), \phi \in L^2(\mathbb {R}^d )\}$ of Gaussian random variables with the covariance structure
	\[
	\mathbb{E}( Y (\phi_1) Y(\phi_2)) = \langle  \phi_1,  \phi_2 \rangle_{L^2}
	\]
	for $\phi_1, \phi_2 \in L^2(\mathbb {R}^d )$.  Furthermore,  the Gaussian white noise has the following Wiener--Ito representation 
	\begin{align}\label{equ:whitenoiseWienerIto}
		Y(x) = \int e^{2\pi \theta \cdot x} dW(\theta)
	\end{align}
	where $dW$ is a random measure with covariance
	\begin{equation}\label{equ:covRandomMeas}
		\mathbb{E}\left(dW(\theta) dW(\kappa)\right) = \delta(\theta + \kappa) d\theta d\kappa.
	\end{equation}
\end{definition}
In  \cite{gubinelliHofmanova2019global} the  decomposition $Y = \xi+ \eta$  for   Gaussian white noise was introduced such that 
\begin{equation}\label{equ:noisedecomp}
\begin{aligned}
\eta  &:= \sum_{n=1}^\infty w_n \sum_{k=-1}^{L_n-1} \Delta_k Y = \sum_{n=1}^\infty w_n \Delta_{< L_n} Y  \\
\xi &:=  \sum_{n=1}^\infty w_n \sum_{k=L_n}^\infty \Delta_k Y= \sum_{n=1}^\infty w_n \Delta_{\geq L_n} Y  
\end{aligned}
\end{equation}
where we used the  notation in Definition \ref{def:deltaCutoffMap} and $L_n$ is a function $\mathbb{N} \rightarrow \mathbb{R}$ to be specified later.  Furthermore, the regularities of $\xi$ and $\eta$ were also determined which we cite in the following suitable form for us.
\begin{theorem}[\cite{gubinelliHofmanova2019global}] \label{thm:decomphighpartHolder}
Let $\alpha:= -1-\varepsilon_0$ and  $\varepsilon_0, \kappa >0$ be constants which can be chosen to be arbitrarily small. The regularities in the decomposition \eqref{equ:noisedecomp}  satisfy	$\xi  \in \sobolev{C}{\alpha}$ and $\eta \in L^{\infty} \left(\frac{1}{1+ |x|^\kappa}\right).$ Consequently it follows, in the sense of quadratic forms, that  
\begin{equation}\label{rem:growthBddPart}
\eta(x) \leq c'\left( |x|^2 + 1\right)
\end{equation}
for some constant $c'>0.$
\end{theorem}

By using this decomposition, our strategy to define the Anderson Hamiltonian is to consider the following operators
\begin{align*}
	A &:=\Delta + \xi -\infty\\
	H &:= -A - \eta\\
	N &:= H + c|x|^2
\end{align*}
for some constant $c>0.$   Namely, we first add the Laplacian only the bounded part of the noise as a potential term and then construct the domain of this operator through paracontrolled distributions \cite{gubinelli2015paracontrolled}.  Then, in the second step we also add as a potential also the bounded part and form the operator $H$ as above. Finally, we benefit from the mild growth of $\eta$ and conclude the essential self-adjointness by showing a commutator estimate as in the Faris-Lavine theorem:
\begin{proposition}\cite{FL94} 
	Let $H$ be  symmetric and $N \geq 1$ be  positive self-adjoint operators with $\DD (N) \subset \DD(H)$  which  satisfy the following estimate of their associated quadratic forms
	\begin{equation}\label{equ:farisLavComm}
		\pm \langle [H, N]f, f \rangle:= \pm i \left(  \innerprod{N f}{Hf}  -\innerprod{ H f}{Nf} \right) \leq c \innerprod{Nf}{f}
	\end{equation}  
	for some constant $c>0.$ Then, it follows that $H$ is an essentially  self-adjoint operator over $\DD (N)$ .
\end{proposition}

In Section \ref{sec:selfAdjoint} we construct a domain $\domain{A}$ for the operator $A$ in the setting of paracontrolled calculus and then in Section \ref{sec:farisLavine} construct the operator $N$ as summarized in the following result.

\begin{theorem}
The operator $A$ can be defined (up to addition of a constant) as a semibounded self-adjoint operator on its domain $\domain{A} \subset L^2.$ For  $\delta\geq 0$ and $s \in (0, \alpha+2],$ there exists a well defined the map $\Gamma:  \sobolew{H}{s}{\delta} \rightarrow \sobolew{H}{s}{\delta}$ which satisfies the following estimates
\begin{align*}
	||\Gamma f||_{\elinfty} &\leq C_{N, \Xi} || f||_{\elinfty}\\
	||\Gamma f||_{\sobolew{H}{s}{\delta}} &\leq C_{N, \Xi} || f||_{\sobolew{H}{s}{\delta}}.
\end{align*} 
and the domain can be characterized as the image of this map i.e. $\domain{A} = \Gamma(\sobolev{H}{2})$.  Furthermore, the operator $N$ can be defined as a self-adjoint positive operator on the domain  
		\begin{equation}
	\mathcal{C}_2 := \{f \in \domain{A} ~|~ \japanbrac^2 f \in L^2 \} = \domain{A} \cap \sobolew{L}{2}{2}.
\end{equation}
\end{theorem}
Additionally, we prove the norm resolvent convergence of the regularized operators $A_\varepsilon,$ defined in terms of the regularizations $\xi_{\varepsilon}$ of the partial noise, to the operator $A.$ This result becomes also very useful later in Section \ref{sec:spde} as an approximation tool, as the domain $\domain{A}$ does not include the smooth functions.

In Section \ref{sec:farisLavine}, we prove the following result  which constructs  the Anderson Hamiltonian on the full space $\mathbb{R}^2$ by showing a commutator estimate.
\begin{theorem}
For some constant $C>0$ and  for all $f \in\CC_2,$  $H$ and $N$ satisfies the following commutator estimate
	\begin{align*}
		\pm i \left(  \innerprod{N f}{Hf}  -\innerprod{ H f}{Nf} \right) \leq C \innerprod{Nf}{f}.
	\end{align*}
which in turn implies that $H$ is an essentially self-adjoint operator on $\CC_2.$
\end{theorem}
The primary difficulty of proving this theorem is unbounded nature of the Gaussian white noise in the full space unlike the bounded domain case \cite{GUZ20, allez_continuous_2015}.  Another major difficulty comes from the fact that the operator $A$ is defined not with the full noise $Y$ but with the partial noise $\xi$ and the construction of the domain needs to be done in the weighted setting. So, the whole machinery and constructions which was done in the bounded domain case \cite{GUZ20, allez_continuous_2015} now needs to be generalized to the weighted setting and the partial noise $\xi.$  One particular difficulty in our case is the renormalization of the ''second order process'' or the enhanced noise 
$$
 \xi \circ (1 - \Delta)^{- 1} \xi$$ 
 for the partial noise $\xi,$ which appears in the definition of the operator $A$ and forms a major portion of the renormalization.  Namely, in Section \ref{sec:selfAdjoint} we prove the following  important result for the partial noise $\xi$ regarding the convergence of the regularizations of the enhanced noise.
 \begin{theorem} 
 	There exists smooth functions $c_\varepsilon(x): \mathbb{R}^2 \rightarrow \mathbb{R}$ which satisfies $c_\varepsilon(x) \rightarrow\infty$ for all $x \in \mathbb{R}^2$ and  a limit point  $\Xi \in \CC^{\alpha} \times \CC^{2
 		\alpha + 2}$ such that as $\varepsilon\rightarrow 0$ the convergence 
 	\[
 	\Xi^\varepsilon = (\xi_\varepsilon, \xi_\varepsilon \circ (1 - \Delta)^{- 1} \xi_\varepsilon-  c_\varepsilon(x)) \rightarrow \Xi
 	\]
 	holds  in \(L^{p}\left(\Omega,  \CC^{\alpha} \times \CC^{2
 		\alpha + 2}\right)\) for sufficiently large \(p>1\) and almost surely in \( \CC^{\alpha} \times \CC^{2
 		\alpha + 2}\). 
 \end{theorem}

In Section \ref{sec:functionIneq}, using the seminboundedness of the operator we define the form domain $\domain{\sqrt{-A}}$ and then also prove $L^p$-embedding properties of the domain and the form domain. Then, we prove some functional inequalities such as the well known Brezis-Gallouet inequality for the operator $A$ on the full space $\mathbb{R}^2$.  These results become also very useful in the treatment of wellposedness for stochastic NLS and NLW equations later in Section \ref{sec:spde}. Namely, we prove full space counterparts (this time for the partial noise $\xi$) for the results which was proved in \cite{GUZ20} for the bounded domain case.

Localization of domain elements with a compactly supported function become very important in several places in this paper, for instance when we want to use finite speed of propagation for the stochastic nonlinear wave equation in Section \ref{sec:spde}.  Accordingly, in Section \ref{sec:extrapolatedOP} we show in what way the elements in $\domain{A}$ can be localized and also prove a precise product formula, which we summarize as the following result.
\begin{theorem}
For fixed $u \in \domain{A}$ and $\phi \in C_c^\infty$, $\phi u$  is in $\domain{\sqrt{-A}}$.  In addition, for $\Psi \in\domain{\sqrt{-A}}$  the following bound holds 
\[
|\innerprod{\phi u}{A \Psi}| \leq \norm{Au}{\eltwo} \norm{A^{1/2} \Psi}{\eltwo} \norm{\phi}{W^{2,\infty}}.
\]
Furthermore, the following formula
\begin{align}
	A (\phi u ) = 	 \phi  A u + 2 \nabla \phi  \nabla u + u \Delta\phi  
\end{align}
holds  in $(\domain{\sqrt{-A}})^*$.
\end{theorem}

\subsection{Renormalization of the enhanced noise and construction of $\domain{A}$}\label{sec:selfAdjoint}
Firstly, we fix the following notation for smoothness indices
\begin{equation}\label{def:gammalp}
	\begin{aligned}
		\alpha &:= -1-\varepsilon_0\\
		\gamma & := \alpha +2
	\end{aligned}
\end{equation}
as in Theorem \ref{thm:decomphighpartHolder}. In this section the operator, with slight change of notation,
\begin{align}\label{equ:tOp}
	T  := \Delta  + \xi  - \infty \end{align}
 will be constructed over a domain $\domain{T}$ and  then by addition of a proper constant in Definition \ref{def:setconstant}, it will finally be denoted by the self-adjoint semibounded operator $A$ for the rest of the paper.
 
 Our purpose is to also prove the norm resolvent convergence of the regularizations 
\begin{align}\label{equ:tepsilonOp}
	T_\varepsilon  := \Delta  + \xi_\varepsilon  - c_\varepsilon(x). \end{align}
where the approximations are defined as
\begin{equation}\label{equ:noisebddpartapp}
	\begin{aligned}
		Y_\varepsilon (x) &:= \int e^{2\pi \theta \cdot x} \sigma(\varepsilon |\theta|)dW(\theta)\\
		\xi_\varepsilon &:= \sum_{n=1}^\infty w_n \Delta_{\geq L_n} Y_\varepsilon\\
		c_\varepsilon(x) &:= \mathbb{E}\left(\xi_\varepsilon(x) \circ (1 - \Delta)^{- 1} \xi_\varepsilon(x)\right).
	\end{aligned}
\end{equation}
where we used the Ito-Wiener representation \eqref{equ:whitenoiseWienerIto} and  $\sigma$ is a smooth radial function on $\mathbb{R} \backslash\{0\}$ with compact support such that
$$
\lim _{x \rightarrow 0} \sigma(x)=1.
$$
The function $c_\varepsilon(x)$ satisfies that $c_\varepsilon(x) \rightarrow \infty$ for every $x\in \R^2$ as $\varepsilon \rightarrow 0,$ which is to be clarified later. In order to do these, we generalize the methods of the bounded space case \cite{GUZ20} to the weighted setting and to the partial noise $\xi.$ 

In order to motivate the definition of the domain and how $T$ is defined on this domain we calculate with the regularized elements
\[
T_\varepsilon \ue = \left(\Delta  + \xi_\varepsilon  - c_\varepsilon(x) \right)\ue
\]
for a $\ue$ satisfying the following ansatz
\begin{equation}\label{equ:simpleAnsatzRep}
	u_\varepsilon =\Delta_{> N}(u_\varepsilon\prec X_\varepsilon +B_{\Xi_\varepsilon}(u_\varepsilon))+u^\sharp
\end{equation}
where we defined 
\begin{align*}
	X_\varepsilon &\assign  (1 - \Delta)^{- 1} \xi_\varepsilon\\
	\Xi_{2}^\varepsilon&:=X_\varepsilon\circ \xi_\varepsilon + c_\varepsilon(x)\\
	B_{\Xi^\varepsilon} (f) &\assign (1 - \Delta)^{- 1} (\Delta f \prec X_\varepsilon + 2 \nabla f
	\prec \nabla X_\varepsilon + \xi_\varepsilon \prec f + f \prec \Xi_2^\varepsilon).
\end{align*}
We first compute
\begin{align}\label{equ:cancelAnsatz0}
	\Delta u_\varepsilon  &= \Delta (\Delta_{> N}(u_\varepsilon\prec X_\varepsilon +B_{\Xi_\varepsilon}(u_\varepsilon))+u^\sharp ). 
\end{align}
We go term by term. We  obtain
\begin{equation}
	\begin{aligned}
		\label{equ:cancelAnsatz1}\Delta (\Delta_{> N}(u_\varepsilon\prec X_\varepsilon)) &= \Delta_{> N}(u_\varepsilon\prec X_\varepsilon) - \Delta_{> N} (u_\varepsilon \prec \xi_\varepsilon) +  \Delta_{> N}(\Delta u_\varepsilon\prec X_\varepsilon) \\
		&+ 2  \Delta_{> N}(\nabla u_\varepsilon\prec  \nabla X_\varepsilon) \end{aligned}\end{equation}
for the first term in \eqref{equ:cancelAnsatz0} and 
\begin{equation} 
	\begin{aligned}\label{equ:cancelAnsatz2}\Delta (\Delta_{> N} B_{\Xi_\varepsilon}(u_\varepsilon) )&=\Delta_{> N} (B_{\Xi_\varepsilon}(u_\varepsilon)) - \left( \Delta_{> N}(\Delta u_\varepsilon\prec X_\varepsilon) - 2  \Delta_{> N}(\nabla u_\varepsilon\prec  \nabla X_\varepsilon) \right. \\  & \left.- \Delta_{>N} \xi_\varepsilon \prec u_\varepsilon - \Delta_{>N} (u_\varepsilon \prec \Xi_2^\varepsilon)  \right). \end{aligned}\end{equation}
for the second term.
By the definition of a paraproduct, $f \prec g$ behaves like $g$ in terms of regularity.  This implies that the term $\Delta_{> N} (u_\varepsilon \prec \xi_\varepsilon)$ would have  the same (low) regularity as the noise $\xi$ in the limit as $\varepsilon \rightarrow 0$.  In order to see that this problem disappears with the help of the representation \eqref{equ:simpleAnsatzRep}, we also calculate the product and  resonant-product terms as
\begin{align*}
	u_\varepsilon \xi_{\varepsilon} &= u_\varepsilon \prec \xi_{\varepsilon} +  \xi_{\varepsilon}  \prec u_\varepsilon + u_\varepsilon \circ \xi_{\varepsilon}\\
	u_\varepsilon \circ \xi_{\varepsilon} - c_\varepsilon(x)\ue(x)&= (\Delta_{> N}(u_\varepsilon\prec X_\varepsilon +B_{\Xi_\varepsilon}(u_\varepsilon))+u^\sharp) \circ \xi_{\varepsilon}\\
	&= C_N (u_\varepsilon, \xi_\varepsilon, X_\varepsilon) + \ue(X_\varepsilon\circ \xi_\varepsilon -c_\varepsilon(x)) + (\Delta_{> N}B_{\Xi_\varepsilon}) \circ \xi_\varepsilon + u^\sharp \circ \xi_\varepsilon\\
	& = C_N (u_\varepsilon, \xi_\varepsilon, X_\varepsilon) + \ue\Xi_2^\varepsilon + (\Delta_{> N}B_{\Xi_\varepsilon}) \circ \xi_\varepsilon + u^\sharp \circ \xi_\varepsilon
\end{align*}
where $C_N (u, \xi, X) := \left(\Delta_{> N} (u\prec X)\right))\circ \xi- u(X\circ \xi)$ is the ``modified commutator". Now, we can  observe the important cancellation of the problematic  term $\Delta_{> N} (u_\varepsilon \prec \xi_\varepsilon)$ in the computation
\begin{align*}
	& \Delta (\Delta_{> N}(u_\varepsilon\prec X_\varepsilon)) +u_\varepsilon \xi_{\varepsilon} \\  &=\Delta_{> N}(u_\varepsilon\prec X_\varepsilon) - \Delta_{> N} (u_\varepsilon \prec \xi_\varepsilon) +  \Delta_{> N}(\Delta u_\varepsilon\prec X_\varepsilon) + 2  \Delta_{> N}(\nabla u_\varepsilon\prec  \nabla X_\varepsilon)\\
	&+ u_\varepsilon \prec \xi_{\varepsilon} +  \xi_{\varepsilon}  \prec u_\varepsilon + u_\varepsilon \circ \xi_{\varepsilon}\\
	&= \Delta_{> N}(u_\varepsilon\prec X_\varepsilon)  +  \Delta_{> N}(\Delta u_\varepsilon\prec X_\varepsilon) + 2  \Delta_{> N}(\nabla u_\varepsilon\prec  \nabla X_\varepsilon)\\
	&+ \Delta_{\leq N}\left(u_\varepsilon \prec \xi_{\varepsilon}\right) +  \xi_{\varepsilon}  \prec u_\varepsilon + u_\varepsilon \circ \xi_{\varepsilon}.
\end{align*}
Together with this important cancellation and also cancellations between \eqref{equ:cancelAnsatz1} and \eqref{equ:cancelAnsatz2} we can re-write our computation as 
\begin{equation}\label{equ:tuepsilon}
	\Delta u_\varepsilon + \xi_\varepsilon\ue + c_\varepsilon\ue  = \Delta u^\sharp + u^\sharp  \circ \xi_\varepsilon + G(u_\varepsilon)
\end{equation}
where we defined  
\begin{align*}
	G(u_\varepsilon)&:=  \Delta_{\le N}(\xi \prec \ue + \ue \prec \xi + \ue \prec  \Xi_{2}^\varepsilon )+  \Delta_{\le N} (u_\varepsilon \prec X_\varepsilon) + \Delta_{> N} B_{\Xi^{\varepsilon}}(u_\varepsilon)\\ &+ \left( \Delta_{> N} B_{\Xi^{\varepsilon}}(u_\varepsilon)\right) \circ \xi_\varepsilon + C_N (u_\varepsilon, \xi_\varepsilon, X_\varepsilon)  + u \succcurlyeq \Xi_2^\varepsilon.
\end{align*}
In order to make sense of \eqref{equ:tuepsilon} in the limit as $\varepsilon \rightarrow0,$  we will show in Theorem \ref{thm:2dren} that $(\xi_\varepsilon, X_\varepsilon\circ \xi_\varepsilon + c_\varepsilon(x)) \rightarrow (\xi,\Xi_2)$ with proper Besov regularities so that the right handside of \eqref{equ:tuepsilon} lies in $L^2.$  This computation motivates the following definition of the domain for $T$ and how it is defined thereon.
\begin{definition}\label{def:2dDXi}
	For $\alpha <0$ and $0<\gamma<1$ as in \eqref{def:gammalp}, we define the space of functions
	paracontrolled by the enhanced noise $\Xi$ as follows
	\begin{equation}\label{equ:ansatzmain}
		\domain{T} \assign \{u \in \sobolev{H}{\gamma} ~\text{s.t.}~ u =\Delta_{>N}  \left(u \prec X + B_{\Xi} (u)\right) + u^{\sharp}, ~\text{for}~ u^{\sharp} \in \sobolev{H}{2}  \}
	\end{equation}
	where $X = (1 - \Delta)^{- 1} \xi \in \CC^{\alpha + 2}$ and $B_{\Xi} (u) \in \sobolev{H}{2}$.
	
	Naturally, we equip $\domain{T}$ with the following inner-product
	\begin{equation}\label{equ:domainInnerProd}
		\innerprod{u}{w}_{\domain{T}}: = \innerprod{u}{w}_{\sobolev{H}{\gamma}} + \innerprod{u^\sharp}{w^\sharp}_{\sobolev{H}{2}}
	\end{equation}
	
	Over the domain $\domain{T},$ we define the  operator $T:\domain{T} \to L^2$ as
	\begin{align} \label{equ:defandersonHam}
		T u : = \Delta u^{\sharp} + u^{\sharp} \circ \xi + G (u),\end{align}
	where in $u \xi= u \prec \xi +  \xi \prec u+ u\circ \xi$
	we have defined
	\begin{align*} u \circ \xi&:= (\Delta_{> N}(u\prec X +B_{\Xi}(u))+u^\sharp) \circ \xi\\
		&= C_N (u, \xi, X) + u\Xi_2  + (\Delta_{> N}B_{\Xi}) \circ \xi + u^\sharp \circ \xi   \\
		G (u) &: = \Delta_{\le N} (u \prec \xi + u \succ \xi + u \prec \Xi_{2})\\ &+ \Delta_{>
			N} (- B_{\Xi} (u) - u \prec X + u \succcurlyeq \Xi_2 + C_N (u,
		X, \xi) + B_{\Xi} (u) \circ \xi). \end{align*}
	Recall that $C_N (u, \xi, X) := \left(\Delta_{> N} (u\prec X)\right))\circ \xi- u(X\circ \xi)$ is the modified commutator.
\end{definition}
Now we set out to show the convergence $(\xi_\varepsilon, X_\varepsilon\circ \xi_\varepsilon + c_\varepsilon(x)) \rightarrow (\xi,\Xi_2)$ in the space $\CC^{\alpha} \times \CC^{2
	\alpha + 2}$. In the following definition, we  introduce the space of enhanced noise analogous to the bounded domain case \cite{GUZ20} which used the Gaussian white noise on the torus.  One distinguishing feature of our case is that the ``constants" $c_\varepsilon$ are no longer constants but functions depending on $x \in \mathbb{R}^2.$  
\begin{definition} \label{def:2dencNoise}
	Let $\alpha \in \mathbbm{R}$ be such that $\ 2\alpha + 2>0.$ For
	$\mathscr{E}^{\alpha} \assign \CC^{\alpha} \times \CC^{2
		\alpha + 2},$ we define the space of enhanced noise, namely $ \mathscr{X}^{\alpha},$ as the closure of the set $\{
	(\eta, \eta \circ (1 - \Delta)^{- 1} \eta + \phi) : \eta \in \CC^{\infty}
	(\mathbbm{R}^2),\phi  \in \holdercont{c}{\infty}(\mathbbm{R}^2)\}$ with respect to the  $\mathscr{E}^{\alpha}$-topology.
\end{definition}
We are ready to state and prove one of the main theorems of this subsection. Namely, we establish the Besov-H\"older regularity of the enhanced noise and prove convergence of regularizations in the enhanced noise space $\mathscr{X}^{\alpha}$.  This main renormalization theorem will be the first step to generalize the methods in \cite{GUZ20} to the operator $T$ (defined with the partial noise $\xi$) in the full space case.  
\begin{theorem} \label{thm:2dren}
	There exists a limit point $\Xi =: (\Xi_1, \Xi_2) \in \mathscr{X}^{\alpha}$ such that as $\varepsilon\rightarrow 0$ the convergence 
	\[
	\Xi^\varepsilon = (\xi_\varepsilon, \xi_\varepsilon \circ X_\varepsilon-  c_\varepsilon(x)) \rightarrow \Xi
	\]
	holds  in \(L^{p}\left(\Omega, \mathscr{E}^{\alpha}\right)\) for sufficiently large \(p>1\) and almost surely in \(\mathscr{E}^{\alpha},\) where $\kappa_0 := -2\alpha-2 >0$ can be chosen arbitrarily small. 
\end{theorem}
\begin{proof}
	For $c_\varepsilon(x) = \expec{ \xi_\varepsilon \circ X_\varepsilon}$ we want to show the existence of $\Xi_2$ and the convergence
\[
\Xi_2^\varepsilon:= \xi_\varepsilon \circ X_\varepsilon-  c_\varepsilon(x) \rightarrow \Xi_2
\]
almost surely in the Besov-H\"older space $\sobolev{C}{-\kappa_0-\frac{1}{p}}(\mathbb{R}^2).$  

In order to do so, for $\varepsilon, \delta >0,$	we want to initially investigate the size of the q-th Littlewood-Paley block of the difference $ (\Xi_2^\varepsilon - \Xi_2^\delta) (x)$  and show an estimate of the form 
\begin{equation}
	\begin{aligned}\label{equ:exptwo}
		&\expec{|\Delta_q( \Xi_2^\varepsilon - \Xi_2^\delta) (x)|^2}\lesssim   2^{q\cdot\kappa_0} |\varepsilon - \delta|^{m}.
	\end{aligned}
\end{equation}
for an arbitrarily small $\kappa_0>0.$ Then we want use hypercontractivity to obtain a similar estimate for an arbitrarily large $p>0$ (as opposed to $p=2$) and finally conclude the result by Kolmogorov-Chentsov theorem.

By using the definition of a Littlewood-Paley block \eqref{equ:LPblock} and Wiener-Ito representation \eqref{equ:whitenoiseWienerIto} we calculate
\begin{equation}\label{equ:psi1psi2}
\begin{aligned}
	&\expec{|\Delta_q( \Xi_2^\varepsilon - \Xi_2^\delta) (x)|^2}\\
	&= \sum_{|A-B| \leq 1} \sum_{n,m = 1}^\infty  \sum_{|A'-B'| \leq 1} \sum_{n',m' = 1}^\infty \int_{\left[\mathbb{R}^2 \right]^{\otimes 6}} \rho(\power{2}{-q}|z_1 + z_2 + \tau_1 + \tau_2|)\rho(\power{2}{-q}|z'_1 + z'_2 + \tau'_1+ \tau'_2| ) \\
	&\rho(\power{2}{-A}|z_1 + \tau_1|)
	\rho(\power{2}{-A'}|z'_1 + \tau'_1|) \frac{\rho(\power{2}{-B}|z_2 + \tau_2|)}{1+|z_2+ \tau_2|^2}   \frac{\rho(\power{2}{-B'}|z'_2 + \tau'_2|)}{1+ |z'_2 + \tau'_2|^2}\\
	&\rho_{\geq L_n} (|\tau_1|) \rho_{\geq L_m} (|\tau_2|)    \rho_{\geq L_{n'}} (|\tau'_1|) \rho_{\geq L_{m'}} (|\tau'_2|)\\
	& \left[\sigma (\varepsilon|\tau_1|) \sigma (\varepsilon|\tau_2|) - \sigma (\delta|\tau_1|) \sigma (\delta|\tau_2|)\right] \left[\sigma (\varepsilon|\tau'_1|) \sigma (\varepsilon|\tau'_2|) - \sigma (\delta|\tau'_1|) \sigma (\delta|\tau'_2|)\right] \\
	&\hat{w}_n(|z_1|) \hat{w}_m(|z_2|)  \hat{w}_{n'}(|z'_1|) \hat{w}_{m'}(|z'_2|) \power{e}{2 \pi i x \cdot (z_1 + z_2+ z'_1 + z'_2  + \tau_1 + \tau_2+  \tau'_1 + \tau'_2)}  \\
	& \left\{ \delta(\tau_1 + \tau'_1) \delta(\tau_2 + \tau'_2)  + \delta(\tau'_1 + \tau_2) \delta(\tau_1 + \tau'_2) \right\}\\
	&dz_1 dz_2 dz'_1 dz'_2 d\tau'_1 d\tau_2 d\tau_1 d\tau'_2 := \Psi_1 + \Psi_2
\end{aligned}
\end{equation}
where we have used Wick's theorem to observe
\begin{align*}
	& \expec{\left[dW(\tau_1) dW(\tau_2)  - \expec{dW(\tau_1) dW(\tau_2)}\right] \times \left[dW(\tau'_1) dW(\tau'_2)  - \expec{dW(\tau'_1) dW(\tau'_2)}\right]} \\
	&= \expec{ dW(\tau_1) dW(\tau'_1)} \expec{ dW(\tau_2) dW(\tau'_2)} + \expec{ dW(\tau_1) dW(\tau'_2)} \expec{ dW(\tau'_1) dW(\tau_2)} \\
	& =\delta(\tau_1 + \tau'_1) \delta(\tau_2 + \tau'_2) d\tau_1 d\tau'_1 d\tau_2  d\tau'_2  + \delta(\tau'_1 + \tau_2) \delta(\tau_1 + \tau'_2) d\tau'_1 d\tau_2 d\tau_1 d\tau'_2  
\end{align*}
and $\Psi_1$ and $\Psi_2$ respectively denote the parts of $\expec{|\Delta_q( \Xi_2^\varepsilon - \Xi_2^\delta) (x)|^2}$ with the terms $\delta(\tau_1 + \tau'_1) \delta(\tau_2 + \tau'_2)$ and $ \delta(\tau'_1 + \tau_2) \delta(\tau_1 + \tau'_2).$ The treatment of these two terms are similar so in the sequel we give the details only for the estimate for the  term $\Psi_1$.

In order to contract the variables according to $\delta(\tau_1 + \tau'_1) \delta(\tau_2 + \tau'_2) $ we take $-\tau_1 = \tau'_1, -\tau_2 = \tau'_2,$ which gives
\begin{equation}\label{equ:contract0}
	\begin{aligned} 
		& \Psi_1 = \\
		& \sum_{|A-B| \leq 1} \sum_{n,m = 1}^\infty  \sum_{|A'-B'| \leq 1} \sum_{n',m' = 1}^\infty \int_{\left[\mathbb{R}^2 \right]^{\otimes 6}} \rho(\power{2}{-q}|z_1 + z_2 + \tau_1 + \tau_2|)\rho(\power{2}{-q}|z'_1 + z'_2 - \tau_1- \tau_2| ) \\
		&\rho(\power{2}{-A}|z_1 + \tau_1|)
		\rho(\power{2}{-A'}|z'_1 - \tau_1|)
		\frac{\rho(\power{2}{-B}|z_2 + \tau_2|)}{1+|z_2+ \tau_2|^2}   \frac{\rho(\power{2}{-B'}|z'_2 - \tau_2|)}{1+ |z'_2 - \tau_2|^2}\\ 
		&\rho_{\geq L_n} (|\tau_1|) \rho_{\geq L_m} (|\tau_2|) \rho_{\geq L_{n'}} (|\tau_1|) \rho_{\geq L_{m'}} (|\tau_2|)\\
		&\left[\sigma (\varepsilon|\tau_1|) \sigma (\varepsilon|\tau_2|) - \sigma (\delta|\tau_1|) \sigma (\delta|\tau_2|)\right]^2\\
		&\hat{w}_n(|z_1|) \hat{w}_m(|z_2|)  \hat{w}_{n'}(|z'_1|) \hat{w}_{m'}(|z'_2|) \\
		&\power{e}{2 \pi i x \cdot (z_1 + z_2+ z'_1 + z'_2)} dz_1 dz_2 dz'_1 dz'_2 d\tau_1 d\tau_2. 
	\end{aligned}
\end{equation} 
In $\Psi_1$ first we single out an individual term from the summations and define
\begin{equation}\label{equ:contract1}
	\begin{aligned} 
		& \large\textbf{I} := \\\
		& \int_{\left[\mathbb{R}^2 \right]^{\otimes 6}} \rho(\power{2}{-q}|z_1 + z_2 + \tau_1 + \tau_2|)\rho(\power{2}{-q}|z'_1 + z'_2 - \tau_1- \tau_2| ) \\
		&\rho(\power{2}{-A}|z_1 + \tau_1|)
		\rho(\power{2}{-A'}|z'_1 - \tau_1|)
		\frac{\rho(\power{2}{-B}|z_2 + \tau_2|)}{1+|z_2+ \tau_2|^2}   \frac{\rho(\power{2}{-B'}|z'_2 - \tau_2|)}{1+ |z'_2 - \tau_2|^2}\\
		& \rho_{k} (|\tau_1|) \rho_{l} (|\tau_2|) \rho_{k'} (|\tau_1|) \rho_{l'} (|\tau_2|)\\
		&\left[\sigma (\varepsilon|\tau_1|) \sigma (\varepsilon|\tau_2|) - \sigma (\delta|\tau_1|) \sigma (\delta|\tau_2|)\right]^2\\
		&\hat{w}_n(|z_1|) \hat{w}_m(|z_2|)  \hat{w}_{n'}(|z'_1|) \hat{w}_{m'}(|z'_2|) \\
		&\power{e}{2 \pi i x \cdot (z_1 + z_2+ z'_1 + z'_2)} dz_1 dz_2 dz'_1 dz'_2 d\tau_1 d\tau_2 
	\end{aligned}
\end{equation}
where we have
\begin{align*}
	\large\textbf{I} &= \large\textbf{I} (x, \varepsilon, \delta, A, B, A', B', k, l, k', l' , m, n, m', n') 
\end{align*}
and we can write
\begin{align}\label{equ:psi}
	\Psi_1 = \sum_{|A-B| \leq 1} \sum_{n,m = 1}^\infty  \sum_{|A'-B'| \leq 1} \sum_{n',m' = 1}^\infty \large\textbf{I}.
\end{align}
In the integral $\large\textbf{I},$ we first integrate with respect to the variable $z'_1$ and single out a Fourier transform term in the space-frequency variables $(-x, z'_1)$ as
\begin{equation}\label{equ:fourierExtraction}
	\begin{aligned}
		&\int \rho(\power{2}{-q}|z'_1 + z'_2 - \tau_1- \tau_2| ) \rho(\power{2}{-A'}|z'_1 - \tau_1|) \hat{w}_{n'}(|z'_1|) \power{e}{2 \pi i x \cdot  z'_1} dz'_1  \\
		&=  {w}_{n'} \ast \widehat{\rho_q \times \rho_{A'}}(-x, \tau_1, \tau_2, z'_2)\\
		&={w}_{n'} \ast \widehat{\rho}_q \ast \widehat{\rho}_{A'}(-x, \tau_1, \tau_2, z'_2)
	\end{aligned}
\end{equation}
By using $|A-B| \leq 1,$ observe that $|z_1 + \tau_1| \sim |\tau_2 + z_2|.$  Through  this observation and \eqref{equ:fourierExtraction}, we obtain the estimate
\begin{equation}\label{equ:integralI'}
	\begin{aligned} 
		&\large\textbf{I} \leq \int_{\left[\mathbb{R}^2 \right]^{\otimes 5}}  |{w}_{n'}| \ast |\widehat{\rho}_q| \ast |\widehat{\rho}_{A'}|(-x) \ \ \rho(\power{2}{-q}|z_1 + z_2 + \tau_1 + \tau_2|) \\
		&\rho(\power{2}{-A}|z_1 + \tau_1|)
		\frac{\rho(\power{2}{-B}|z_2 + \tau_2|)}{1+|z_1+ \tau_1|^2}   \frac{\rho(\power{2}{-B'}|z'_2 - \tau_2|)}{1+ |z'_2 - \tau_2|^2} \rho_{k} (|\tau_1|) \rho_{l} (|\tau_2|) \rho_{k'} (|\tau_1|) \rho_{l'} (|\tau_2|)\\
		& \left[\sigma (\varepsilon|\tau_1|) \sigma (\varepsilon|\tau_2|) - \sigma (\delta|\tau_1|) \sigma (\delta|\tau_2|)\right]^2\\
		&|\hat{w}_n|(|z_1|) \  |\hat{w}_m|(|z_2|) \    |\hat{w}_{m'}|(|z'_2|) \\
		& dz_1 dz_2  dz'_2 d\tau_1 d\tau_2 =: \int_{\mathbb{R}^{\otimes 5}}  \textbf{A} \ dZ dT   =:\large\textbf{I}'
	\end{aligned}
\end{equation}
where on the rightmost side we defined the integrand as $\textbf{A}$, the integral on $\left[\mathbb{R}^2 \right]^{\otimes 5}$ as $\large\textbf{I}'$ and integration variables as $dT dZ$ which satisfy
\begin{align*}
	\large\textbf{I}' &= \large\textbf{I}' (x, \varepsilon, \delta, A, B, A', B', k, l, k', l' , m, n, m', n') \\
	\large\textbf{A} &= \large\textbf{A} (x, \varepsilon, \delta, A, B, A', B', k, l, k', l' , m, n, m', n', z_1, z_2, z'_1, z'_2, \tau_1, \tau_2)\\
	dZ &= dz_1 dz_2  dz'_2, \ \ dT= d\tau_1 d\tau_2.
\end{align*}

For an arbitrary constant $0<a_0<1,$ we introduce the following cylindrical sets in $\left[\mathbb{R}^2 \right]^{\otimes 5}:$
\begin{align*}
	A_0(\tau_1) &= \left\{(z_1, \tau_1) \in \mathbb{R}^2 \times \mathbb{R}^2 ~\left|~  \frac{|z_1|}{|\tau_1|} < a_0 \right. \right\}\\
	A_0(\tau_2) &= \left\{(z_2, \tau_2) \in \mathbb{R}^2 \times \mathbb{R}^2 ~\left|~ \frac{|z_2|}{| \tau_2|} < a_0\right.\right\}\\
	A_1(\tau_2) &= \left\{(z'_2, \tau_2) \in \mathbb{R}^2 \times \mathbb{R}^2 ~\left|~ \frac{|z'_2|}{| \tau_2|} < a_0\right.\right\}.
\end{align*}
Accordingly we define the following partition of $\left[\mathbb{R}^2 \right]^{\otimes 5}:$
\begin{align*}
	\mathcal{A}_0 &=  A_0(\tau_1)  \cap A_0(\tau_2) \cap A_1(\tau_2) \\
	\mathcal{A}_1 &=  A_0(\tau_1)^c  \cap A_0(\tau_2) \cap A_1(\tau_2) \\
	\mathcal{A}_2 &= A_0(\tau_1) \cap ((A_0(\tau_2)^c \cap A_1(\tau_2) ) \cup (A_0(\tau_2) \cap A_1^c(\tau_2) ) \cup (A_0(\tau_2)^c \cap A_1(\tau_2)^c )  )\\
	\mathcal{A}_3 &= A_0(\tau_1)^c \cap ((A_0(\tau_2)^c \cap A_1(\tau_2) ) \cup (A_0(\tau_2) \cap A_1^c(\tau_2) ) \cup (A_0(\tau_2)^c \cap A_1(\tau_2)^c )  ).
\end{align*}
After this point, our purpose is to estimate the integrand $\textbf{A}$ in \eqref{equ:integralI'}  on each element of the partition.  Accordingly we define
\begin{align*}
	\textbf{A}_i = \mathbbm{1}_{\mathcal{A}_i} \textbf{A}.
\end{align*}
We decompose the integration in \eqref{equ:integralI'} with respect to this partition and write
\begin{equation}\label{equ:integralIprime}
	\large\textbf{I}'= \int_{\left[\mathbb{R}^2 \right]^{\otimes 5}} \left\{ \large\textbf{A}_0 + \large\textbf{A}_1 + \large\textbf{A}_2 +\large\textbf{A}_3 \right\} \   dZ dT.
\end{equation}
In order to obtain an estimate in the form \eqref{equ:exptwo}, we want to estimate each $\large\textbf{A}_i$ separately and then finally use \eqref{equ:psi} and \eqref{equ:integralI'}.  In the following, the constant $\kappa >0$ is always thought to be as small as possible.

$\bullet$ \textbf{Estimate on $\mathcal{A}_0$:}  

By Lemma \ref{lemm:renomemmaAbsEqui} we have $|\tau_1+z_1| \sim |\tau_1|$ and $|-z'_2 + \tau_2| \sim |\tau_2|$.  From $|A-B| \leq 1,$ we also have the relation $|z_1 +\tau_1| \sim |z_2+ \tau_2|$ we can write all of this as
\begin{align}\label{equ:worstcase}
	|\tau_1| \underset{I}{\sim}  |z_1 +\tau_1|  \sim |z_2+ \tau_2| \underset{II}{\sim}  |\tau_2| \underset{III}{\sim} |-z'_2 + \tau_2|.
\end{align}
By using $|z_1 +\tau_1| \sim |z_2+ \tau_2|$ and \eqref{equ:worstcase} we have that
\[
|z_1 +\tau_1+ z_2+ \tau_2| \leq |z_1 +\tau_1| + |z_2+ \tau_2| \lesssim |\tau_1| 
\]
in which case we can write 
\[
\frac{1}{1+ |\tau_1|^h } \lesssim \frac{1}{1+ |z_1 +\tau_1+ z_2+ \tau_2|^h }
\]
for any number $h>0.$
Finally, combining these with \eqref{equ:worstcase} leads to the upper bound
\begin{align*} 
	&	\large\textbf{A}_0 \leq\\
	&  |{w}_{n'}| \ast |\widehat{\rho}_q| \ast |\widehat{\rho}_{A'}|(-x) \ \ \rho(\power{2}{-q}|z_1 + z_2 + \tau_1 + \tau_2|) \\
	&\nonumber\rho(\power{2}{-A}| \tau_1+z_1|)
	\frac{\rho(\power{2}{-B}| \tau_2+z_2|)}{1+| z_1 +\tau_1+ z_2+ \tau_2|^{2-\kappa}}   \frac{\rho(\power{2}{-B'}| z'_2-\tau_2|)}{1+ | \tau_2|^{2+\kappa/2}} \frac{\rho_{\geq L_m} (|\tau_2|)}{1+ | \tau_2|^{\kappa/8}} \frac{\rho_{\geq L_n} (|\tau_1|)}{1+ | \tau_1|^{\kappa/8}} \frac{\rho_{\geq L_{n'}} (|\tau_1|)}{1+ | \tau_1|^{\kappa/8}} \frac{\rho_{\geq L_{m'}} (|\tau_2|)}{1+ | \tau_2|^{\kappa/8}}\\
	&\nonumber \left[\sigma (\varepsilon|\tau_1|) \sigma (\varepsilon|\tau_2|) - \sigma (\delta|\tau_1|) \sigma (\delta|\tau_2|)\right]^2\\
	&\nonumber|\hat{w}_n|(|z_1|) \  |\hat{w}_m|(|z_2|) \    |\hat{w}_{m'}|(|z'_2|) 
\end{align*}
where $\kappa>0$ can be taken arbitrarily small.

Observe that we have
\begin{equation}\label{equ:Lnbound}
	\frac{\rho_{\geq L_n} (|\tau|)}{1+ | \tau|^{\kappa/8}}\lesssim \frac{1}{1+ | 2^{L_n}|^{\kappa/8}} 
\end{equation}
and similar bound holds  also for other terms with $\rho_{\geq L_{n'}}, \rho_{\geq L_m}$ and  $\rho_{\geq L_{m'}}.$

By using these bounds, Lemma \ref{lemm:sigmaLemma} and \eqref{equ:worstcase} we obtain
\begin{align*} 
	&	\large\textbf{A}_0 \leq\\
	&  |{w}_{n'}| \ast |\widehat{\rho}_q| \ast |\widehat{\rho}_{A'}|(-x) \ \ \rho(\power{2}{-q}|z_1 + z_2 + \tau_1 + \tau_2|) \\
	&\nonumber\rho(\power{2}{-A}| \tau_1+z_1|)
	\frac{\rho(\power{2}{-B}| \tau_2+z_2|)}{1+| z_1 +\tau_1+ z_2+ \tau_2|^{2-\kappa}}   \frac{\rho(\power{2}{-B'}| z'_2-\tau_2|)}{1+ | \tau_2|^{2+\kappa/2}} 
	\frac{1}{1+ | 2^{L_m}|^{\kappa/8}}\frac{1}{1+ | 2^{L_n}|^{\kappa/8}}\\
	&\nonumber\frac{1}{1+ | 2^{L_{n'}}|^{\kappa/8}}\frac{1}{1+ | 2^{L_{m'}}|^{\kappa/8}}  \ 2 \  |\varepsilon - \delta|^{m_0}  |\tau_2|^{m_0} \\
	&\nonumber|\hat{w}_n|(|z_1|) \  |\hat{w}_m|(|z_2|) \    |\hat{w}_{m'}|(|z'_2|) 
\end{align*}
for some sufficiently small constants $\kappa, m_0 >0.$

Now thinking of right hand side as integrand to be integrated in the end of the proof, we apply the change of variables $ z_1 +\tau_1+ z_2+ \tau_2 \rightarrow X$ so that $\tau_1 = X - z_1 - z_2- \tau_2.$ We obtain
\begin{equation}\label{estimateAzero}
	\begin{aligned} 
		& \large\textbf{A}_0 \leq\\
		& |{w}_{n'}| \ast |\widehat{\rho}_q| \ast |\widehat{\rho}_{A'}|(-x) \ \ \rho(\power{2}{-q}|X|) \\
		&\rho(\power{2}{-A}|  X  - z_2- \tau_2|)
		\frac{\rho(\power{2}{-B}| \tau_2 + z_2 |)}{1+| X|^{2-\kappa}}   \frac{\rho(\power{2}{-B'}|z'_2 - \tau_2|)}{1+ | \tau_2|^{2+\kappa/2}} 
		\frac{1}{1+ | 2^{L_m}|^{\kappa/8}}\frac{1}{1+ | 2^{L_n}|^{\kappa/8}}\\
		&\frac{1}{1+ | 2^{L_{n'}}|^{\kappa/8}}\frac{1}{1+ | 2^{L_{m'}}|^{\kappa/8}}  \ 2 \  |\varepsilon - \delta|^{m_0}  |\tau_2|^{m_0} \\
		&|\hat{w}_n|(|z_1|) \  |\hat{w}_m|(|z_2|) \    |\hat{w}_{m'}|(|z'_2|) 
	\end{aligned}
\end{equation}
where we think of the constants $\kappa, m_0 >0$ to be very small and $\kappa/2 >> m_0$ for integrability in the $\tau_2$-variable.

$\bullet$ \textbf{Estimate on $\mathcal{A}_2$:}

Here we focus on the estimate over
\begin{align*}
	&\mathcal{A}_2 = A_0(\tau_1) \cap ((A_0(\tau_2)^c \cap A_1(\tau_2) ) \cup (A_0(\tau_2) \cap A_1^c(\tau_2) )\\ 
	&\cup (A_0(\tau_2)^c \cap A_1(\tau_2)^c )  )\\
	&=  \left\{ A_0(\tau_1) \cap  (A_0(\tau_2)^c \cap A_1(\tau_2) ) \right\} \cup \left\{ A_0(\tau_1) \cap (A_0(\tau_2) \cap A_1^c(\tau_2) ) \right\} \\
	&\cup \left\{ A_0(\tau_1)  \cap (A_0(\tau_2)^c \cap A_1(\tau_2)^c )  \right\}.
\end{align*}
Namely, we estimate $\large\textbf{A}_2.$ Since the different cases in the union have similar proof, we now work out the details for the term
\begin{align*}
	A_0(\tau_1) \cap (A_0(\tau_2) \cap A_1^c(\tau_2) ).
\end{align*}
Observe that in this case, by the definition of $A_0(\tau_1)$ and $|A-B| \leq 1$    we obtain
\begin{align}\label{equ:worstcase2}
	|\tau_1| \sim  |z_1 +\tau_1|  \sim |z_2+ \tau_2| \sim |\tau_2|.
\end{align}
By Lemma \ref{lemm:tauZexchange}, up to a constant,  we can write 
\begin{align*} 
	& \large\textbf{A}_2\leq\\
	&  \left| {w}_{n'} \ast \widehat{\rho_q \times \rho_{A'}}(-x)\right| \ \ \rho(\power{2}{-q}|z_1 + z_2 + \tau_1 + \tau_2|) \\
	&\nonumber\rho(\power{2}{-A}|z_1 + \tau_1|)
	\frac{\rho(\power{2}{-B}|z_2 + \tau_2|)}{1+|z_1+ \tau_1|^{2-2\kappa}}   \rho(\power{2}{-B'}|z'_2 - \tau_2|) \frac{1}{1+a_0|\tau_2|^{N-2\kappa}} \\
	&\frac{\rho_{\geq L_m} (|\tau_2|)}{1+ | \tau_2|^{\kappa}} \frac{\rho_{\geq L_n} (|\tau_1|)}{1+ | \tau_1|^{\kappa}} \frac{\rho_{\geq L_{n'}} (|\tau_1|)}{1+ | \tau_1|^{\kappa}} \frac{\rho_{\geq L_{m'}} (|\tau_2|)}{1+ | \tau_2|^{\kappa}}\\
	&\nonumber \left[\sigma (\varepsilon|\tau_1|) \sigma (\varepsilon|\tau_2|) - \sigma (\delta|\tau_1|) \sigma (\delta|\tau_2|)\right]^2\\
	&\nonumber|\hat{w}_n|(|z_1|) \  |\hat{w}_m|(|z_2|) \    |\hat{w}_{m'}|(|z'_2|)^{1/2}
\end{align*}
By using \eqref{equ:Lnbound} and Lemma \ref{lemm:sigmaLemma} we obtain
\begin{align*} 
	& \large\textbf{A}_2\leq\\
	&  \left| {w}_{n'} \ast \widehat{\rho_q \times \rho_{A'}}(-x)\right| \ \ \rho(\power{2}{-q}|z_1 + z_2 + \tau_1 + \tau_2|) \\
	&\nonumber\rho(\power{2}{-A}|z_1 + \tau_1|)
	\frac{\rho(\power{2}{-B}|z_2 + \tau_2|)}{1+| \tau_1|^{2-2\kappa}}   \rho(\power{2}{-B'}|z'_2 - \tau_2|) \frac{1}{1+a_0|\tau_2|^{N-2\kappa}} \\
	&\frac{1}{1+ | 2^{L_m}|^{\kappa}}\frac{1}{1+ | 2^{L_n}|^{\kappa}}\frac{1}{1+ | 2^{L_{n'}}|^{\kappa}}\frac{1}{1+ | 2^{L_{m'}}|^{\kappa}}  \ 2 \  |\varepsilon - \delta|^{m_0}  |\tau_2|^{m_0}\\
	&\nonumber|\hat{w}_n|(|z_1|) \  |\hat{w}_m|(|z_2|) \    |\hat{w}_{m'}|(|z'_2|)^{1/2}
\end{align*}
Now, we apply the change of variables $ z_1 +\tau_1+ z_2+ \tau_2 \rightarrow X$ so that $\tau_1 = X - z_1 - z_2- \tau_2$  in order to obtain
\begin{equation}\label{estimateAtwo}
	\begin{aligned} 
		& \large\textbf{A}_2 \leq\\
		& |{w}_{n'}| \ast |\widehat{\rho}_q| \ast |\widehat{\rho}_{A'}|(-x) \ \ \rho(\power{2}{-q}|X|) \\
		&\rho(\power{2}{-A}|  X -  z_2- \tau_2|)
		\frac{\rho(\power{2}{-B}|z_2+ \tau_2|)}{1+| X|^{2-\kappa}}   \frac{\rho(\power{2}{-B'}| z'_2-\tau_2|)}{1+ | \tau_2|^{N-2 \kappa}} 
		\frac{1}{1+ | 2^{L_m}|^{\kappa}}\frac{1}{1+ | 2^{L_n}|^{\kappa}}\\
		&\frac{1}{1+ | 2^{L_{n'}}|^{\kappa}}\frac{1}{1+ | 2^{L_{m'}}|^{\kappa}}  \ 2 \  |\varepsilon - \delta|^{m_0}  |\tau_2|^{m_0} \\
		&|\hat{w}_n|(|z_1|) \  |\hat{w}_m|(|z_2|) \    |\hat{w}_{m'}|(|z'_2|) 
	\end{aligned}
\end{equation}

$\bullet$ \textbf{Estimate on $\mathcal{A}_1$:}

Now, we work out the case $	\mathcal{A}_1 =  A_0(\tau_1)^c  \cap A_0(\tau_2) \cap A_1(\tau_2).$  In this case, we have the equivalences:
\begin{align}\label{equ:worstcaseT2}
	|z_1 +\tau_1|  \sim |z_2+ \tau_2| \underset{II}{\sim}  |\tau_2| \underset{III}{\sim} |-z'_2 + \tau_2|.
\end{align}
where $II$ and $III$ hold on $A_0(\tau_2) \cap A_1(\tau_2).$  By Lemma \ref{lemm:tauZexchange}, up to a constant,   we can write 
\begin{align*} 
	& \large\textbf{A}_1\leq\\
	& \left| {w}_{n'} \ast \widehat{\rho_q \times \rho_{A'}}(-x)\right| \ \ \rho(\power{2}{-q}|z_1 + z_2 + \tau_1 + \tau_2|) \\
	&\nonumber\rho(\power{2}{-A}|z_1 + \tau_1|)
	\frac{\rho(\power{2}{-B}|z_2 + \tau_2|)}{1+| \tau_2|^{2-2\kappa}}   \rho(\power{2}{-B'}|z'_2 - \tau_2|) \frac{1}{1+a_0|\tau_1|^{N-2\kappa}} \\
	&\frac{\rho_{\geq L_m} (|\tau_2|)}{1+ | \tau_2|^{\kappa}} \frac{\rho_{\geq L_n} (|\tau_1|)}{1+ | \tau_1|^{\kappa}} \frac{\rho_{\geq L_{n'}} (|\tau_1|)}{1+ | \tau_1|^{\kappa}} \frac{\rho_{\geq L_{m'}} (|\tau_2|)}{1+ | \tau_2|^{\kappa}}\\
	&\nonumber \left[\sigma (\varepsilon|\tau_1|) \sigma (\varepsilon|\tau_2|) - \sigma (\delta|\tau_1|) \sigma (\delta|\tau_2|)\right]^2\\
	&\nonumber|\hat{w}_n|(|z_1|)^{1/2} \  |\hat{w}_m|(|z_2|) \    |\hat{w}_{m'}|(|z'_2|)^{}
\end{align*}
where we have chosen $N \gg 2.$  After this point we proceed very similar to the estimate of $\large\textbf{A}_2$ and apply the change of variables $ z_1 +\tau_1+ z_2+ \tau_2 \rightarrow X$ with $\tau_2 = X - z_1 - z_2- \tau_1.$  This gives the estimate
\begin{equation}\label{estimateAone}
	\begin{aligned} 
		& \large\textbf{A}_1 \leq\\
		& |{w}_{n'}| \ast |\widehat{\rho}_q| \ast |\widehat{\rho}_{A'}|(-x) \ \ \rho(\power{2}{-q}|X|) \\
		&\rho(\power{2}{-A}| z_1 +\tau_1|)
		\frac{\rho(\power{2}{-B}| X - z_1 - \tau_1|)}{1+| X|^{2-\kappa}}   \frac{\rho(\power{2}{-B'}| z'_2-X +z_1 + z_2+ \tau_1|)}{1+ | \tau_1|^{N-2 \kappa}} 
		\frac{1}{1+ | 2^{L_m}|^{\kappa}}\frac{1}{1+ | 2^{L_n}|^{\kappa}}\\
		&\frac{1}{1+ | 2^{L_{n'}}|^{\kappa}}\frac{1}{1+ | 2^{L_{m'}}|^{\kappa}}  \ 2 \  |\varepsilon - \delta|^{m_0}  |\tau_2|^{m_0} \\
		&|\hat{w}_n|(|z_1|) \  |\hat{w}_m|(|z_2|) \    |\hat{w}_{m'}|(|z'_2|).
	\end{aligned}
\end{equation}

$\bullet$ \textbf{Estimate on $\mathcal{A}_3$:}

In this case, we finally get an estimate over 	$\mathcal{A}_3 = A_0(\tau_1)^c \cap ((A_0(\tau_2)^c \cap A_1(\tau_2) ) \cup (A_0(\tau_2) \cap A_1^c(\tau_2) ) \cup (A_0(\tau_2)^c \cap A_1(\tau_2)^c )  ).$    Since, all of them follow similar lines we will give the details for the term
\[
A_0(\tau_1)^c \cap	A_0(\tau_2) \cap A_1(\tau_2)^c.
\]
This time  we obtain By Lemma \ref{lemm:tauZexchange}, up to a constant,  we obtain 
\begin{equation}\label{estimateAthree}
	\begin{aligned} 
		& \large\textbf{A}_3\lesssim\\
		&  \left| {w}_{n'} \ast \widehat{\rho_q \times \rho_{A'}}(-x)\right| \ \ \rho(\power{2}{-q}|z_1 + z_2 + \tau_1 + \tau_2|) \\
		&\rho(\power{2}{-A}|z_1 + \tau_1|)
		\frac{\rho(\power{2}{-B}|z_2 + \tau_2|)}{1+| \tau_2|^{N-2\kappa}}   \rho(\power{2}{-B'}|z'_2 - \tau_2|) \frac{1}{1+a_0|\tau_1|^{N-2\kappa}} \\
		&\frac{\rho_{\geq L_m} (|\tau_2|)}{1+ | \tau_2|^{\kappa}} \frac{\rho_{\geq L_n} (|\tau_1|)}{1+ | \tau_1|^{\kappa}} \frac{\rho_{\geq L_{n'}} (|\tau_1|)}{1+ | \tau_1|^{\kappa}} \frac{\rho_{\geq L_{m'}} (|\tau_2|)}{1+ | \tau_2|^{\kappa}}\\
		& \left[\sigma (\varepsilon|\tau_1|) \sigma (\varepsilon|\tau_2|) - \sigma (\delta|\tau_1|) \sigma (\delta|\tau_2|)\right]^2\\
		&|\hat{w}_n|(|z_1|)^{1/2} \  |\hat{w}_m|(|z_2|)^{}  \    |\hat{w}_{m'}|(|z'_2|)^{1/2}.
	\end{aligned}
\end{equation}

$\bullet$ \textbf{Final estimate:}

From \eqref{equ:psi1psi2} we recall 
\begin{align*}
	\expec{|\Delta_q( \Xi_2^\varepsilon - \Xi_2^\delta) (x)|^2} = \Psi_1 + \Psi_2
\end{align*}
which implies, by hypercontractivity, that
\begin{equation}\label{expectationpnormestimate}
	\expec{|\Delta_q( \Xi_2^\varepsilon - \Xi_2^\delta) (x)|^p}\leq \left(\Psi_1 + \Psi_2\right)^{p/2}\lesssim_p \Psi_1^{p/2} + \Psi_2^{p/2}
\end{equation}

By using \eqref{equ:integralI'} and \eqref{equ:integralIprime}, we obtain the estimate
\begin{equation}\label{equ:AE}
	\begin{aligned}
		&\Psi_1^{p/2} \lesssim_p \left[\sum_{|A-B| \leq 1} \sum_{n,m = 1}^\infty  \sum_{|A'-B'| \leq 1} \sum_{n',m' = 1}^\infty \int_{\left[\mathbb{R}^2 \right]^{\otimes 6}}  \left\{ \large\textbf{A}_0 + \large\textbf{A}_1 + \large\textbf{A}_2 + \large\textbf{A}_3\right\}dT dZ\right]^{p/2} \\
		&\lesssim_p \sum_{w=0}^3 \left[\sum_{|A-B| \leq 1} \sum_{n,m = 1}^\infty  \sum_{|A'-B'| \leq 1} \sum_{n',m' = 1}^\infty \int_{\left[\mathbb{R}^2 \right]^{\otimes 6}}   \large\textbf{A}_w \ dT dZ\right]^{p/2}.
	\end{aligned}
\end{equation}
In the next, in order to get a bound on this $\Psi_1$-term we want to push the power $\frac{p}{2}$ to inside by Jensen's inequality and use the estimates \eqref{estimateAzero}, \eqref{estimateAone},  \eqref{estimateAtwo} and \eqref{estimateAthree}. We write the detailed estimate only for the term $w=0,$ as the proof for other terms follow similar lines and in fact produces lower order terms due to faster decay in $\tau$-variables.  
We integrate in the $x$-variable and use \eqref{estimateAzero} and Jensen's inequality to obtain
\begin{align*}
	&\int_{\mathbb{R}^2}\left[\sum_{|A-B| \leq 1} \sum_{n,m = 1}^\infty  \sum_{|A'-B'| \leq 1} \sum_{n',m' = 1}^\infty \int_{\left[\mathbb{R}^2 \right]^{\otimes 5}}   \large\textbf{A}_0 \ dT dZ\right]^{p/2} dx\\
	& \lesssim \ 2^{p/2} \  |\varepsilon - \delta|^{m_0\frac{p}{2}}   \\
	&\sum_{|A-B| \leq 1} \sum_{n,m = 1}^\infty  \sum_{|A'-B'| \leq 1} \sum_{n',m' = 1}^\infty \int_{\left[\mathbb{R}^2 \right]^{\otimes 5}}\int_{\mathbb{R}^2} \frac{\left\{ |{w}_{n'}| \ast |\widehat{\rho}_q| \ast |\widehat{\rho}_{A'}|(-x)\right\}^{p/2}}{2^{n' \cdot \nu\cdot p/2}} dx\ \  \\
	& \rho(\power{2}{-A}|  X  - z_2- \tau_2|)N(dA) \  \rho(\power{2}{-B}| \tau_2+z_2|)N(dB) \  \rho(\power{2}{-B'}| z'_2-\tau_2|) N(dB') \\
	&  \frac{|\hat{w}_m|(|z_2|) N(dm) \ dz_2}{1+ | 2^{L_m}|^{\kappa/8}} \ \frac{|\hat{w}_n|(|z_1|) N(dn) \  dz_1}{1+ | 2^{L_n}|^{\kappa/8}} \ \frac{N(dn')}{1+ |2^{L_{n'}\cdot\kappa/8 -n'\cdot \nu}|} \  \frac{|\hat{w}_{m'}|(|z'_2|) N(dm') \  dz'_2}{1+ | 2^{L_{m'}}|^{\kappa/8}}\\
	&
	\frac{\rho(\power{2}{-q}|X|) dX}{1+| X|^{2-\kappa}}   \frac{d\tau_2}{1+ | \tau_2|^{2+\kappa/2-m_0}} 
\end{align*}
where $N(d\cdot)$ denotes the counting measure on $\mathbb{N}$ and $v>0$ is chosen so that $$L_{n'}\cdot\kappa/8 -n'\cdot \nu>0.$$ The finiteness of the measures are in particular guaranteed by $\sum_{M=-1}^\infty \rho(\power{2}{-M}| \cdot|) =1,\ $ $L_n = \varepsilon_{2} n>0$ and the estimate
	\[
\int |\hat{w}_{m'}|(|z'_2|)  dz'_2 =  \int 2^{2m'} |\hat{w}_{0}|(|2^{m'} z'_2|)  dz'_2 = \int 2^{2m'} |\hat{w}_{0}|(|z'_2|)  dz'_2 \leq \norm{\hat{w}_{0}}{L^1}
\]
which follows from scaling properties of the Fourier transform.

By Young's inequality for convolutions we observe the bound
\begin{align*}
	& \norm{|{w}_{n'}| \ast |\widehat{\rho}_q| \ast |\widehat{\rho}_{A'}|(\cdot)}{L^{p/2}}^{p/2} \leq  \norm{w_{n'}}{L^{p/2}}^{p/2}  \norm{\hat{\rho}_{A'}}{L^1}^{1/2} \norm{\hat{\rho}_{q}}{L^1}^{p/2}
\end{align*}
and parse the estimate as 
\begin{align*}
	&\int_{\mathbb{R}^2}\left[\sum_{|A-B| \leq 1} \sum_{n,m = 1}^\infty  \sum_{|A'-B'| \leq 1} \sum_{n',m' = 1}^\infty \int_{\left[\mathbb{R}^2 \right]^{\otimes 6}}   \large\textbf{A}_0 \ dT dZ\right]^{p/2} dx\\
	& \lesssim 2^{p/2} \  |\varepsilon - \delta|^{\frac{p}{2}m_0} \int_{\left[\mathbb{R}^2 \right]^{\otimes 6}}  \left\{\sum_{B'=-1}^{\infty}\rho(\power{2}{-B'}| z'_2-\tau_2|)\right\} \left\{\sum_{n'=1}^\infty\frac{\norm{w_{n'}}{L^{p/2}}^{p/2}}{2^{n' \cdot \nu\cdot p/2}}\right\} \left\{\sum_{B=-1}^\infty\rho(\power{2}{-B}| \tau_2+z_2|)\right\}  \\
	&\left\{\sum_{A=-1}^\infty\rho(\power{2}{-A}|  X  - z_2- \tau_2|)\right\}    
	\left\{\sum_{n'=1}^\infty\frac{1}{(1+ 2^{L_{n'}\cdot\kappa/8 -n'\cdot \nu})} \right\}    \left\{\sum_{m'=1}^\infty\frac{1}{1+ | 2^{L_{m'}}|^{\kappa/8}}\right\}
	\left\{\sum_{m=1}^\infty\frac{1}{1+ | 2^{L_m}|^{\kappa/8}} \right\} \\
	& |\hat{w}_n|(|z_1|) \  |\hat{w}_m|(|z_2|) \    |\hat{w}_{m'}|(|z'_2|) 	\frac{\rho(\power{2}{-q}|X|)}{1+| X|^{2-\kappa}} \frac{1}{1+ | \tau_2|^{2-m_0+\kappa/2}}  \ dX d\tau_2 dZ  \\
	&\lesssim 2^{\kappa\cdot q} \ 2^{p/2} \  |\varepsilon - \delta|^{\frac{p}{2}m_0} 
\end{align*}
At this point $p>>1$ is chosen in such a way that $\frac{p}{2}m_0>1$ and $\nu\cdot\frac{p}{2} > 2.$  The treatment of the other terms in \eqref{equ:AE} with $\large\textbf{A}_w, w=1, 2, 3$ can be obtained in a similar manner. By using  \eqref{equ:AE} the estimate for the $\Psi_1$-term follows and the $\Psi_2$-term can be estimated in an analogous way.  Consequently, by \eqref{expectationpnormestimate} we arrive at the estimate
\begin{align*}
	\expec{\norm{\Xi_2^\varepsilon - \Xi_2^\delta}{\sobolevb{-\kappa/p}{p}{\infty}}^p} \lesssim  2^{p/2} \  |\varepsilon - \delta|^{\frac{p}{2}m_0}
\end{align*}
which, by Besov embedding, implies 
\begin{align*}
	\expec{\norm{\Xi_2^\varepsilon - \Xi_2^\delta}{\sobolevb{-\kappa/p-\frac{1}{p}}{\infty}{\infty}}^p} \lesssim 2^{p/2} \  |\varepsilon - \delta|^{\frac{p}{2}m_0}.
\end{align*}
By Kolmogorov-Chentsov Theorem it follows that $\Xi_2^\varepsilon$ converges in $L^p(\Omega, \sobolevb{-\kappa/p-\frac{1}{p}}{\infty}{\infty})$ and almost surely converges to a random variable $\Xi_2$  in $\sobolevb{-\kappa/p-\frac{1}{p}}{\infty}{\infty}.$  Hence, the result follows.
\end{proof}

After the proof of the theorem we directly observe that the operator is defined in $L^2$.
\begin{proposition}\label{prop:operatordef}
	The operator $T$ on $\domain{T}$  is an operator in $L^2$.
\end{proposition}
\begin{proof}
	This is clear by the calculation we have done at the beginning of the section, where in  \eqref{equ:tuepsilon}, we can now simply take the limit as  $\varepsilon \rightarrow 0$.  By Theorem \ref{thm:2dren}, we  already know that $(X_\varepsilon\circ \xi_\varepsilon + c_\varepsilon(x)) \rightarrow \Xi_2$ .  Hence by regularity of the para- and resonant products (Proposition \ref{prop:bonyEst}) the result readily follows.
\end{proof}
\begin{remark}
	Observe that although we defined the domain with specific parameters $\alpha$  and $\gamma,$ in fact both the definition and Proposition \ref{prop:operatordef} are also valid for a range of these smoothness parameters, as $- \frac{4}{3} < \alpha < - 1$ and $- \frac{\alpha}{2} < \gamma \le \alpha + 2.$  In addition, though the domain norm \eqref{equ:domainInnerProd} is just an abstract definition for the time being, we will later on (see Proposition \ref{lem:formdom}) clarify the equivalence of the normed spaces $(\domain{T},  \norm{\cdot}{\domain{T}}) \sim (\domain{T},  \norm{\cdot}{2}) $ where $\norm{u}{2}:= \norm{Tu}{\eltwo}.$
\end{remark}

We use the Fourier cut off trick from \cite{GUZ20} to define a map $\Gamma : \sobolev{H}{2}(\R^2) \rightarrow \domain{T}$.  Namely in the ansatz \eqref{equ:ansatzmain}, by choosing a suitable constant $N = N(\Xi)$ (depending on the realization  of the enhanced noise) we can define the map $\Gamma$ such that $u  = \Gamma (u^\sharp).$  That is to say, we characterize the $\domain{T}$ as the image of $\sobolev{H}{2}(\R^2)$ under a well defined map.  

Later, we will employ  the map $\Gamma$ also in establishing the norm resolvent convergence of the regularizations \eqref{equ:tepsilonOp} as well as the embedding properties  and functional inequalities for the domain and the form-domain. Let us fix a realization of the enhanced noise $\Xi = (\Xi_1, \Xi_2)$  and $N = N(\Xi)>0,$ which will be fixed later. For  $s \in (0, \gamma],$ inspired by the ansatz \eqref{equ:ansatzmain}, for a fixed $f^\sharp \in \sobolew{H}{s}{\delta}$ we define the following  auxiliary map
\[
\Psi_{f^\sharp}(f) := \Delta_{> N} (f\prec X + B_{\Xi}(f) )+ f^\sharp.
\]
Next, we  show the fixed point property of the maps $\Psi_{f^\sharp}(\cdot)$ and  introduce the $\Gamma$-map, that satisfies certain weighted-Sobolev bounds.
\begin{proposition} \label{lem:gamma}
	Let  $\delta\geq 0$ and $s \in (0, \gamma].$ For every realization $\Xi \in \mathscr{X}^\alpha$ of the enhanced noise, there exists  $N = N(\Xi) >0$ such that the map $\Psi_{f^\sharp}:\sobolew{H}{s}{\delta} \rightarrow \sobolew{H}{s}{\delta}$ has a fixed point. For $f^\sharp \in \sobolew{H}{s}{\delta}$ and the corresponding fixed point $f= \Psi_{f^\sharp}(f),$  the map $\Gamma:  \sobolew{H}{s}{\delta} \rightarrow \sobolew{H}{s}{\delta}$ defined as $\Gamma(f^{\sharp}): = f$  satisfies the following estimates
	\begin{align*}
		||\Gamma f||_{\elinfty} &\leq C_{N, \Xi} || f||_{\elinfty}\\
		||\Gamma f||_{\sobolew{H}{s}{\delta}} &\leq C_{N, \Xi} || f||_{\sobolew{H}{s}{\delta}}.
	\end{align*}
\end{proposition}
\begin{proof}
	As the $\elinfty$-estimate can be proved in similar way to the bounded domain case  \cite{GUZ20}, we omit this proof and prove in detail the contraction property for sufficiently large $N = N(\Xi) >0$.
	We expand as
	\[
	\Psi_{f^\sharp}(f) - \Psi_{f^\sharp}(g) = \Delta_{> N} ((f-g)\prec X + B_{\Xi}(f-g)).
	\]
	For an arbitrarily small $\kappa>0, $ by Propositions \ref{prop:productBesovWeight}, \ref{prop:paraproductEstWeight} and \ref{prop:cutRegularity}  we estimate
	\begin{align*}
		& \norm{ \Delta_{> N} ((f-g)\prec X + B_{\Xi}(f-g)) )}{\sobolew{H}{s}{\delta}}\\ 
		&\leq \norm{ \Delta_{> N} (f-g)\prec X   }{\sobolew{H}{s}{\delta}}\\ 
		&+\norm{ \Delta_{> N} \left( (1 - \Delta)^{- 1} (\Delta (f-g) \prec X + 2 \nabla (f-g)
			\prec \nabla X+ \xi \prec (f-g) - (f-g) \prec \Xi_2)  \right)}{\sobolew{H}{s}{\delta}}\\
		& \leq 2^{2N\cdot(s+\kappa-\gamma)} \norm{  (f-g)\prec X   }{\sobolew{H}{\gamma-\kappa}{\delta}}\\
		&+ 2^{2N\cdot(s+\kappa-\gamma)} \left\|  (1 - \Delta)^{- 1} \left( \Delta (f-g) \prec X + 2 \nabla (f-g) \prec \nabla X
		\right. \right.\\
		&\left. \left.   + \xi \prec (f-g) - (f-g) \prec \Xi_2 \right) \right\|_{\sobolew{H}{\gamma-\kappa}{\delta}}\\
		&\leq C_\Xi 2^{2N\cdot(s+\kappa-\gamma)} \norm{f-g}{\sobolew{H}{s}{\delta}}.
	\end{align*}
	For $s+\kappa< \gamma$, taking $N$ large so that $C_\Xi 2^{2N\cdot(s+\kappa-\gamma)} <1$ shows  $\Psi_{f^\sharp}$ is a contraction map $\sobolew{H}{s}{\delta} \rightarrow \sobolew{H}{s}{\delta}$.  The endpoint case follows from inspection of regularities  and a density argument. In order to prove the weighted-Sobolev estimate, one uses the same reasoning as above.	We omit the details.
\end{proof}
\begin{remark}\label{rem:gammaApp}
	By \eqref{equ:ansatzmain} it follows that $\domain{T} = \Gamma(\sobolev{H}{2}).$
	We also define the regularized version  of the $\Gamma$-map as follows
	\begin{equation}\label{equ:smoothNotation}
		\Gamma_\varepsilon f := \Delta_{> N} (\Gamma_\varepsilon f \prec X_\varepsilon + B_{\Xi_\varepsilon} (\Gamma_\varepsilon f)) + f^\sharp. \end{equation}
	Observe that $\Gamma_\varepsilon$ also satisfies the estimates given in Proposition \ref{lem:gamma}, and one can show this with the same argument.
	
	In the sequel, when we want to emphasize the dependence of $f^\sharp$ on $N>0$ we will write $f_N^\sharp$.  That is,
	\begin{equation}\label{equ:ansatzinfsharp}
		f = \Delta_{> N} (f \prec X + B_{\Xi} (f)) + f_N^\sharp. 
	\end{equation}
	But usually we will not make a distinction because the value of $N$ will only change the estimates by a constant depending on $N$.   The following lemma clarifies this point.
\end{remark}
\begin{lemma}\label{lem:usharpequiv}
	Suppose that $\gamma \leq s \leq 2 $ and  for  fixed $N > M\geq 0$ consider 
	\begin{equation}
		\begin{aligned}\label{equ:cutNdifferent}
			f &= \Delta_{> N} ( f \prec X + B_{\Xi} ( f)) + f_N^\sharp\\
			f &= \Delta_{> M} (f \prec X + B_{\Xi} (f)) + f_M^\sharp.
		\end{aligned}
	\end{equation}
	It follows that
	\[
	\norm{f_M^\sharp}{\sobolev{H}{s}} \asymp \norm{f_N^\sharp}{\sobolev{H}{s}}
	\]
	and the same estimate holds when the representation \eqref{equ:cutNdifferent} is written with $X_\varepsilon$ and regularized enhanced noise $\Xi_\varepsilon.$
\end{lemma}
\begin{proof}
	By definitions we have that
	\[
	f_N^\sharp  = f_M^\sharp + (\Delta_{>M}- \Delta_{>N} )( f \prec X + B_{\Xi} ( f)).
	\]
	Therefore, we  obtain
	\begin{align}\label{equ:equiMNsharp}
		\norm{	f_M^\sharp- f_N^\sharp}{\sobolev{H}{s}} \leq C_\Xi N 2^{N\cdot(s-\gamma)} \norm{f}{\sobolev{H}{\gamma}} \lesssim C_\Xi N 2^{N\cdot(s-\gamma)} \norm{f^\sharp_M}{\sobolev{H}{s}}
	\end{align}
	where we used Lemma \ref{lem:gamma} and Proposition \ref{prop:bonyEst}.  The upper bound with $f^\sharp_N$ is obtained in the same way.  The estimate for the regularized enhanced noise $\Xi_\varepsilon$ follows similarly.  Hence, the result follows.
\end{proof}

As a preparation to self-adjointness, we gather the results that establish $T$ as a closed and symmetric operator over a dense domain and give useful approximation methods.  
\begin{proposition} \label{prop:opclosed}
	Let $\delta\geq 0$ and $\Gamma_\varepsilon$ be defined as in \eqref{equ:smoothNotation}. The following statements hold:
	\begin{itemize}
		\item It follows that 
		\begin{equation}\label{equ:gammaCovId}
			|| \text{id} - \Gamma  \Gamma_\varepsilon^{-1}||_{\sobolew{H}{\gamma}{\delta} \rightarrow \sobolew{H}{\gamma}{\delta}} \rightarrow 0.
		\end{equation}
		Therefore, the space $\domain{T}$  is dense in $\sobolew{H}{\gamma}{\delta},$ therefore dense  in $L^2$.  
		\item For every $u= \Gamma(u^\sharp) \in \domain{T} \subset \mathscr{H}^\gamma,$ there exists a sequence $u_\varepsilon^\sharp \in \sobolew{H}{2}{\delta}$ with $\ue:=\Gamma_\varepsilon(u_\varepsilon^\sharp)$ such that 
		\begin{equation} \label{equ:h2app}
			\| u - u_\varepsilon  \|_{\sobolew{H}{\gamma}{\delta}} + \| u^\sharp - u_\varepsilon^\sharp \|_{\sobolew{H}{2}{\delta}} \rightarrow 0
		\end{equation}
		as $\varepsilon \rightarrow 0$.  The approximating sequence can be chosen to be $u_\varepsilon = \Gamma_\varepsilon(u^\sharp).$
	\end{itemize}
\end{proposition}
\begin{proof}
	As the statements have similar proofs that uses the bilinearity of the para- and resonant products, we demonstrate only the first one in detail.  For $f \in \sobolew{H}{\gamma}{\delta},$ by using Propositions \ref{lem:gamma} and  \ref{prop:paraproductEstWeight} we have the estimate
	\begin{equation}
		\begin{aligned}\left\|f-\Gamma \Gamma_{\varepsilon}^{-1}(f)\right\|_{\sobolew{H}{\gamma}{\delta}} &=\left\|\Gamma\left(f-f \prec X-B_{\Xi}(f)\right)-\Gamma\left(f-f \prec X_{\varepsilon}-B_{\Xi^{\varepsilon}}(f)\right)\right\|_{\sobolew{H}{\gamma}{\delta}} \\ &=\left\|\Gamma\left(f \prec\left(X_{\varepsilon}-X\right)+B_{\left(\Xi^{\varepsilon}-\Xi\right)}(f)\right)\right\|_{\sobolew{H}{\gamma}{\delta}}\\ & \leq C_{\Xi}\|f\|_{\sobolew{H}{\gamma}{\delta}}\left\|\Xi^{\varepsilon}-\Xi\right\|_{\mathcal{E}^\alpha}  \end{aligned}.
	\end{equation}
	Hence, as $\varepsilon \rightarrow 0$ the result follows by Theorem \ref{thm:2dren}.
\end{proof}

\begin{proposition} \label{prop:opApp}
	Consider the following assumptions:
	\begin{itemize}
		\item $u^\sharp \in \mathcal{H}^2$ and $u  = \Gamma(u^\sharp ).$  \item $u_\varepsilon  := \Gamma_\varepsilon(u^\sharp ).$
		\item $u_n^\sharp$ be an arbitrary sequence in $\sobolev{H}{2}.$ 
	\end{itemize}
	With these assumptions, the following statements hold:
	\begin{itemize}
		\item It follows that
		\begin{equation}\label{equ:normOpConv2}
			\begin{aligned}
				&\| T\Gamma u^\sharp   - T\Gamma u_n^\sharp \|_{L^2} \lesssim_{\Xi}  \|u^\sharp - u_n^\sharp \|_{\mathcal{H}^2}\\
				&\| T_\varepsilon\Gamma_\varepsilon u^\sharp   - T_\varepsilon\Gamma_\varepsilon u_n^\sharp \|_{L^2} \lesssim_{\Xi}  \|u^\sharp - u_n^\sharp \|_{\mathcal{H}^2}\\
				&\| Tu - T_\varepsilon u_\varepsilon \|_{L^2} \lesssim_{\Xi}   \| \Xi_\varepsilon - \Xi\|_{\mathscr{X}^\alpha} \|u^\sharp\|_{\mathcal{H}^2}
			\end{aligned}
		\end{equation}
		which implies
		\begin{equation} \label{equ:normOpConv3}
			\| T\Gamma - T_\varepsilon \Gamma_\varepsilon   \|_{\mathcal{H}^2 \rightarrow L^2}  \rightarrow 0.
		\end{equation}
		That is, $T_\varepsilon \Gamma_\varepsilon$ converges to $T\Gamma$ in operator norm.  
		\item $T$ is a symmetric operator on its dense domain $\domain{T}.$ 
		\item $T$ satisfies the following estimate
		\begin{equation}\label{equ:h2bound}
			\| u^{\sharp} \|_{\ssp^2} \lesssim 2 \| T u \|_{L^2} + C_{\Xi} \| u
			\|_{L^2}. \end{equation}
		and a closed  operator on its dense domain $\domain{T}.$ 
		\item There exists a constant $C_{\Xi}>0$ such that
		\begin{equation} \label{equ:2dconstant}
			\frac{1}{2} \langle \nabla u^{\sharp}, \nabla u^{\sharp} \rangle
			\le - \langle u, T u \rangle + C_{\Xi} \| u \|_{L^2}^2
		\end{equation}
		where the estimate holds, with a uniform constant, also for the regularized operators $T_\varepsilon$ and $u_\varepsilon = \Gamma_\varepsilon(u^\sharp).$
		
	\end{itemize}
\end{proposition}
\begin{proof}
	The first part is a direct consequence of the bilinearity of para- and resonant- products and the formula \eqref{equ:defandersonHam}. Indeed, for the first estimate in \eqref{equ:normOpConv2}, by using the formula \eqref{equ:defandersonHam} we obtain
	\begin{align*}
		\norm{T\Gamma u^\sharp   - T\Gamma u_n^\sharp }{\eltwo} &\leq \norm{\Delta \left( u^{\sharp}- u_n^\sharp \right)}{\eltwo}+  \norm{\left( u^{\sharp}- u_n^\sharp \right) \circ \xi }{\eltwo}+ \norm{G  \left( u^{\sharp}- u_n^\sharp \right)}{\eltwo}\\
		& \lesssim_{\Xi}  \|u^\sharp - u_n^\sharp \|_{\mathcal{H}^2}.
	\end{align*}
	The second and third estimates in \eqref{equ:normOpConv2} can be obtained in a similar way by addition and subtraction of suitable (standard and  $\varepsilon$-approximated) cross terms.

	The symmetry of the operator, the estimate \eqref{equ:h2bound} and \eqref{equ:2dconstant} can be proved in a similar way to the bounded domain case \cite{GUZ20}. The closedness shortly follow: suppose that $(u_n) \subset \mathscr{D} (T)$ is
	a sequence with  $u_n \rightarrow u \in L^2$ and $T u_n \rightarrow g  \in L^2$ for some $g \in L^2 .$  Then, by  \eqref{equ:h2bound}  we can write
	\[ \| u_n^{\sharp} - u_m^{\sharp}\|_{\ssp^2} \le 2 \| T u_n -Tu_m \|_{L^2} + C_{\Xi} \| u_n - u_m
	\|_{L^2}  \]
	which implies that  $u_n^{\sharp}$ forms a Cauchy
	sequence in $\ssp^2$ converging to a limit
	$w^{\sharp}$, such that $\Gamma w^{\sharp} = u.$  That is, $u \in \mathscr{D} (T)
	.$  Then, application of  \eqref{equ:h2bound} for a second time yields
	\begin{eqnarray*}
		\| T u - g \|_{L^2} & \le & \| T u - T u_n \|_{L^2} + \| T u_n - g
		\|_{L^2}\\
		& \le & C_\Xi \| w^{\sharp}- u_n^{\sharp} \|_{\mathcal{H}^2} + \| T u_n - g
		\|_{L^2}
	\end{eqnarray*}
	Therefore, we
	obtain $T u= g$.  Hence, the result follows.
\end{proof}

Ultimately, we would like to show that $T$ is a self-adjoint operator.  We have already established $T$ as a symmetric, closed and semibounded operator over the dense domain $\domain{T}$.  We can now show that there exists a real number in the resolvent set and then conclude self-adjointness by using classical results. Namely, we note the following result which follows from \eqref{equ:2dconstant} and classical Lax-Milgram arguments in Hilbert spaces \cite{bab71}.
\begin{proposition}\label{prop:AHselfadj}
	There exists a constant $K_\Xi>0$ which is independent of $\varepsilon$ s.t. 
	\begin{align}
		(K_{\Xi} - T)^{- 1} : L^2 & \rightarrow  \mathscr{D} (T) \label{shift1} \\	
		(K_{\Xi} - T_{\varepsilon})^{- 1} : L^2 & \rightarrow  \ssp^2 \label{shift2}
	\end{align}
	are bounded.  Consequently, $T$ and  $T_\varepsilon$ are a self-adjoint operators respectively over $\domain{T}$ and $\mathscr{H}^2.$
\end{proposition}

For technical simplicity, in the rest of the paper  we shift the operators $T$ and $T_\varepsilon$ by a constant to obtain a positive operator.

\begin{definition} \label{def:setconstant}
	We define the following shifted operators
	\begin{align*}
		A_\varepsilon & :=   T_\varepsilon - K_\Xi \\
		A & := T -  K_\Xi.  \\
	\end{align*}
	When we want to emphasize or indicate the dependence of the operators on this chosen constant we use the notations $A_\varepsilon^{K_\Xi}$ and $A^{K_\Xi}.$
\end{definition}
\begin{remark}\upshape\label{rem:mdissi} 
	In the sequel the constant $K_\Xi$ can be updated to be larger, as needed, without notice. After this addition of constant, by  \eqref{equ:2dconstant} and Definition \ref{def:setconstant}, it follows that $A$ and $A_\varepsilon$ are  positive operators.  
\end{remark} 
In the next theorem, we show the norm resolvent convergence of $A_\varepsilon$ to $A$  in the $\mathcal{H}^\gamma$-norm.
\begin{theorem} \label{thm:normResolventMain}
	We have
	\[
	\| A^{-1}  - A_\varepsilon^{-1} \|_{L^2 \rightarrow \mathscr{H}^\gamma} \rightarrow 0
	\]
	as $\varepsilon \rightarrow 0$. Namely, $A_\varepsilon $ converges to $A$  in  the  norm resolvent sense.
\end{theorem}
\begin{proof}
	Recall that $\Gamma : \mathscr{H}^2 \rightarrow \domain{A}$ and $\Gamma_\varepsilon : \mathscr{H}^2 \rightarrow \mathscr{H}^2$  in which case we have $\Gamma^{-1} : \domain{A} \rightarrow  \mathscr{H}^2 $ and $\Gamma_\varepsilon^{-1} :  \mathscr{H}^2 \rightarrow  \mathscr{H}^2$ .
	By Proposition \ref{prop:opApp} we obtain 
	\[
	\| A_\varepsilon \Gamma_\varepsilon- A\Gamma \|_{\mathscr{H}^2 \rightarrow L^2} \rightarrow 0
	\]
	which implies the   norm resolvent convergence  
	\[
	\| \Gamma_\varepsilon^{-1} A_\varepsilon^{-1}  - \Gamma^{-1}  A^{-1} \|_{L^2 \rightarrow \mathscr{H}^2}  \rightarrow 0.
	\]
	
	To conclude, by using Proposition \ref{lem:gamma} we can write the estimate
	\[
	\| \Gamma \Gamma_\varepsilon^{-1} A_\varepsilon^{-1} - \Gamma \Gamma^{-1}  A^{-1} \|_{\mathscr{H}^\gamma} \lesssim \| \Gamma_\varepsilon^{-1} A_\varepsilon^{-1} - \Gamma^{-1}  A^{-1} \|_{\mathscr{H}^\gamma}
	\]
	
	where, as $\varepsilon \rightarrow 0$ by \eqref{equ:gammaCovId}, we get the convergence
	\[
	\|  A_\varepsilon^{-1}  - A^{-1} \|_{L^2 \rightarrow \mathscr{H}^\gamma} \rightarrow 0.
	\]
	Hence, the result.
\end{proof}

As a corollary of the norm resolvent convergence, we state the following standard result in our setting and then note an application of this to our case, that will be useful when we deal with SPDEs.
\begin{corollary}\cite[VIII.20]{reedsimon1}\label{reedSimonCorrNorm}~\\
	Let $K_\Xi>0$ be as in Definition \ref{def:setconstant}.  The following statements hold:
	\begin{itemize}
		\item If $f$ is a continuous function on $[K_\Xi, \infty)$ that vanishes at infinity then $ \| f\left(-A_{\varepsilon}\right)-f(-A) \| \rightarrow 0.$
		\item For any bounded continuous function $f:[K_\Xi,\infty)\to\mathbb{C}$ we get  
		\[ f(-A_{\varepsilon}) g \to f(-A) g \text{ in }L^2  \]
		for any $g\in L^2$ i.e. strong operator convergence.
	\end{itemize}
\end{corollary}
\begin{proposition}\label{thm:mainResolventHalfDerivative}
	It follows that
	\begin{align*}
		\norm{(-A_\varepsilon)^{-1/2} - (-A)^{-1/2}}{\eltwo \rightarrow \eltwo} \rightarrow 0
	\end{align*}
	as $\varepsilon\rightarrow 0.$
\end{proposition}
\begin{proof}
	By Theorem \ref{thm:normResolventMain} we obtain 
	\[
	\norm{(-A_\varepsilon)^{-1} - (-A)^{-1}}{\eltwo \rightarrow \eltwo} \leq \norm{(-A_\varepsilon)^{-1} - (-A)^{-1}}{\eltwo \rightarrow \sobolev{H}{\gamma}} \rightarrow 0
	\]
	as $\varepsilon\rightarrow 0.$ Observe that  $f(x) = \frac{1}{x^{1/2}}$ continuous function on $[K_\Xi, \infty)$ that vanishes at infinity. By Corollary \ref{reedSimonCorrNorm} we immediately obtain
	\[
	\norm{f(-A_\varepsilon) - f(-A)}{\eltwo \rightarrow \eltwo} \rightarrow 0.
	\]
\end{proof}

\subsection{Energy domain and Functional Inequalities}\label{sec:functionIneq}
Recall that, as in Definition \ref{def:setconstant}, we have constructed the operator $A$ in the full space as a positive self-adjoint operator.  This makes it possible to also define the form domain $\domain{\sqrt{-A}}$ which we do in this section.  After that, we investigate the $L^p$ embedding properties of $\domain{A}$ and $\domain{\sqrt{-A}}$ and how $\norm{u}{\domain{A}}$ and $\norm{u}{\domain{\sqrt{-A}}}$  relate to the standard Sobolev-norms of $u^\sharp$ for $u = \Gamma(u^\sharp).$ In a nutshell,  for the operator $A$ we obtain the full space generalizations of the functional inequalities in the bounded domain case \cite{GUZ20}, including the Brezis-Gallouet inequality.  These functional inequalities will be very important in Section \ref{sec:spde} when we treat the well-posedness of stochastic and NLS.

In the estimates \eqref{equ:2dconstant} and \eqref{equ:h2bound}, respectively, we have seen that the $\sobolev{H}{1}$ and $\sobolev{H}{2}$ norms of $u^\sharp$ can be bounded by $\innerprod{Au}{u}$ and $\norm{Au}{L^2}.$  It turns out that in fact the converse directions in these estimates are also correct.  In the following,  we first define the form domain of $A$ and state this equivalence as Proposition \ref{lem:formdom}.
\begin{definition} \label{def:energySp2d}
	The form domain of $A$, that we denote as $\ED$, is defined as the closure of the domain under the following norm
	\[ \| u \|_{\ED} \assign
	\sqrt{ \langle u, -{A} u \rangle}
	. \]
	Since the operator $-A$ is self-adjoint and positive this is in fact a norm.
\end{definition}
\begin{proposition}
	\label{lem:formdom}{\tmdummy}
	For $u \in \domain{A}$ and $\domain{\sqrt{-A}}$ respectively, we have the following characterizations for the domain and form domain norms
	\begin{equation} \label{equ:normequ1}
	\| u^{\sharp} \|_{\ssp^2} \lesssim \| {A} u \|_{L^2}
	\lesssim \| u^{\sharp} \|_{\ssp^2 .} \end{equation}
	\[ \| u^{\sharp} \|_{\ssp^1} \lesssim \| u
	\|_{\ED} \lesssim \| u^{\sharp} \|_{\ssp^1}. \]
	The same estimates also hold for the regularized operators $A_\varepsilon$ and $\ue = \Gamma(u^\sharp).$
\end{proposition}
\begin{proof}
	The first estimates in these equivalences are respectively given by the estimates \eqref{equ:h2bound} and \eqref{equ:2dconstant}.  The converse direction of the first one is obtained in an easier way by expanding $Au$ as in \eqref{equ:defandersonHam} and then using the estimate \eqref{lem:gamma}.  The other part of the second estimate can be proved similar to the proof of  \eqref{equ:2dconstant}.  Since these results also hold for regularizations, the second statement also follows.  We omit the details.
\end{proof}

In the following, we obtain  $L^p$-embedding type results  for the domain and form  domain of the operator.  We see that these mirror their classical counter parts for the Laplacian operator in $L^2$.  Observe that the results also hold true for the regularized operators $A_\varepsilon$ because the inequality constants depend only on the $\mathscr{E}^\alpha$-norm of the  enhanced noise $\norm{\Xi}{\mathscr{E}^\alpha}$ and we have that $\norm{\Xi^\varepsilon}{\mathscr{E}^\alpha} \lesssim \norm{\Xi}{\mathscr{E}^\alpha}$. For brevity, we give the proofs only for the operator $A$, not the approximations.

\begin{lemma}[$L^p$-estimate]\label{lem:estLpAH}
	For $u\in \ED$ and $p\in[2,\infty)$ we have
	\begin{align}
	\|u\|_{L^p}\lesssim_\Xi \|u\|_{\ED}.
	\end{align}
	Moreover, for $v\in \mathscr{D}(\sqrt{ - A_\varepsilon})=\ssp^1,$ we have
	\begin{align}
	\|v\|_{L^p}\lesssim_\Xi \|v\|_{\mathscr{D}(\sqrt{ - A_\varepsilon})}.
	\end{align}
\end{lemma}
\begin{proof}
	For $p<\infty$ and $\delta = \delta(p)>0$ small enough we have by Sobolev embedding and Propositions \ref{lem:gamma} and \ref{lem:formdom}
	\begin{align*}
	\|u\|_{L^p}\lesssim\|u\|_{\ssp^{1-\delta}}\lesssim\|u^\sharp\|_{\ssp^{1-\delta}}\lesssim\|u^\sharp\|_{\ssp^{1}}\lesssim_\Xi\|(-A)^{1/2}u\|_{L^2}.
	\end{align*}
\end{proof}

In virtue of Proposition \ref{lem:formdom}, the following result is an analogue of the embedding $\mathcal{H}^2 \subset L^\infty$ in 2d.
\begin{lemma} \label{lem:embedinfty}
	For $u \in \domain{A}$ and $v \in  \domain{A_\varepsilon}$ we have
	\begin{align*}
	\|u\|_{L^\infty} \lesssim \|Au\|_{L^2}\\
	\|v\|_{L^\infty} \lesssim \|A_\varepsilon v\|_{L^2}
	\end{align*}
\end{lemma}
\begin{proof}
	By using $\mathcal{H}^2 \subset L^\infty$ and  Propositions \ref{lem:gamma} and \ref{lem:formdom} we have the following chain of inequalities:
	\[
	\|u\|_{L^\infty} \lesssim_\Xi \|u^\sharp\|_{L^\infty}  \lesssim_\Xi  \|u^\sharp\|_{\mathcal{H}^2} \lesssim_\Xi  \|Au\|_{L^2}.
	\]
	Hence, the result.
\end{proof}

In addition to the above result,  we can also prove an inequality that, in some sense, interpolates the  $L^\infty$-norm between the energy norm and the logarithm of the domain norm.  Namely, we  prove a version of Brezis-Gallouet inequality for the operator $A$ in the full space setting.  

We recall the following version of the Brezis-Gallouet inequality  \cite{brga80} and than prove  the version for the operator $A$.

\begin{theorem} [\cite{brga80}]\label{thm:brezisGallouetstandard}
	Let $\Omega$ be a domain in $\mathbb{R}^2$ with compact smooth boundary or simply $\Omega =\mathbb{R}^2$.  Then, for $v \in \mathcal{H}^2(\Omega)$ we have
	\begin{align*}
	\|v\|_{L^\infty}\lesssim \|v\|_{\ssp^1} \sqrt{1+\log(1+\|v\|_{\ssp^2})}.
	\end{align*}
\end{theorem}
\begin{theorem}\label{lem:brgaineq}
	
	For $v \in \domain{A}$ we have
	\[ \| v \|_{L^{\infty}} \lesssim_\Xi\|v\|_{\ED}  \sqrt{(1 + \log (1+ \|v \|_{\domain{A}}))}    . \]
	
	As a corollary, we obtain, for $v \in \mathscr{D}(A_\varepsilon)=\ssp^2$, 
	\[ \| v \|_{L^{\infty}} \lesssim_\Xi\|(-A_\varepsilon)^{1/2}v\|_{L^2}  \sqrt{(1 + \log (1+  \|A_\varepsilon v \|_{L^2}))}    , \]
	where the constant depends on the limiting noise $\Xi$ and can be chosen independently of $\varepsilon$.
	\end{theorem}
\begin{proof}
	By using Theorem \ref{thm:brezisGallouetstandard}, through  \eqref{equ:2dconstant} and Proposition\ref{lem:gamma}, we have that 
	\begin{align*}
	\|v\|_{L^\infty}&\lesssim_\Xi\|v^\sharp\|_{L^\infty}
	\lesssim_\Xi \|v^\sharp\|_{\ssp^1}\sqrt{1+\log(1+\|v^\sharp\|_{\ssp^2})} \\
	&\lesssim_\Xi\|v\|_{\ED}\sqrt{1+\log(1+\|v\|_{\domain{A}})}
	\end{align*}
	Observe that  the same estimate also holds  for $\mathscr{D}(A_\varepsilon)$, since the estimate   \eqref{equ:2dconstant} hold with constants independent of $\varepsilon$. 
	Hence, the results follow.
\end{proof}

\subsection{Localization Formula for the domain}\label{sec:extrapolatedOP}

As the space $\mathbb{R}^2$ is unbounded the treatment of, for instance, well-posedness of PDEs need to use exploit some form of "localization''.  For the case of stochastic NLS we can use the $\Gamma$-map, that we developed in Proposition \ref{lem:gamma} in weighted setting, to obtain weighted subspaces of the domain and form domain.  For the stochastic NLW one needs to exploit the idea of finite speed of propagation property which requires being able to localize with functions $\phi \in C_c^\infty.$ However a quick inspection of regularities in the ansatz \eqref{equ:ansatzmain} shows that there are no compactly supported/smooth functions in the domain or there does not seem to be an obvious way to localize domain elements.  In this section, we solve this problem and show a localization property for the domain elements and prove a precise product formula. Namely, for $u \in \domain{A}$ and $\phi \in C_c^\infty,$  we give a formula for $A(\phi u)$ which we will use intensively in the treatment of stochastic NLW later on.

We first note the following Lemma which gives the embedding of certain Sobolev spaces with negative regularity into $ (\domain{\sqrt{-A}})^*.$

\begin{lemma}\label{lem:dualSobolevOptimal}
	For every $0 <\beta <1$ we have that
	\[
	\sobolev{H}{-\beta} \subset (\domain{\sqrt{-A}})^*.
	\]
\end{lemma}
\begin{proof}
	We observe the following chain of inequalities
	\[
	\norm{u}{\sobolev{H}{\beta}} \leq \norm{u^\sharp}{\sobolev{H}{\beta}} \leq \norm{u^\sharp}{\sobolev{H}{1}}\lesssim  \norm{u}{\domain{\sqrt{-A}}}
	\]
	where we have used Lemma \ref{lem:gamma} in the first and Proposition \ref{lem:formdom} in the third inequality.
	This implies that
	\[
	\domain{\sqrt{-A}} \subset \sobolev{H}{\beta}.
	\]
	But in the duality the embeddings are reversed  so that we obtain 
	\[
	\sobolev{H}{-\beta} \subset (\domain{\sqrt{-A}})^*.
	\]
\end{proof}

We are ready to prove the main results of this section, which provide a localization property for the domain elements.  In the subsequent result, we also give a precise formula for $A(\phi u).$
\begin{proposition}\label{prop:localDomain}
	For fixed $u \in \domain{A}$ and $\phi \in C_c^\infty$, $\phi u$  is in $\domain{\sqrt{-A}}$.  In addition, for $\Psi \in\domain{\sqrt{-A}}$  the following bound holds 
	\[
	|\innerprod{\phi u}{A \Psi}| \leq \norm{Au}{\eltwo} \norm{A^{1/2} \Psi}{\eltwo} \norm{\phi}{W^{2,\infty}}.
	\]
	Moreover, for $\ue = \Gamma_\varepsilon(u^\sharp)$ we have that $\phi u_\varepsilon$  is in $\domain{\sqrt{-A_\varepsilon}}$ and for $\Psi_\varepsilon \in\domain{\sqrt{-A_\varepsilon}}$  the same estimate also holds for the approximations as well, with constants uniform  in $\varepsilon >0$:
	\[
	|\innerprod{\phi \ue}{A_\varepsilon \Psi_\varepsilon}| \leq \norm{A_\varepsilon \ue}{\eltwo} \norm{A_\varepsilon^{1/2} \Psi_\varepsilon}{\eltwo} \norm{\phi}{W^{2,\infty}}.
	\]
\end{proposition}

\begin{proof}
	We directly show that $\phi u$ defines a bounded linear functional over $\domain{\sqrt{-A}}$, with respect to the inner product $\innerprod{\cdot}{A\cdot}$.  Over the dense space $\Psi \in \domain{A}$, regularize $\Psi_\varepsilon = A_\varepsilon^{-1} A \Psi$.   We have that
	\begin{align*}
	\innerprod{\phi u}{A \Psi} = \innerprod{\phi u}{A_\varepsilon \Psi_\varepsilon}
	\end{align*} 
	Observe that for $u_\varepsilon = \Gamma_\varepsilon(u^{\sharp})$ we have
	\[
	\innerprod{\phi u_\varepsilon }{A_\varepsilon \Psi_\varepsilon} \rightarrow \innerprod{\phi u}{A \Psi}
	\]
	as $\varepsilon \rightarrow 0$. For the first term we observe that
	\begin{align*}
	\innerprod{\phi u_\varepsilon }{A_\varepsilon \Psi_\varepsilon} &= \innerprod{A_\varepsilon \phi u_\varepsilon }{\Psi_\varepsilon} \\
	 &=\innerprod{ A_\varepsilon u_\varepsilon }{\phi   \Psi_\varepsilon} + \innerprod{\nabla u_\varepsilon }{2 \nabla \phi  \Psi_\varepsilon} + \innerprod{u_\varepsilon \Delta\phi  }{\Psi_\varepsilon}.
	\end{align*}
	Now, by using Proposition \ref{prop:leibnizProductRule} we can write
	\begin{align*}
	& |\innerprod{\phi u_\varepsilon }{A_\varepsilon \Psi_\varepsilon}|\\
	&\lesssim \norm{A_\varepsilon u_\varepsilon}{\eltwo} \norm{\phi   \Psi_\varepsilon}{\eltwo} + \norm{\ue}{\sobolev{H}{1-\kappa}} \norm{\nabla \phi  \Psi_\varepsilon}{\sobolev{H}{\kappa}}+ \norm{u_\varepsilon}{ \eltwo } \norm{\Delta\phi \Psi_\varepsilon}{\eltwo}\\
	& \lesssim \norm{A_\varepsilon u_\varepsilon}{\eltwo} \norm{A_\varepsilon^{1/2} \Psi_\varepsilon}{\eltwo} \norm{\phi}{W^{2,\infty}}.
	\end{align*}
	Taking the limit as $\varepsilon \rightarrow 0$ gives the bound
	\[
	|\innerprod{\phi u}{A \Psi}|  \lesssim \norm{A u}{\eltwo} \norm{A^{1/2} \Psi}{\eltwo}\norm{\phi}{W^{2,\infty}}.
	\]
	This shows that for each fixed $u \in \domain{A}$ and $\phi \in C_c^\infty,$ $\phi u$ defines a linear functional over $\domain{\sqrt{-A}}$.  Hence, the result.
\end{proof}
\begin{theorem}\label{thm:compactAHformula}
	With the notation of Proposition \ref{prop:localDomain}, for any $\Psi \in \domain{\sqrt{-A}}$ it follows that
	\[
	\innerprod{A_\varepsilon (\phi \ue)}{\Psi} \rightarrow \innerprod{A(\phi u)}{\Psi} 
	\]
	Consequently, we deduce that the formula 
	\begin{align}
	A (\phi u ) = 	 \phi  A u + 2 \nabla \phi  \nabla u + u \Delta\phi  
	\end{align}
	holds  in $(\domain{\sqrt{-A}})^*$.
\end{theorem}

\begin{proof}
	For $\Psi \in \domain{A}$ and $ \Psi_\varepsilon = \Gamma_\varepsilon(\Psi^\sharp)$ we have that
	\begin{align*}
	\innerprod{A_\varepsilon \phi u_\varepsilon }{\Psi_\varepsilon} =  
	\innerprod{\phi  A_\varepsilon u_\varepsilon }{  \Psi_\varepsilon} + \innerprod{2 \nabla \phi  \nabla u_\varepsilon }{ \Psi_\varepsilon} + \innerprod{u_\varepsilon \Delta\phi  }{\Psi_\varepsilon}.
	\end{align*}
	For the term on the left hand side we have 
	\begin{align*}
	\innerprod{A_\varepsilon \phi u_\varepsilon }{\Psi_\varepsilon} = 	\innerprod{\phi u_\varepsilon }{A_\varepsilon  \Psi_\varepsilon} \rightarrow \innerprod{\phi u }{A  \Psi}
	\end{align*}
	as $\varepsilon \rightarrow 0$.  For the rest of the terms we readily have the convergence
	\begin{align*}
	\innerprod{\phi  A_\varepsilon u_\varepsilon }{  \Psi_\varepsilon} + \innerprod{2 \nabla \phi  \nabla u_\varepsilon }{ \Psi_\varepsilon} + \innerprod{u_\varepsilon \Delta\phi  }{\Psi_\varepsilon} \rightarrow 	 \innerprod{\phi  A u }{  \Psi} + \innerprod{2 \nabla \phi  \nabla u }{ \Psi} + \innerprod{u \Delta\phi  }{\Psi}
	\end{align*}
	Therefore we can conclude the equality that holds in the weak sense
	\begin{align}\label{equ:locequ1}
	A (\phi u ) = 	 \phi  A u + 2 \nabla \phi  \nabla u + u \Delta\phi  
	\end{align}
	 in the space $(\domain{A})^*,$ where we have used Proposition \ref{prop:localDomain}.  By inspecting the regularities through Lemma \ref{lem:dualSobolevOptimal}, we obtain that in fact the equality \eqref{equ:locequ1} holds in $(\domain{\sqrt{-A}})^*$.
	\end{proof}

\subsection{Construction of the Anderson Hamiltonian}\label{sec:farisLavine}

So far we have constructed the operator $A,$ which has the partial noise $\xi$ as a potential, in the full space as a positive self-adjoint operator and proved the norm resolvent convergence of the regularizations $A_\varepsilon.$  In particular, the Proposition \ref{lem:gamma} gives us the possibility to identify weighted subspaces of the domain as $\Gamma(\sobolew{H}{2}{s}) \subset \domain{A} = \Gamma(\sobolev{H}{2}).$  In this section, we construct the full Anderson hamiltonian $H = -A- \eta,$ where $\eta$ is a potential with a mild growth in the full space. That is to say, we have a Hamiltonian $H$ which is a sum of an operator that we have shown to be positive  self-adjoint and a potential $\eta$ which does not grow too wildly by  \eqref{rem:growthBddPart}. In a result proved in \cite{FL94}, known as Faris-Lavine Theorem,  the essential self-adjointness of such Hamiltonians was proved by establishing a commutator estimate with a well-chosen auxiliary operator $N,$ which needs to be positive self-adjoint. In this section, we define precisely the auxiliary operator $N$ and show that it is positive and self-adjoint over a well defined weighted subspace of $\domain{A}.$  Finally, we establish the commutator estimate of the full Anderson Hamiltonian $H$ and the operator $N$ and conclude the essential self-adjointness of the Anderson Hamiltonian in the full space. 

We first recall the following result which we will use from \cite{FL94} with its proof this time and then precisely introduce the operators $H$ and $N$ in our setting.

\begin{proposition}\cite{FL94} \label{prop:FLperpArg}
	Let $H$ be  symmetric and $N \geq 1$ be  positive self-adjoint operators with $\DD (N) \subset \DD(H)$  which  satisfy the following estimate of their associated quadratic forms
	\begin{equation}\label{equ:farisLavComm}
	\pm \langle [H, N]f, f \rangle:= \pm i \left(  \innerprod{N f}{Hf}  -\innerprod{ H f}{Nf} \right) \leq c \innerprod{Nf}{f}
	\end{equation}  
for some constant $c>0.$ Then, it follows that $H$ is an essentially  self-adjoint operator over $\DD (N)$ .
\end{proposition}
\begin{proof}
	Recall that $H$ is essentially self-adjoint if and only if the range of $H - \mu i$ is dense when $|\mu|$ is sufficiently large \cite{reedsimon2}.
	
	Suppose that $f \in L^2$ is perpendicular to the range of $H - \mu i.$  By assumption we have that $N^{-1} f  \in \domain{H}.$  Then, we have in particular that $\innerprod{f}{(H - \mu i)N^{-1} f } = 0.$  This implies
	\begin{align*}
	&0= \text{Im}\left[\innerprod{f}{(H - \mu) i(N^{-1} f )} \right]\\
	&= \frac{\innerprod{f}{(H - \mu i)N^{-1} f } - \innerprod{(H - \mu i)N^{-1} f} {f}}{2i}
	\end{align*}
	which leads to the equality
	\begin{align*}
	-4i\mu \innerprod{f}{N^{-1}f} &= \innerprod{f}{H N^{-1} f }  - \innerprod{H N^{-1} f }{f}\\
	&= -\left\langle N^{-1} f, [H, N] N^{-1} f\right\rangle
	\end{align*}
	By assumption we obtain that
	\begin{align*}
	\pm 4\mu \innerprod{f}{N^{-1}f}  \leq c \innerprod{f}{N^{-1}f} 
	\end{align*}
	By choosing $4|\mu|>c,$ we see that $f=0,$ since $N\geq 1$ and thereby $\innerprod{f}{N^{-1}f}  \not = 0.$  Hence, the result follows.
	\end{proof}
\begin{definition}\label{def:operators}
We define
	\begin{equation}
	\mathcal{C}_2 := \{f \in \domain{A} ~|~ \japanbrac^2 f \in L^2 \} = \domain{A} \cap \sobolew{L}{2}{2}.
	\end{equation}
	Observe that $\mathcal{C}_2 $ is non-empty as $\Gamma(\sobolew{H}{2}{2}) \subset \mathcal{C}_2$ by Proposition \ref{lem:gamma}.  For each element $f \in \mathcal{C}_2$ we consider the approximation $f_\varepsilon := \Gamma_\varepsilon(f^\sharp)$ for $f^\sharp \in \sobolew{H}{2}{2}$ so that we have $f_\varepsilon \in \sobolew{H}{2}{2}.$

	Over the domain $\mathcal{C}_2$ we introduce the following operators
\begin{align*}
	H &:= -A - \eta\\
	N &:= H + c|x|^2
	\end{align*}
	and over $\sobolew{H}{2}{2}$ their natural approximations
	\begin{equation}\label{equ:farisDefOperators}
	\begin{aligned}
	H_\varepsilon &:= -A_\varepsilon - \eta_\varepsilon\\
	N_\varepsilon &:= H_\varepsilon + c|x|^2
	\end{aligned}
	\end{equation}
	where the operators $A$ and $A_\varepsilon$ are as  in Definition \ref{def:setconstant}, $\eta_\varepsilon = \eta \cdot \mathds{1}_{\{\eta \leq 1/\varepsilon\}}$ and $c>0$ is a precise constant, to be fixed in the following Remark.  
\end{definition}
\begin{remark}\label{remm:farisConstantAdjust}
	By  \eqref{rem:growthBddPart},  $\eta$ satisfies
	\begin{equation}
	-\eta(x) \geq -c'\left( |x|^2 + 1\right)
	\end{equation}
	in the sense of quadratic forms, for some constant $c'>0$.   We simply choose $c \geq c'+2.$
	
	In Definition \ref{def:setconstant} we choose $K_\Xi >c+1.$ In  \eqref{equ:farisDefOperators} we  absorb the constant $c>0$ to the constant $K_\Xi,$ thereby updating the constant as $K_\Xi \rightarrow  K_\Xi - c.$  Observe that these adjustments are made so that $A_\varepsilon, A, N_\varepsilon, N$ are all positive operators. With these choice of constants observe that we have the following inequalities of quadratic forms 
	\begin{align}\label{equ:etaLowerPositive}
	\langle f_\varepsilon, (-A _\varepsilon+ c|x|^2 - \eta_\varepsilon)f_\varepsilon \rangle \geq \langle f_\varepsilon, (c(|x|^2+1) - \eta_\varepsilon) f_\varepsilon \rangle \geq \langle f_\varepsilon,( |x|^2+1) f_\varepsilon \rangle.
	\end{align}
\end{remark}
For ease of notation, we introduce the operators $p$ and $q$ and recall their well-known commutation properties.
\begin{definition}
	Over the domain $\sobolew{H}{2}{2}$  we define the following operators
	\begin{align*}
	p &:= -i \cdot \nabla = -i (\frac{\partial}{\partial x_1}, \dots, \frac{\partial}{\partial x_j})\\
	q &:= (x_1, \dots, x_j)
	\end{align*}
	so that
	\begin{align*}
	p^2 &=  \Delta \\
	q \cdot q  &= q^2 = |x|^2 = x_1^2 + \dots + x_j^2.
	\end{align*}
	Accordingly we define
	\begin{align*}
	p_j &:= -i \cdot \frac{\partial}{\partial x_j} \\
	q_j &:= x_j \cdot
	\end{align*}
We also recall the following standard commutation relations that $p$ and $q$ satisfy:
\begin{equation}\label{lem:spaceMomentum1}
\begin{aligned}
 & [p,q] =- i\\
&\pm i [p^2, q^2] = \pm 2 (p \cdot q + q \cdot p).
\end{aligned}
\end{equation}
\end{definition}
In order to prove that $H$ is essentially self-adjoint  we want to show  the operators $H$ and $N$ satisfy the assumptions of Proposition \ref{prop:FLperpArg}.  In particular, we would like to check the commutator condition \eqref{equ:farisLavComm} and that $N$ is a self-adjoint operator.  Before proving these results, we need the following important estimates, essentially showing that $\domain{N}$-norm bounds both the $\domain{H}$-norm and the $\sobolew{L}{2}{2}$-norm.  In the subsequent result, by using the estimate, we conclude that $(N, \mathcal{C}_2)$ is a  self-adjoint operator.
\begin{proposition}\label{prop:nWeightedL2bound}
	The operator $(N, \mathcal{C}_2)$ and the regularized operators $(N_\varepsilon, \sobolew{H}{2}{2})$  satisfy the following estimates
	\begin{align*}	
	|| H f||_{L^2}^2 + a||q^2 f||_{L^2}^2  &\leq  ||Nf||_{L^2}^2 + b ||f||_{L^2}^2\\
	|| H_\varepsilon f||_{L^2}^2 + 	a||q^2 f_\varepsilon||_{L^2}^2  &\leq ||N_\varepsilon f||_{L^2}^2 + b ||f_\varepsilon||_{L^2}^2
	\end{align*} 
	with precise constants $a, b>0.$
\end{proposition}
\begin{proof}
	For $f \in \CC_{2}$ we use the approximations $f_\varepsilon = \Gamma_\varepsilon(f^\sharp),$ from Proposition \ref{prop:opclosed}.  For $c_0 = \frac{c}{3}$, by using the identities \eqref{lem:spaceMomentum1} we directly  calculate 
	\begin{equation}\label{equ:farisQuadNormBound}
	\begin{aligned}
	\langle N_{\varepsilon} f_\varepsilon, N_{\varepsilon}  f_\varepsilon \rangle &= \langle H_{\varepsilon} f_\varepsilon, H_{\varepsilon} f_\varepsilon \rangle + 9c_0^2 \langle q^2 f_\varepsilon, q^2 f_\varepsilon \rangle + 3 c_0 (\langle H_{\varepsilon}f_\varepsilon  ,  q^2 f_\varepsilon \rangle +  \langle q^2 f_\varepsilon,  H_{\varepsilon}f_\varepsilon \rangle ) \\
	&=  \langle H_{\varepsilon} f_\varepsilon, H_{\varepsilon}f_\varepsilon \rangle+ 9c_0^2   \langle q^2 f_\varepsilon, q^2 f_\varepsilon \rangle + 6c_0 \sum_j \langle  H_{\varepsilon} q_j f_\varepsilon,   q_j f_\varepsilon \rangle\\
	&+ 3c_0 \sum_j \langle [q_j, [q_j, H_\varepsilon]]f_\varepsilon, f_\varepsilon \rangle \\
	&= \langle H_{\varepsilon} f_\varepsilon, H_{\varepsilon} f_\varepsilon \rangle+(9c_0^2 - 6 c_0) \langle q^2 f_\varepsilon, q^2 f_\varepsilon \rangle+ 6c_0 \sum_j \langle  (H_{\varepsilon} + q^2) q_j f_\varepsilon,   q_j f_\varepsilon \rangle \\
	&- 12c_0 \langle  f_\varepsilon, f_\varepsilon \rangle
	\end{aligned}
	\end{equation}
	where $9c_0^2 - 6 c_0>0$ by Remark \ref{remm:farisConstantAdjust}.  By  \eqref{rem:growthBddPart}, there exist a constant $|e| <\infty$ such that  $q^2 > e$ implies $-\eta(x) + q^2 \geq 1,$ in the sense of quadratic forms of multiplication operators.  Therefore there exists a constant $|z|<\infty$ defined as
	\begin{equation}\label{equ:setConstInf}
	\underset{q^2 \leq e}{\inf} -\eta(x) + q^2 = z.
	\end{equation}
	Now, we update the constant $K_\Xi$ to $K'_\Xi =  K_\Xi + |z| +1.$  By using the notation in definition \ref{def:setconstant} we have 
	\[
	H_\varepsilon + q^2 = -A_\varepsilon^{K_\Xi + z +1}  -\eta_\varepsilon + q^2 =  -A_\varepsilon^{K_\Xi}  -\eta_\varepsilon + q^2 + z +1.
	\]
	Therefore, by the last item of Proposition \ref{prop:opApp}, Definition \ref{def:setconstant} and \eqref{equ:setConstInf}  it follows that
	\begin{equation}\label{equ:farisPosPhiN}
	6c_0 \sum_j \langle  (H_{\varepsilon} + q^2) q_j  f_\varepsilon,   q_j  f_\varepsilon \rangle \geq 0.
	\end{equation}
	By using \eqref{equ:farisQuadNormBound} and \eqref{equ:farisPosPhiN},  we obtain the estimate
	\begin{equation}
	\begin{aligned}
	& \langle N_{\varepsilon} f_\varepsilon, N_{\varepsilon}  f_\varepsilon \rangle + 12c_0 \langle  f_\varepsilon, f_\varepsilon \rangle \geq \langle H_{\varepsilon}  f_\varepsilon, H_{\varepsilon} f_\varepsilon \rangle + (9c_0^2 - 6 c_0) \langle q^2  f_\varepsilon, q^2  f_\varepsilon \rangle.
	\end{aligned}
	\end{equation}
	This settles the estimate with the approximations.  In order to get the other result,  we  take  the limit as  $\varepsilon\rightarrow 0,$  by using Propositions \ref{prop:opclosed} and \ref{prop:opApp}. Hence, both of the results follow.
\end{proof}
\begin{theorem}
	The operator $N$ is a self-adjoint operator over $\mathcal{C}_2$.
\end{theorem}
\begin{proof}
	Observe that density of the domain and symmetry is immediate from the definition and with a proof similar to Proposition \ref{prop:opApp}.  The non-trivial part is showing  closedness of the operator which we do now.  
	
	We want to show that $f_n \rightarrow f$  and $ Nf_n \rightarrow g$  in $L^2$  implies  $f \in \mathcal{C}_2$ and $Nf = g.$   This essentially follows from Proposition \ref{prop:nWeightedL2bound}.  Indeed,  we have the estimate
	\begin{equation*}
	a||q^2( f_n - f_m)||_{L^2}^2  \leq ||N(f_n - f_m)||_{L^2}^2 + b ||f_n - f_m||_{L^2}^2
	\end{equation*}
	which implies  $f_n \rightarrow f$ also in $\sobolew{L}{2}{2}.$  We have that
	\begin{equation}
	||(cq^2 + \eta)( f - f_m)||_{L^2}^2 \lesssim 2c \ ||(q^2+1)( f - f_m)||_{L^2}^2.
	\end{equation}
	That is to say $A f_n$ converges to $g - (cq^2 + \eta)f .$  Since, $A$ is a closed operator we have that $A f  = g - (cq^2 + \eta)f$ which in turn implies  $ Nf = g.$  This shows closedness.
	Semiboundedness follows from  \eqref{equ:2dconstant} and \eqref{equ:etaLowerPositive}.   Finally, self-adjointness of $N$  follows from Proposition \ref{prop:selfadj}.
\end{proof}

Finally, after these results we are ready to prove the main result of this section.
\begin{theorem}\label{thm:essSAgeneral}
	For some constant $C>0$ and   $f \in\CC_2,$  $H$ and $N$ satisfies the following commutator estimate
	\begin{align*}
	\pm i \left(  \innerprod{N f}{Hf}  -\innerprod{ H f}{Nf} \right) \leq C \innerprod{Nf}{f}.
	\end{align*}
Consequently, $H$ is an essentially self-adjoint operator on $\CC_2.$
\end{theorem}
\begin{proof}
	We show the estimate with the approximations $f_\varepsilon = \Gamma_\varepsilon(f^\sharp)$ and then pass to the limit. We want to show 
	\begin{align*}
	\pm i \left(  \innerprod{N_\varepsilon f_\varepsilon}{H_\varepsilon f_\varepsilon}  -\innerprod{ H_\varepsilon f_\varepsilon}{N_\varepsilon f_\varepsilon} \right) \leq C \innerprod{N_\varepsilon f_\varepsilon}{f_\varepsilon}.
	\end{align*}
	By \eqref{lem:spaceMomentum1} we have that
	\begin{align*}	
	\pm i \left(  \innerprod{N_\varepsilon f_\varepsilon}{H_\varepsilon f_\varepsilon}  -\innerprod{ H_\varepsilon f_\varepsilon}{N_\varepsilon f_\varepsilon} \right) &=  \innerprod{\left( \pm i[p^2, q^2] \right) f_\varepsilon}{f_\varepsilon}\\
	&= \innerprod{\left( \pm 2 (p\cdot q + q\cdot p)\right) f_\varepsilon}{f_\varepsilon}\\
	&=\sum_{i=1}^{2}  \innerprod{\left( \pm 2 (p_i q_i + q_i  p_i)\right) f_\varepsilon}{f_\varepsilon}.
	\end{align*}
	As the inequality is required to be true up to a constant we omit the exact constants in the later calculations. We want to show that for some large constant $C>0$ we have
	\[
	| \langle q_i f_\varepsilon, p_if_\varepsilon \rangle| \leq C \langle f_\varepsilon, (-A_\varepsilon + cq^2  -\eta_\varepsilon)f_\varepsilon \rangle.
	\]
	By using a similar argument to the proof of Proposition \ref{lem:gamma} one can also show the reverse estimate
	\begin{equation*}
	\norm{f^\sharp}{\sobolew{L}{2}{}} \leq C_\Xi \norm{f_\varepsilon}{\sobolew{L}{2}{}}. 
	\end{equation*}
	Then, by Proposition \ref{lem:formdom} and interpolation we obtain
	\begin{equation*}
	\norm{f^\sharp}{\sobolew{H}{1/2}{1/2}} \leq \norm{f_\varepsilon}{\domain{\sqrt{-A_\varepsilon}}}^{1/2}\norm{f_\varepsilon}{\sobolew{L}{2}{}}^{1/2} .
	\end{equation*}
	We calculate
	\begin{equation}\label{equ:farisLavinend1}
	\begin{aligned}
	\langle q_i f_\varepsilon, p_if_\varepsilon \rangle &=  \langle q_i^{1/2} f_\varepsilon, p_i q_i^{1/2} f_\varepsilon \rangle - \frac{1}{2}\langle f_\varepsilon, f_\varepsilon \rangle \\
	&=  \langle p_i^{1/2} q_i^{1/2} f_\varepsilon, p_i^{1/2} q_i^{1/2} f_\varepsilon \rangle - \frac{1}{2}\langle f_\varepsilon, f_\varepsilon \rangle .
	\end{aligned}
	\end{equation}
	By Propositions \ref{lem:gamma} and \ref{lem:formdom} we have
	\begin{equation}\label{equ:farisLavinend2}
	\begin{aligned}
	&\langle p_i^{1/2} q_i^{1/2} f_\varepsilon, p_i^{1/2} q_i^{1/2} f_\varepsilon \rangle \\
	&\leq \norm{ p_i^{1/2} q_i^{1/2} f_\varepsilon}{\eltwo}^2\\
	&\leq \norm{f_\varepsilon}{\sobolew{H}{1/2}{1/2}}^2  \leq \norm{f^\sharp}{\sobolew{H}{1/2}{1/2}}^2 \leq \norm{f_\varepsilon}{\domain{\sqrt{-A_\varepsilon}}} \norm{f_\varepsilon}{\sobolew{L}{2}{}}
	\end{aligned}
	\end{equation}
and by Remark \ref{remm:farisConstantAdjust} the following holds
	\begin{equation}\label{equ:farisLavinend3}
	\begin{aligned}
		\norm{f_\varepsilon}{\sobolew{L}{2}{}}^{2}  \leq \langle f, (c(q^2+1) - \eta_\varepsilon) f \rangle
	\leq \langle f_\varepsilon, (-A _\varepsilon+ cq^2  -\eta_\varepsilon)f_\varepsilon \rangle.
	\end{aligned}
	\end{equation}
	By putting together \eqref{equ:farisLavinend1}, \eqref{equ:farisLavinend2} and \eqref{equ:farisLavinend3}, we finally arrive at the  estimate
	\begin{align*}
	| \langle q_if_\varepsilon, p_if_\varepsilon \rangle| &\leq \norm{f_\varepsilon}{\domain{\sqrt{-A_\varepsilon}}} \norm{f_\varepsilon}{\sobolew{L}{2}{}} + \frac{1}{2}\langle f_\varepsilon, f_\varepsilon \rangle\\ &\leq   \langle f_\varepsilon, (-A_\varepsilon + cq^2 - \eta_\varepsilon)f_\varepsilon \rangle + \frac{1}{2}\langle f_\varepsilon, f_\varepsilon \rangle
	\end{align*}
	and taking the limit as $\varepsilon\rightarrow 0$ returns the proposed estimate. We have shown that all the assumptions of Proposition \ref{prop:FLperpArg} are satisfied, hence the result follows.
	\end{proof}

\section{Stochastic nonlinear PDEs}\label{sec:spde}
To recap, in Section \ref{sec:anderson} we have constructed the operator $A$ (from Definition \ref{def:setconstant}) which is defined with the partial noise $\xi,$ as a semibounded self-adjoint operator and also constructed the form domain $\domain{\sqrt{-A}}$.  In particular, we have realized the domain as the image of a well defined map as $\domain{A} = \Gamma(\sobolev{H}{2})$ which we showed, in Lemma \ref{lem:gamma}, to satisfy weighted Sobolev estimates. We have also proved the norm resolvent convergence of the regularized versions $A_\varepsilon$ and obtained useful functional inequalities giving Sobolev embedding results for the $\domain{A}$ and  $\domain{\sqrt{-A}}.$  In particular, we have proved the Brezis-Gallouet inequality in the full space.  Furthermore, in Proposition \ref{prop:localDomain} we have obtained a way to localize the elements of the $\domain{A}$ with smooth functions and obtained the following product type formula
	\begin{align*}
	A (\phi u ) = 	 \phi  A u + 2 \nabla \phi  \nabla u + u \Delta\phi. 
\end{align*}
All of these developments, especially our results in the weighted setting and localization for domain elements sets the background for treatment of stochastic NLS and NLW equations in the full space.  Because the treatment in the full space, due to non-compactness and the growing potential $\eta,$ requires obtaining some form of localization property of the corresponding PDE.

Thereby, in this section we set out to study the well-posedness of the stochastic cubic NLS and NLW equations with multiplicative noise on the full space $\mathbb{R}^2,$ whose linear part is given by the full Anderson Hamiltonian $H= A+ \eta.$  In Subsection \ref{sec:NLS}, we obtain the following convergence, existence and uniqueness result for the NLS.
\begin{theorem}\label{thm:mainwellposed2}
	For $ l>0,$ consider the following stochastic (defocusing) NLS with localized initial conditions
	\begin{align}\label{equ:nlsstandard2}
		i \partial_t u (t, x) &= A u(t,x) + \eta(x) u(t,x) - u(t,x) |u(t,x)|^2, \ x\in \mathbb{R}^2\\
		u(0)&:= u^o \in \Gamma\left(\sobolew{H}{2}{l/2}\right).
	\end{align}
	With the regularized and localized initial data, there exists a (sub-) sequence of solutions 
	\begin{equation}
		\ue \in C^1([0, T], L^2)  \cap C([0, T], \sobolev{H}{2}) 
	\end{equation} 
	to the regularized PDE
	\begin{align}
		i\partial_t u_\varepsilon = A_\varepsilon \ue + \eta_\varepsilon u_\varepsilon  - \ue |\ue|^2 \nonumber \\
		u_\varepsilon(0)=: u_\varepsilon^o :=\Gamma_\varepsilon \left((u^o)^{\sharp}\right)  \in \sobolew{H}{2}{l/2}
	\end{align}
	which converge to the following unique  solution of \eqref{equ:nlsstandard2} in $\domain{\sqrt{-A}})^*$ for  all $t \in [0,T]$
	\begin{equation}
		u \in C^{1} ([0,T], (\domain{\sqrt{-A}})^*) \cap C([0,T], \domain{\sqrt{-A}}).
	\end{equation}
	Moreover, the equation has the conserved energy
	\(E(u):=-\frac{1}{2} \innerprod{A u}{u} +\frac{1}{2} \int \eta |u|^2 dx + \frac{1}{4}\int |u|^{4} \)
	such that $E(u(t)) = E(u(0))$ for  all $t \in [0,T].$
\end{theorem}

In fact, it is possible to show that there also exists a $\domain{A}$-solution $u \in W^{1, \infty} ([0,T], L^2) \cap L^\infty([0,T], \domain{A})$ that satisfies the equation uniquely in $\eltwo$ for almost all $t \in [0,T],$ which we show in Theorem \ref{thm:weakExist}.  By using interpolation arguments, it follows that for precise $\delta >0, \nu>0$ the solution satisfies
	\[
u \in C([0,T], \domain{\sqrt{-A}}) \cap C([0,T], \sobolew{L}{2}{\mu}).
\]

In Subsection \ref{sec:nlw} we treat the stochastic cubic NLW equation.  More precisely, by using our localization results from Section \ref{sec:extrapolatedOP} intensively for the construction and approximation of the initial data and finite speed of propagation property. We prove the following convergence, existence and uniqueness theorem for stochastic NLW.

\begin{theorem}\label{corr:convWaveSecDer2}
	For $\phi_R \in \holdercont{c}{\infty}(\R^2)$ supported in the unit ball $B_R(0),$ let  $(\phi_R u_0, \phi_R u_1) \in \domain{A} \times \domain{\sqrt{-A}}$ be the initial data for the stochastic cubic NLW
	\begin{align}\label{equ:waveEquApp3}
		\partial_t^2 u(t,x) &= A u(t,x) + \eta(x) u(t,x) -u(t,x)^3, \ x \in \mathbb{R}^2.
	\end{align} 
	With the regularized and localized initial data, there exists a (sub-) sequence of solutions 
	\begin{equation}\label{thm:existanceWithVarep}
		u_\varepsilon \in C([0,T];\ssp^2)\cap C^1([0,T];\ssp^1)\cap C^2([0,T];L^2).
	\end{equation} 
	to the regularized PDE
	\begin{align}\label{equ:waveEquApp2}
		\partial_t^2 \ue &= A_\varepsilon \ue + \eta_\varepsilon\ue -\ue^3 \nonumber \\
		\ue(0) :=u_0^\varepsilon&:=\phi_R (-A_\varepsilon)^{-1}(-A) u_0\\
		\partial_t \ue(0):= u_1^\varepsilon&:=\phi_R (-A_\varepsilon)^{-1/2}(-A)^{1/2} u_1
	\end{align}
	which converge to the following unique  solution of \eqref{equ:waveEquApp3} 
	\begin{equation}\label{equ:existenceStrongsol}
		u \in C^{2} ([0,T], (\domain{\sqrt{-A}})^*) \cap C^{1}([0,T], L^2) \cap C([0,T], \domain{\sqrt{-A}}).
	\end{equation}
	which satisfies the finite speed of propagation property. 	Moreover, the equation has the conserved energy
	\(E(u):=\frac{1}{2}\left\langle\partial_{t} u, \partial_{t} u\right\rangle-\frac{1}{2}\langle u, A u\rangle+\frac{1}{2}\langle u, \eta u\rangle+\frac{1}{4} \int|u|^{4}\)
	such that $\frac{d}{dt}E(u(t)) = 0$ for all $t \in [0,T].$
\end{theorem}
Of course, a natural question is whether or not one can obtain also  solutions locally in $C^{2} ([0,T], (\domain{\sqrt{-A}})^*) \cap C^{1}([0,T], L^2) \cap C([0,T], \domain{\sqrt{-A}})$ when the initial data is not necessarily compactly supported. Similar to the standard wave equation,  this is possible by approximating the initial data with compactly supported data and then using the finite speed of propagation.  We detail these comments in Remark \ref{rem:nlwglobal}.   

The proof of these results heavily uses our weighted space setting and the localization results in Section \ref{sec:extrapolatedOP}. For NLS, the compactness is obtained in the weighted spaces after proving the a priori estimates.  Essential to the proof of a priori estimates is  energy-mass conservation, Brezis-Gallouet inequality and embedding theorems that  we obtained in Section \ref{sec:functionIneq}. Of course, it would not be possible to define the energy properly without the construction of the form domain $\domain{\sqrt{-A}},$ which was also done in Section \ref{sec:functionIneq}.  

As a remark on notation, we use the regularizations that we have introduced at the beginning of Section \ref{sec:extrapolatedOP}.  Recall  the regularization we have introduced for $\eta$
\begin{equation}\label{remm:etavarepsilonToEta}
	\eta_\varepsilon = \eta \cdot \mathds{1}_{\{\eta \leq 1/\varepsilon\}}
\end{equation}
which satisfies
$$||(\eta_\varepsilon-\eta) \japanbrac^{-\kappa}||_\infty \rightarrow 0$$
for any $\kappa>0$ by Theorem \ref{thm:decomphighpartHolder}.   Throughout  this Section, in the compactness proofs for simplicity when we choose a subsequence from a sequence indexed by $\varepsilon$ i.e. $\ue,$ we keep the same notation i.e. $\ue.$


\subsection{NLS in full space}\label{sec:NLS}

In this section, our purpose is to study the well-posedness  for the following stochastic (defocusing) NLS
\begin{equation}\label{equ:nlsstandard}
i \partial_t u = A u + \eta u - u |u|^r
\end{equation}
with initial conditions $u^o \in \Gamma\left(\sobolew{H}{2}{l/2}\right)$ for $ l>0,$ to be specified later. Here, by Lemma \ref{lem:estLpAH} we are able to treat powers $r <\infty$, though the classical case is with $r=2$ i.e. cubic NLS, for which we write the precise proofs for simplicity. 

For $V_\varepsilon: = \xi_\varepsilon + c_\varepsilon(x) - K_\Xi$ we introduce  the following regularized versions of \eqref{equ:nlsstandard}
\begin{equation}\label{equ:nlsdecompapp}
i\partial_t u_\varepsilon = \lap \ue + V_\varepsilon u_\varepsilon + \eta_\varepsilon u_\varepsilon - \ue |\ue|^2= A_\varepsilon \ue + \eta_\varepsilon u_\varepsilon  - \ue |\ue|^2
\end{equation}
with the regularized initial data $ \Gamma_\varepsilon ((u^o)^{\sharp}) =: u_\varepsilon^o \in \sobolew{H}{2}{l/2}.$ For an arbitrary  $T>0$ by classical results \cite{cazenave2003semilinear} the regularized PDE \eqref{equ:nlsdecompapp} has a solution
\begin{equation}\label{equ:solNLS}
\ue \in C^1([0, T], L^2)  \cap C([0, T], \sobolev{H}{2}) 
\end{equation}
for an initial data $u_\varepsilon^o \in \sobolev{H}{2}.$  Additionally, for a localized initial data (i.e. initial data in a weighted space) the solution is also localized as we will show in Proposition \ref{prop:loclemma}.  	

Moreover, for each fixed  $\varepsilon >0$ \eqref{equ:nlsdecompapp}  has the  conserved $L^2$-norm
\[
N(\ue):=\int |\ue|^2 dx = N(\ue^o) 
\]
and also the conserved energy
\begin{equation}\label{equ:energyAppNLS}
\begin{aligned}
E(\ue) &:=-\frac{1}{2} \innerprod{(-\lap + V_\varepsilon) \ue}{\ue} +\frac{1}{2} \int \eta_\varepsilon |u_\varepsilon|^2 dx + \frac{1}{4}\int |\ue|^{4} \\ &= -\frac{1}{2}\innerprod{A_\varepsilon\ue}{\ue} + \frac{1}{2}\int \eta_\varepsilon |u_\varepsilon|^2 dx + \frac{1}{4}\int |\ue|^{4}= E(\ue^o) 
\end{aligned}
\end{equation}
Similarly, we formally denote the conserved quantities for \eqref{equ:nlsstandard} by $N(u)$ and $E(u)$. 

\begin{remark}
	Observe that there is a slight abuse of notation here: in fact we should write $E_\varepsilon(\ue),$ as in the regularized energy we have $\eta_\varepsilon$ and $A_\varepsilon$ instead of $\eta$ and $A$.  However, in the sequel, which energy we refer to will be very clear from the context and we will keep this notational abuse for simplicity.
\end{remark}

Before we start the a priori estimates, below we note the convergence of  $E(\ue^o)$ and $N(\ue^o)$. This is needed because in later estimates the uniform bounds will come from the various (weighted) $L^p$ and  Sobolev norms of the initial data.  Therefore, the following Lemma will justify that the bounds we obtain in the sequel are indeed ``uniform bounds", depending on the initial data and possibly $T>0.$   

\begin{lemma}\label{lem:initialConvNLS}
		The solution $\ue$ to \eqref{equ:nlsdecompapp} satisfies the following estimate
		\begin{equation}\label{prop:energybound}
		\begin{aligned}
	-\frac{1}{2}\innerprod{A_\varepsilon \ue}{\ue} + \frac{1}{4}\int |\ue|^{4} \leq E(\ue^o) + \norm{\japanbrac^{-l}\eta_\varepsilon}{\sobolevplain{L}{\infty}} \norm{\ue}{\sobolew{L}{2}{l/2}}^2.
	\end{aligned}
	\end{equation}
	Moreover, the following convergences of mass and energy at time $t=0$ hold
	\begin{align*}
	N(\ue^o) \rightarrow N(u^o)\\
	E(\ue^o) \rightarrow E(u^o)
	\end{align*}
	as $\varepsilon \rightarrow 0.$
\end{lemma}

\begin{proof}
The estimate follows directly from the definition \eqref{equ:energyAppNLS} and conservation of energy.  The convergence of the mass follows from the choice of the initial data and Proposition \ref{prop:opclosed}.  For the convergence of the energy we go term-by-term. The convergence of the term with $A_\varepsilon$-energy follows from Proposition \ref{prop:opApp}.
For the convergence of the nonlinear term in the energy, first observe the identity
\begin{align*}
&|v|^4 - |w|^4 = (\bar{v})^2 v^2 - (\bar{w})^2 w^2  =\\
& = (\bar{v})^2 v^2 - (\bar{v})^2 w^2 + (\bar{v})^2 w^2 -   (\bar{w})^2 w^2 \\
& = (\bar{v})^2 (v+ w)(v- w) + w^2 \overline{(v-w)} \overline{(v+w)}.
\end{align*}
Taking $v = u^o, w = \ue^o$ and using Cauchy-Schwarz yields 
\begin{align*}
\left|\int \left( |u^o|^4  - |\ue^o|^4\right) dx \right| \leq \left(\norm{(\bar{v})^2 (v+ w)}{\eltwo} + \norm{ w^2 \overline{(v+w)}}{\eltwo}\right) \norm{u^o - \ue^o}{\eltwo} 
\end{align*}
and the convergence follows from Lemma \ref{lem:estLpAH} and Proposition \ref{prop:opclosed}. The convergence of the potential term can be observed similarly. 
\end{proof}

In the rest of the section, we first show a priori estimates for the solution $\ue$ of \eqref{equ:nlsdecompapp} and then prove the existence of a solution to \eqref{equ:nlsstandard} by compactness arguments in the weighted setting.  In the first such result below we generalize the argument, that appeared in \cite{DM19} with the Laplacian and $H^1$-energy, to our case with the $A_\varepsilon$-energy by using our specific embedding results and estimates from Section \ref{sec:anderson}.  This essentially shows that for a localized initial data, the  weighted $L^2$-norm of the regularized solution $\ue$  stays bounded, that is to say it stays localized too.  This result will be used to counter the growth of the unbounded part $\eta$ in later estimates.
\begin{proposition}\label{prop:loclemma}
	For $0<l\leq 1$ and arbitrary $t>0,$ the solution $\ue$ satisfies the following weighted estimate
	\begin{equation}
	\begin{aligned}\label{prop:energybound0}
&\norm{\ue}{\sobolew{L}{2}{l/2}}^2 
 &\leq \left(1+ t \right) \left[ \left(\norm{\japanbrac^{-l}\eta_\varepsilon}{\sobolevplain{L}{\infty}} +1\right)\int \japanbrac^{l} |\ue^o|^2+   E(\ue^o) \right]  e^{t \norm{\japanbrac^{-l}\eta_\varepsilon}{\sobolevplain{L}{\infty}}. }:= E'_0
	\end{aligned}
	\end{equation}
	In addition, we also obtain
		\begin{equation}\label{prop:energybound2}
	\begin{aligned}
	&-\frac{1}{2}\innerprod{A_\varepsilon \ue}{\ue} + \frac{1}{4}\int |\ue|^{4} \leq E(\ue^o) \\
	&+ \norm{\japanbrac^{-l}\eta_\varepsilon}{\sobolevplain{L}{\infty}} \left[\left(1+ t \right) \left[ \left(\norm{\japanbrac^{-l}\eta_\varepsilon}{\sobolevplain{L}{\infty}} +1\right)\int \japanbrac^{l} |\ue^o|^2+   E(\ue^o) \right]  e^{t \norm{\japanbrac^{-l}\eta_\varepsilon}{\sobolevplain{L}{\infty}} }\right] =:E_0.
	\end{aligned}
	\end{equation}

\end{proposition}
\begin{proof}
	Recall from Theorem \ref{thm:decomphighpartHolder} that $\norm{\japanbrac^{-j}\eta_\varepsilon}{\sobolevplain{L}{\infty}}<\infty$ for all $j>0$.   For $ 0<l \leq 1$, we calculate
	\begin{align*}
	\timeDer \int \japanbrac^{l} |\ue|^2 dx &= \timeDer \int \japanbrac^{l}\overline{\ue} \ue dx = \int \japanbrac^{l} \left( \overline{\ue }\timeDer \ue + \ue \timeDer \overline{\ue } \right) \\
	& = 2 \text{Re} \left( \int \japanbrac^{l}  \overline{\ue }\timeDer \ue \right)\\
	&=2 \text{Re} \left( -i \int \japanbrac^{l}   (\lap \ue + V_\varepsilon u_\varepsilon + \eta_\varepsilon u_\varepsilon + \ue |\ue|^r ) \overline{\ue }\right)\\
	& = 2 \text{Re} \left(-i \int  \japanbrac^{l - 1} \nabla \ue \overline{\ue}  \right)
	\end{align*}
	Let $\kappa< \frac{1}{2}$ and $0<\gamma<1,$ as in \eqref{def:gammalp}, such that $1-\kappa \leq \gamma.$   We continue to estimate as
	\begin{align*}
	& \timeDer \int \japanbrac^{l} |\ue|^2 dx= 2 \text{Re} \left(-i \int  \japanbrac^{l - 1} \nabla \ue \overline{\ue}  \right) \\
	& \lesssim \norm{\japanbrac^{l - 1} \nabla \ue}{\sobolev{H}{-\kappa}} \norm{\ue}{\sobolev{H}{\kappa}}\\
	& \lesssim \norm{\nabla \ue}{\sobolev{H}{-\kappa}(\japanbrac^{l -1})} \norm{\ue}{\sobolev{H}{\kappa}} \lesssim \norm{ \ue}{\sobolev{H}{1-\kappa}(\japanbrac^{l -1})} \norm{\ue}{\sobolev{H}{\kappa}}\\
	&\lesssim \norm{ \ue^\sharp}{\sobolev{H}{1-\kappa}}^2 \lesssim \norm{\ue}{\domain{\sqrt{-A_\varepsilon}}}^2  \\
	&\lesssim  \left(E(\ue^o) + \norm{\japanbrac^{-l}\eta_\varepsilon}{\sobolevplain{L}{\infty}} \norm{\ue}{\sobolew{L}{2}{l/2}}^2\right)
	\end{align*}
	where we have used Propositions \ref{lem:gamma} and \ref{lem:formdom} and Lemma \ref{lem:initialConvNLS}.
	Integrating both sides from $0$ to $t$ yields
	\begin{align*}
	&\int \japanbrac^{l} |\ue|^2 \lesssim \int \japanbrac^{l} |\ue^o|^2+  t E(\ue^o) + \int_{0}^{t}\norm{\japanbrac^{-l}\eta_\varepsilon}{\sobolevplain{L}{\infty}} \norm{\ue}{\sobolew{L}{2}{l/2}}^2\\
	&\lesssim  \left(1+ t \right) \left[ \left(\norm{\japanbrac^{-l}\eta_\varepsilon}{\sobolevplain{L}{\infty}} +1\right)\int \japanbrac^{l} |\ue^o|^2+   E(\ue^o) \right] + \int_{0}^{t}\norm{\japanbrac^{-l}\eta_\varepsilon}{\sobolevplain{L}{\infty}} \norm{\ue}{\sobolew{L}{2}{l/2}}^2.
	\end{align*}
	By an application of Gronwall we finally obtain
	\[
	\norm{\ue}{\sobolew{L}{2}{l/2}}^2 \leq \left(1+ t \right) \left[ \left(\norm{\japanbrac^{-l}\eta_\varepsilon}{\sobolevplain{L}{\infty}} +1\right)\int \japanbrac^{l} |\ue^o|^2+   E(\ue^o) \right]  e^{t \norm{\japanbrac^{-l}\eta_\varepsilon}{\sobolevplain{L}{\infty}}. }
	\]
	The estimate \eqref{prop:energybound2} follows by combining this result with \eqref{prop:energybound}.  Hence, both results follow.
\end{proof}

We are ready to prove the $\sobolev{H}{2}$-type bound for $\ue$ by using the functional inequalities, in particular Brezis-Gallouet inequality that we proved in Section \ref{sec:functionIneq}.  As a result, we also obtain certain $L^p$ and $\sobolev{H}{\gamma}$-bounds ($0<\gamma<1$  as in \eqref{def:gammalp}) for the solution $\ue$ and the time derivative $\partial_t\ue$.
\begin{proposition}[a priori-bounds] \label{prop:h2bound}
	For any fixed $t>0,$ there exists a constant $R= R(u^o, t)>0$ such that
	\[
	||A_\varepsilon u_\varepsilon(t)||_2 \leq R
	\]
	uniformly in $\varepsilon >0$.  Therefore, it follows that the norms $||\ue||_{L^\infty}$, $\norm{\ue}{L^p},$ $\norm{\ue}{\sobolev{H}{\gamma}}$ are uniformly bounded. Consequently, the $L^2$-norm of the time derivative $\norm{\partial_t \ue }{\eltwo}$ is also uniformly bounded.
\end{proposition}
\begin{proof}
	Let $\varphi_0$ be a smooth mollifier supported in the unit ball and define \(\varphi_{m}(t):=m \varphi_{0}(m t)\). For a fixed $\epsilon>0$ with  \(\frac{1}{m}<\epsilon\) let  \(I_\epsilon:=(\epsilon, T-\epsilon).
	\) 
	We convolve both sides of \eqref{equ:nlsdecompapp} with $\varphi_{m}$ to obtain
	\[
	i\partial_t u_\varepsilon\ast\varphi_{m} = A_\varepsilon \ue\ast\varphi_{m} + \eta_\varepsilon u_\varepsilon \ast\varphi_{m}+ \eta_\varepsilon u_\varepsilon \ast\varphi_{m} - (\ue |\ue|^2)\ast\varphi_{m}
	\]
	we differentiate both sides which gives
	\[
	i u_\varepsilon\ast\varphi_{m}'' = A_\varepsilon \ue\ast\varphi_{m}' + \eta_\varepsilon u_\varepsilon \ast\varphi_{m}'+ \eta_\varepsilon u_\varepsilon \ast\varphi_{m}' - (\ue |\ue|^2)\ast\varphi_{m}'
	\]
	and then pair  with $u_\varepsilon\ast\varphi_{m}'$ to obtain
	\[
	i \innerprod{u_\varepsilon\ast\varphi_{m}'' }{u_\varepsilon\ast\varphi_{m}'} = \innerprod{A_\varepsilon \ue\ast\varphi_{m}'}{u_\varepsilon\ast\varphi_{m}'} + \innerprod{\eta_\varepsilon u_\varepsilon \ast\varphi_{m}'}{u_\varepsilon\ast\varphi_{m}'} + \innerprod{- (\ue |\ue|^2)\ast\varphi_{m}'}{u_\varepsilon\ast\varphi_{m}'}
	\]
	From this equality, it follows that
	\begin{align*}
	&\frac{d}{dt} \norm{u_\varepsilon\ast\varphi_{m}'}{\eltwo}^2 = 2 \text{Re} \innerprod{- (\ue |\ue|^2)\ast\varphi_{m}'}{u_\varepsilon\ast\varphi_{m}'}\\
	&\leq \norm{ (\ue |\ue|^2)\ast\varphi_{m}'}{\eltwo} \norm{u_\varepsilon\ast\varphi_{m}'}{\eltwo}
	\end{align*}
	Observe that we have $\frac{d}{dt} \left[\ue |\ue|^2\right] = 2 |\ue|^2 \frac{d}{dt}\ue + \ue^2 \frac{d}{dt}\overline{\ue}.$ As $m \rightarrow \infty,$ we have $u_\varepsilon\ast\varphi_{m}' \rightarrow \frac{d}{dt}\ue$ and $ (\ue |\ue|^2)\ast\varphi_{m}' \rightarrow 2 |\ue|^2 \frac{d}{dt}\ue + \ue^2 \frac{d}{dt}\overline{\ue}$.  For $w_\varepsilon:= \frac{d}{dt}\ue,$ in the limit   we obtain the estimate
	\begin{equation}\label{equ:h2bound1}
	\begin{aligned}
	\frac{d}{dt} ||\we ||_2^2 \lesssim ||\we ||_2^2 \norm{\ue}{\sobolevplain{L}{\infty}}^2.
	\end{aligned}
	\end{equation}
	for all $t \in I_\epsilon.$  Since $\epsilon>0$ was arbitrary we obtain the estimate for all $t \in [0, T].$
	
	By using Theorem \ref{lem:brgaineq} we have
	\begin{align}\label{equ:h2boundBrezis}
	\norm{\ue}{\sobolevplain{L}{\infty}} \leq C_\Xi \norm{\ue}{\domain{\sqrt{-A_\varepsilon}}} \sqrt{ 1+ \log\left( 1+ \norm{\ue}{\domain{A_\varepsilon}}\right)}.
	\end{align}
	By using \eqref{equ:nlsdecompapp} we get
	\begin{equation}\label{equ:boundAe}
	\begin{aligned}
	&\norm{\ue}{\domain{A_\varepsilon}} \leq \norm{w_\varepsilon}{\eltwo} + \norm{\japanbrac^{-l}\eta}{\elinfty} \norm{\ue}{\sobolew{L}{2}{l/2}} + \norm{\ue}{L^6}^3\\
	&\leq \norm{w_\varepsilon}{\eltwo} + \norm{\japanbrac^{-l}\eta}{\elinfty} E'_0 + E_0^{3/2}.
	\end{aligned}
	\end{equation}
	where we used Proposition \ref{prop:loclemma}, \eqref{prop:energybound0},  \eqref{prop:energybound2} and Lemma \ref{lem:estLpAH}.  For simplicity we put $C_0:= 5+ \norm{\japanbrac^{-l}\eta}{\elinfty} E'_0 + E_0^{3/2}.$

	By combining  this with \eqref{equ:h2bound1} and \eqref{equ:h2boundBrezis} we obtain
	\begin{align*}
	&\frac{d}{dt}\left( ||\we ||_2^2 + C_0 \right) 
	=	\frac{d}{dt}||\we ||_2^2 \\
	&\leq 	||\we ||_2^2 \norm{\ue}{\sobolevplain{L}{\infty}}^2 \\
	&\leq  ||\we ||_2^2 \left\{ C_\Xi \norm{\ue}{\domain{\sqrt{-A_\varepsilon}}}^2 \left( \sqrt{1 +  \log\left( 1+ \norm{\ue}{\domain{A_\varepsilon}}  \right) } \right)^2\right\} \\
	& \lesssim C_\Xi \norm{\ue}{\domain{\sqrt{-A_\varepsilon}}}^2 ||\we ||_2^2 \left( 1 +   \log \left(1+ \norm{\ue}{\domain{A_\varepsilon}} \right) \right)\\
	& \leq 2 C_\Xi E_0 \left( ||\we ||_2^2 + C_0 \right)   \log\left(||\we ||_2^2 + C_0\right) 
	\end{align*}
	where in the last step we used the inequality $1+ \log (1+ |x|) \leq 2 \log(1+ |x|)$ for $|x| > 3.$  We can write the estimate in integral form as
	\begin{align*}
	&||\we ||_2^2 + C_0 \leq ||\we(0) ||_2^2 + C_0 +  2 C_\Xi E_0  \int_{0}^{t} \left( ||\we ||_2^2 + C_0 \right)   \log\left(||\we ||_2^2 + C_0\right) \\
	& \lesssim  \norm{\ue(0)}{\domain{A_\varepsilon}}^2  + \norm{\japanbrac^{-l}\eta}{\eltwo}^2 (E'_0 )^2+ E_0^{3} + C_0 +  2 C_\Xi E_0  \int_{0}^{t} \left( ||\we ||_2^2 + C_0 \right)   \log\left(||\we ||_2^2 + C_0\right)\\
	& \lesssim  \norm{(u_\varepsilon^o)^\sharp}{\sobolev{H}{2}}^2  + \norm{\japanbrac^{-l}\eta}{\eltwo}^2 (E'_0)^2 + E_0^{3} + C_0 +  2 C_\Xi E_0  \int_{0}^{t} \left( ||\we ||_2^2 + C_0 \right)   \log\left(||\we ||_2^2 + C_0\right)
	\end{align*}
	where in the last step we used Proposition \ref{lem:formdom}. By logarithmic Gronwall (Proposition \ref{lem:loggronwall}) and \eqref{equ:boundAe}  we finally obtain
	\begin{equation}\label{equ:h2boundgronwall}
	|| A_\varepsilon \ue||_2^2 \lesssim ||\we ||_2^2 + C_0  \lesssim  e^{R_1 e^{R_2t}}
	\end{equation}
	for some constants  $R_1 = R_1(t, \eta, u^o), R_2= R_2(t, \eta, u^o)$. Hence, the main estimate follows.
	
	The bound for $||\ue||_{L^\infty}$ is immediate from Lemma \ref{lem:embedinfty}.  The bound for $\norm{\ue}{L^p},$ follows from Lemma \ref{lem:estLpAH} and    \eqref{prop:energybound2}.  For $\sobolev{H}{\gamma}$-bound, for a fixed $\varepsilon>0$, through Propositions \ref{lem:formdom} and \ref{lem:gamma} we observe that
	\[
	\norm{\ue}{\sobolev{H}{\gamma}}  = \norm{A_\varepsilon^{-1}A_\varepsilon \ue}{\sobolev{H}{\gamma}} \lesssim \norm{A_\varepsilon A_\varepsilon^{-1}A_\varepsilon \ue}{\eltwo} =  \norm{A_\varepsilon \ue}{\eltwo} \leq R.
	\]
	
	The uniform bound for time derivative follows by using  \eqref{equ:nlsdecompapp} as
	\begin{align*}
	\norm{\partial_t u_\varepsilon }{\eltwo}&\leq  \norm{A_\varepsilon u_\varepsilon}{\eltwo} +\norm{ \eta_\varepsilon u_\varepsilon }{\eltwo}+ \norm{ \ue |\ue|^2}{\eltwo}\\
	&\leq R + \norm{\ue}{\sobolew{L}{2}{l/2}} \norm{\eta}{\sobolew{L}{\infty}{-l}} + E_0^{3/2}
	\end{align*}
	where the uniform boundedness follows from Proposition \ref{prop:loclemma}.  Hence, all results have been shown.
\end{proof}

In the following result we  extract a subsequence $\ue,$ that we use the same notation for simplicity, and prove the existence of a limit point $u \in C^{1/2}([0,T], \sobolew{L}{2}{\mu}) \cap W^{1, \infty}([0,T], L^2),$ for $\mu>0.$
\begin{proposition}\label{prop:cauchconvyInL2}
	Let $T>0$ be arbitrary and $0 \leq \delta < l/2 $, where $0<l\leq 1$ as in Proposition \ref{prop:loclemma}. The following results hold:
	\begin{itemize}
	\item There exists a subsequence $\ue$  and   a limit point $u$ such that we have
	\[
	\norm{\ue - u }{C([0,T], \sobolew{L}{2}{\delta}} \rightarrow 0.
	\]
	\item For $0 \leq \mu \leq \frac{l}{4}$ we have that $u \in C^{1/2}([0,T], \sobolew{L}{2}{\mu}).$  That is to say, for any $t_1, t_2 \in [0,T]$ we have 
	\[
	\norm{\japanbrac^{\mu}(u(t_1) - u(t_2))}{\eltwo} \lesssim |t_1 - t_2|^{1/2}.
	\]
	\item The limit point satisfies $u \in W^{1, \infty}([0,T], L^2)$ and (up to choosing a subsequence, which we use same notation) $\partial_{t}\ue$ converges to $\partial_{t}u$ in the weak-$\star$ sense in $L^{\infty}([0,T], L^2)$. Additionally, it also holds that  $ \norm{\partial_t u }{\eltwo} \leq C$ for a uniform constant $C = C(u^o, T)>0.$
	\end{itemize}
\end{proposition}
\begin{proof}
  
  For the first part, by Proposition \ref{prop:h2bound} we have $\ue \in W^{1,\infty}([0,T], L^2)$.  Furthermore, together with Propositions \ref{prop:loclemma}-\ref{prop:h2bound} and the interpolation for weighted spaces (Proposition  \ref{prop:interpolationWeighted}) it follows
  \begin{equation}\label{equ:nlsLimitPt1}
  \norm{\ue}{\sobolew{H}{\gamma'}{\delta'}} \leq \norm{\ue}{\sobolev{H}{\gamma}}^\theta \norm{\ue}{\sobolew{L}{2}{l/2}}^{1-\theta}
  \end{equation}
  where $\delta' = (1-\theta)\cdot l/2, \gamma' = \gamma \cdot \theta$ such that
$\delta' > \delta.$  The embedding $\sobolew{H}{\gamma'}{\delta'} \subset \sobolew{L}{2}{\delta}$ is compact by Proposition \ref{prop:compactweightEmdTri} and $\ue$ is uniformly bounded in $L^{\infty}([0, T], \sobolew{H}{\gamma'}{\delta'}) \cap W^{1,\infty}([0,T], L^2)$ by \eqref{equ:nlsLimitPt1}. By Aubin-Lions Lemma, we obtain a subsequence (that we denote with the same notation for simplicity) $\ue$ that converges to a limit $u \in C([0,T], \sobolew{L}{2}{\delta}).$ Hence, the first part follows.
	
	For the second part, we have that
	\begin{align*}
	&\norm{\ue(t_1) - \ue(t_2)}{\sobolew{L}{2}{\mu}}^2 = \innerprod{\ue(t_1) - \ue(t_2)}{\japanbrac^{2\mu} \ue(t_1) - \ue(t_2)}\\
	&\leq \norm{\ue(t_1) - \ue(t_2)}{\eltwo}\norm{\ue(t_1) - \ue(t_2)}{\sobolew{L}{2}{2\mu}}\\
	& \leq \int_{t_1}^{t_2} \norm{\partial_{t} \ue}{\eltwo} ds\left\{\norm{\ue(t_1)}{\sobolew{L}{2}{2\mu}} + \norm{\ue(t_2)}{\sobolew{L}{2}{2\mu}}\right\}\\
	&\lesssim |t_1 - t_2|
	\end{align*}
	where the boundedness follows from Propositions \ref{prop:loclemma} with $4 \mu \leq l$ and Proposition \ref{prop:h2bound}.  Finally, taking the limit as $\varepsilon \rightarrow 0$ returns the result by the first part.
	
	The last result $u \in W^{1, \infty}([0,T], L^2)$ essentially follows from the convergence $\ue \rightarrow u$ as above together with the classical result (Proposition \ref{prop:appendixWeakDiff}) applied to $\ue.$  The bound follows from the weak lower semicontinuity of the norm, from the same result (Proposition \ref{prop:appendixWeakDiff}).   The weak-$*$ convergence is immediate by boundedness.
	\end{proof}

As a preparation to the first well-posedness theorem of this section, in the following result, we  show the convergence of $A_\varepsilon \ue(t)$ by using Theorem \ref{thm:normResolventMain}.  Observe that this convergence is  important and to be handled  carefully as we have the singular behavior of the domains $\domain{A} \cap \domain{A_\varepsilon} = \{0\},$ despite the norm resolvent convergence.
\begin{proposition}\label{corr:weakConvAetoA}
	For every $0 \leq t  \leq T$ we have  $u(t) \in \domain{A}$ and the weak convergence $A_\varepsilon u_\varepsilon (t) \rightharpoonup  A u(t)$ holds in $\eltwo.$  Consequently, we have that $u \in L^\infty([0,T], \domain{A})$.
\end{proposition}
\begin{proof}
	We  use  Theorem \ref{thm:normResolventMain} and Proposition \ref{prop:h2bound} to prove the result.   We estimate
	\begin{align}\label{equ:convAe1}
	&\norm{A^{-1}A_\varepsilon^{}\ue(t) - u(t)}{\eltwo}\\
	&= \norm{A^{-1}A_\varepsilon^{}\ue(t) - A_\varepsilon^{-1}A_\varepsilon^{}\ue(t)  + A_\varepsilon^{-1}A_\varepsilon^{}\ue(t)  - u(t)}{\eltwo} \\
	&= \norm{A^{-1}A_\varepsilon^{}\ue(t) - A_\varepsilon^{-1}A_\varepsilon^{}\ue(t)  + \ue(t)  - u(t)}{\eltwo}\\
	&\leq \norm{A^{-1} - A_\varepsilon^{-1}}{\eltwo\rightarrow \eltwo} \norm{A_\varepsilon^{}\ue(t) }{\eltwo} + \norm{\ue(t) - u(t)}{\eltwo} 
	\end{align}
	and obtain $\norm{A^{-1}A_\varepsilon^{1}\ue(t) - u(t)}{\eltwo} \rightarrow 0$ as $\varepsilon \rightarrow 0.$  Observe the equality
	\begin{align*}
	&A^{-1}A_\varepsilon^{}\ue(t) - A^{-1}A_\varepsilon^{}\ue(s) \\
	&= A^{-1}A_\varepsilon^{}\ue(t)  -A_\varepsilon^{-1}A_\varepsilon^{1}\ue(t)  + A_\varepsilon^{-1}A_\varepsilon^{}\ue(t) \\
	&-A_\varepsilon^{-1}A_\varepsilon^{}\ue(s)  + A_\varepsilon^{-1}A_\varepsilon^{}\ue(s) 
	- A^{-1}A_\varepsilon^{1}\ue(s)\\
	&= (A^{-1}-A_\varepsilon^{-1}) A_\varepsilon^{}\ue(t)  + \ue(t) - \ue(s)  + (A_\varepsilon^{-1}-A^{-1}) A_\varepsilon^{}\ue(s)
	\end{align*}
	and in the $L^2$- norm we get
	\begin{align*}
	&\norm{A^{-1}A_\varepsilon^{}\ue(t) - A^{-1}A_\varepsilon^{}\ue(s)}{\eltwo}\\
	&\lesssim \norm{A^{-1}-A_\varepsilon^{-1}}{\eltwo \rightarrow\eltwo} \underset{t\in [0, T]}{\sup}\norm{A_\varepsilon^{}\ue(t)}{\eltwo}  + |t-s|\underset{s\in [0, T]}{\sup}\norm{\partial_{t}\ue(s)}{\eltwo}
	\end{align*}
	
	We have essentially shown that $A_\varepsilon^{}\ue(t)$ is an equicontinuous family in $(\domain{{A}})^{*}.$ So, there exists a $w(t) \in C([0, T], (\domain{{A}})^{*})$ such that $w: [0, T] \rightarrow L^2$ is weakly continuous and a subsequence (that we denote by the same notation) with $\norm{A_\varepsilon^{}\ue(t) - w(t)}{(\domain{{A}})^{*}} \rightarrow 0.$ 
	
	By \eqref{equ:convAe1} it follows that $u(t) = A^{-1} w(t),$ which implies $u \in \domain{{A}}$ and  also the weak convergence. By Proposition \ref{prop:bochnerLpLqbdd}, it  follows $u \in L^{\infty}([0, T],\domain{{A}} ).$
\end{proof}

As a last step to showing the (weak) existence of solutions to \eqref{equ:nlsstandard}, we show the convergence of the potential term and the nonlinearity in the following Lemma.
\begin{lemma}\label{lem:convNonlinearity}
	As $\varepsilon \rightarrow 0$ we have that
	\begin{align*}
	\eta_\varepsilon(x) \ue (t,x) &\rightarrow \eta(x) u(t,x) \\
	\ue(t) |\ue(t)|^2  &\rightarrow  u(t) |u(t)|^2 
	\end{align*}	
	uniformly in $t \in [0,T].$   Consequently, we  have $\eta(x) u(t,x), u(t) |u(t)|^2 \in  \elinfty([0,T], L^2)$. 
\end{lemma}
\begin{proof}
	For the first convergence, We have 
	\begin{align*}
	\underset{t \in [0,T]}{\sup} ||\eta_\varepsilon(t) \ue - u(t) \eta||_2 & \leq 	\underset{t \in [0,T]}{\sup} || \eta_\varepsilon \ue(t)   - \eta_\varepsilon u(t) || + 	\underset{t \in [0,T]}{\sup} || \eta_\varepsilon u(t) - u(t)\eta|| \\ & \leq  ||\eta_\varepsilon \japanbrac^{-l}||_\infty \underset{t \in [0,T]}{\sup} || u(t) - \ue(t) ||_{L^2 (\japanbrac^{l/2})}\\ &+ ||(\eta_\varepsilon-\eta) \japanbrac^{-l}||_\infty\underset{t \in [0,T]}{\sup}  || u(t)||_{L^2 (\japanbrac^{l/2})}
	\end{align*}
	which goes to zero by  \eqref{remm:etavarepsilonToEta} and Proposition \ref{prop:cauchconvyInL2}.	
	
	For the second part, we use the identity
	\begin{align}\label{equ:identityNLdiff}
	|u|^2 u - |v|^2 v = \bar u u^2 - u^2 \bar v + u^2 \bar v - \bar v v^2 = u^2 \overline {(u-v)} + \bar v (u+v) (u-v)
	\end{align}
	Then, for each $t>0$, replacing $v=\ue(t)$ we obtain 
	\begin{align*}
	&	\underset{t \in [0,T]}{\sup} \norm{|u(t)|^2 u(t) - |\ue(t)|^2 \ue(t)}{\eltwo} \\
	&\leq 	\underset{t \in [0,T]}{\sup} \norm{u(t)}{\elinfty}^2 \norm{u(t) - \ue(t)}{\eltwo} +	\underset{t \in [0,T]}{\sup} \norm{\ue(t)}{\elinfty} \norm{\ue(t) + u(t)}{\elinfty} \norm{u(t) - \ue(t)}{\eltwo}
	\end{align*}
	which tends to zero as $\varepsilon \rightarrow 0$ by Propositions \ref{prop:h2bound} and \ref{prop:cauchconvyInL2}.  Hence, the result.
\end{proof}

Finally in the next  theorem, one of the main results of this section, we show that the limit $u$ that we obtained  in Proposition \ref{prop:cauchconvyInL2} is a weak solution to \eqref{equ:nlsstandard}. 
\begin{theorem}[weak existence]\label{thm:weakExist}
	For $0<l<1,$ let $u^o \in \Gamma(\sobolew{H}{2}{l/2})$ be the initial data  and $u$ be the limit point from Proposition \ref{prop:cauchconvyInL2}.  For an arbitrary $T>0,$ the limit point $u \in W^{1, \infty} ([0,T], L^2) \cap L^\infty([0,T], \domain{A})$ satisfies the equation
	\begin{equation}\label{equ:existenceWeak}
	i \partial_t u(t, x) = A u + \eta (x) u(t,x) - u(t,x)|u(t,x)|^2
	\end{equation}
	uniquely in $\eltwo$ for almost all $t \in [0,T]$.
	
	Moreover, the equation has the conserved energy
	\(E(u):=-\frac{1}{2} \innerprod{A u}{u} +\frac{1}{2} \int \eta |u|^2 dx + \frac{1}{4}\int |u|^{4} \)
	such that $E(u(t))= E(u(0))$ for  all $t \in [0,T].$
\end{theorem}
\begin{proof}
	By Lemma \ref{lem:convNonlinearity} and   Propositions \ref{corr:weakConvAetoA} and \ref{prop:cauchconvyInL2} we obtain that for every $g \in \eltwo, \varphi \in \holdercont{c}{\infty}(0,T)$
	\begin{align*}
	&0= \int_{0}^{T} i \innerprod{ u(t, x) }{g} \varphi'  dt - \int_{0}^{T} i \innerprod{-A u - \eta (x) u(t,x) + u(t,x)|u(t,x)|^2}{g} \varphi  dt\\
	&= \lim_{\varepsilon\to0} \int_{0}^{T} i \innerprod{ \ue(t, x) }{g} \varphi'  dt - \int_{0}^{T} i \innerprod{-A_\varepsilon \ue - \eta_\varepsilon (x) \ue(t,x) + \ue(t,x)|\ue(t,x)|^2}{g} \varphi  dt
	\end{align*}
	where we have also used the dominated convergence theorem. Therefore the weak existence follows.
	
	For the uniqueness, suppose there are two solutions $u, v \in  W^{1, \infty} ([0,T], L^2) \cap L^{\infty}([0,T], \domain{A})$ that satisfy the equations
	\begin{align*}
	i \partial_t u(t, x) &= A u + \eta (x) u(t,x) + u(t,x)|u(t,x)|^2\\
	i \partial_t v(t, x) &= A v + \eta (x) v(t,x) + v(t,x)|v(t,x)|^2
	\end{align*}
By subtracting these side by side we obtain
	\begin{align*}
	&i \partial_t (u - v )= A (u- v) + \eta (x) (u - v) + u(t,x)|u(t,x)|^2 - v(t,x)|v(t,x)|^2\\
	&= A (u- v) + \eta (x) (u - v) + u(t,x)|u(t,x)|^2 - v(t,x)|v(t,x)|^2	
	\end{align*}
and then apply the standard Gronwall argument:
	\begin{align*}
	&\frac{d}{dt} \norm{u-v}{\eltwo}^2 = 2 \text{Im} \innerprod{-i\partial_t(u-v)}{(u-v)}\\
	& = \innerprod{u(t,x)|u(t,x)|^2 - v(t,x)|v(t,x)|^2}{u-v}\\
	& = \innerprod{ u^2 \overline {(u-v)} + \bar v (u+v) (u-v)}{u-v}\\
	& \lesssim \left[\norm{u}{\elinfty}^2 + \norm{v}{\elinfty}^2\right] \norm{u-v}{\eltwo}^2.
	\end{align*}
	It follows that
	\[
	\norm{u-v}{\eltwo} \lesssim \power{e}{t\left(\norm{u}{\elinfty}^2 + \norm{v}{\elinfty}^2\right)} \norm{u(0) - v(0)}{\eltwo}
	\]
	where we have also used Lemma \ref{lem:embedinfty}. It follows that $u = v$ in $L^2$, a.s. in $[0,T].$

Now, we show the conservation of energy. We first establish the energy inequality $E(u(s)) \leq E(u(0))$ for all $s\in [0,T].$  This follows from Propositions \ref{prop:cauchconvyInL2} and \ref{corr:weakConvAetoA}: we have that $ u: [0,T] \rightarrow \domain{\sqrt{-A}}$ is weakly continuous and the weak lower semicontinuity of the norm implies $E(u(s)) \leq E(u(0)).$
	
In order to show the conservation of energy, define $v(t) = \overline{{u}(s-t)}$ over the interval $[0, s].$  Observe that this is essentially running the equation over the interval with initial data $\overline{u(s)}.$  Since we have shown that $u(s) \in \domain{A} \cap \sobolew{L}{2}{\delta'}$ for some $\delta'>0,$ all previous analysis can be repeated to similarly show  $E(u(0)) \leq E(u(s)).$ By uniqueness as above, we obtain $E(u(s)) = E(u(0)).$  Hence, the result.
\end{proof}

In fact the weak existence we have shown can be improved to the strong existence of a $(\domain{\sqrt{-A}})^*$-solution by using the convergence of energies.  Before we state and prove the strong existence, we first show that  $u \in C([0,T],\domain{\sqrt{-A}}).$  
\begin{proposition}\label{prop:convEnergyU}
	For every $t \in [0,T]$	we have the strong convergence
	\[
	\norm{A_\varepsilon^{1/2} \ue(t) - A^{1/2} u(t) }{\eltwo} \rightarrow 0
	\]
	Moreover, for every $t, s \in [0,T]$ we have
	\[
	\norm{A^{1/2}u(t) - A^{1/2} u(s)}{\eltwo} \lesssim |t-s|^{1/2}.
	\]
	Consequently, we obtain that $\ue \rightarrow u$ in $C([0,T],\domain{\sqrt{-A}}).$	
\end{proposition}
\begin{proof}
	We first show $\ue \in C([0,T], \domain{\sqrt{-A_\varepsilon}}),$ for a fixed $\varepsilon>0$.  This directly follows from \eqref{equ:solNLS} and the computation 
	\begin{align*}
	&\norm{A_\varepsilon^{1/2}\ue(t) - A_\varepsilon^{1/2}\ue(s)}{\eltwo}^2\\
	&= \innerprod{A_\varepsilon^{1/2}\ue(t) - A_\varepsilon^{1/2}\ue(s) }{A_\varepsilon^{1/2}\ue(t) - A_\varepsilon^{1/2}\ue(s) }\\
	&= \innerprod{\ue(t) - \ue(s) }{A_\varepsilon\ue(t) - A_\varepsilon\ue(s) }\\
	&\lesssim \norm{\ue(t) - \ue(s)}{\eltwo} \underset{t \in [0,T]}{\sup}\norm{A_\varepsilon\ue(t) }{\eltwo}\\
	&\leq R \norm{\ue(t) - \ue(s)}{\eltwo}
	\end{align*}
	where we have used Proposition \ref{prop:h2bound}.
	
	Now, we can show the strong convergence to $A^{1/2}u.$ In Lemma \ref{lem:initialConvNLS} we obtained $E(u_\varepsilon^o) \rightarrow E(u(0)).$  In virtue of Theorem \ref{thm:weakExist} we obtain $E(u_\varepsilon(t)) \rightarrow E(u(t))$  for  all $t \in [0, T].$  This directly implies 
	\begin{equation}\label{equ:convNLSContAe0}
	\norm{A_\varepsilon^{1/2} \ue(t)}{\eltwo} \rightarrow \norm{A^{1/2} u(t)}{\eltwo}
	\end{equation}
	for  all $t \in [0, T],$ since the convergence of the other terms in the energy easily follows from considerations as in Lemma \ref{lem:initialConvNLS} and Proposition \ref{lem:convNonlinearity}. The weak convergence 
		\begin{equation}
	A_\varepsilon^{1/2} \ue(t) \rightharpoonup A^{1/2} u(t)
	\end{equation}
	already holds by Proposition \ref{corr:weakConvAetoA} and now with \eqref{equ:convNLSContAe0} the strong convergence 
	\begin{equation}\label{equ:convNLSContAe}
	A_\varepsilon^{1/2} \ue(t) \rightarrow A^{1/2} u(t)
	\end{equation}
	holds uniformly for all times.  
	
	In order to prove the continuity of $A^{1/2}u$ first observe that in fact $t \rightarrow A^{1/2}u(t)$ is continuous in $(\domain{\sqrt{-A}})^{*}$ by Proposition \ref{prop:cauchconvyInL2}.  This implies that $t \rightarrow A^{1/2}u(t)$ is weakly continuous in $L^2.$  We have shown that $t \rightarrow A_\varepsilon\ue(t)$ is continuous, so it follows that  $t \rightarrow \norm{A_\varepsilon^{1/2} \ue(t)}{\eltwo} $ is also continuous.   Then, as being the uniform limit of continuous functions $t \rightarrow \norm{A^{1/2} u(t)}{\eltwo}$ is also continuous.  Together with weak continuity this implies $t \rightarrow A^{1/2}u(t)$ is continuous.  
	
	For the H\"older continuity, with the approximations we observe  that
	\begin{align}\label{equ:compactsqrtAet}
	&\norm{A_\varepsilon^{1/2} \ue (t) - A_\varepsilon^{1/2} \ue (s)}{\eltwo}^2 = \innerprod{A_\varepsilon^{1/2}\ue(t) - A_\varepsilon^{1/2} \ue (s)}{A_\varepsilon^{1/2}\ue(t) - A_\varepsilon^{1/2} \ue (s)}\nonumber \\
	&= \innerprod{\ue(t) - \ue(s)}{A_\varepsilon \left(\ue(t) - \ue(s) \right)}\nonumber\\
	& \leq \norm{\int_{s}^{t} \partial_t \ue ds}{\eltwo} \left\{ \norm{A_\varepsilon\ue(t)}{\eltwo} +\norm{A_\varepsilon\ue(s)}{\eltwo} \right\}\nonumber\\
	& \leq |s-t| \norm{ \partial_t \ue}{\eltwo} \left\{ \norm{A_\varepsilon\ue(t)}{\eltwo} +\norm{A_\varepsilon\ue(s)}{\eltwo} \right\}
	\end{align}
	
	Similarly, observing the uniform bounds as before and taking the limit $\varepsilon \rightarrow 0$ gives the result.  In order to obtain the second part, we take the limit in \eqref{equ:compactsqrtAet}  for fixed $s,t \in [0, T]$, where we have used the uniform bounds coming from Proposition \ref{prop:h2bound}.  Hence, all results follow.
\end{proof}

Now, we are ready to show that the solution $u$ is in fact a strong $(\domain{\sqrt{-A}})^*$-solution.
\begin{theorem}[strong existence]\label{thm:mainwellposed}
		For $l>0,$ let $u^o \in \Gamma(\sobolew{H}{2}{l/2})$ be the initial data  and $u$ be the limit point from Proposition \ref{prop:cauchconvyInL2}. For an arbitrary $T>0,$ $ u \in C^{1} ([0,T], (\domain{\sqrt{-A}})^*) \cap C([0,T], \domain{\sqrt{-A}})$ satisfies the equation
	\[
	i \partial_t u(t, x) = A u + \eta (x) u(t,x) - u(t,x)|u(t,x)|^2
	\]
	 in $\domain{\sqrt{-A}})^*$ for  all $t \in [0,T]$.
	 
	 	Moreover, the equation has the conserved energy
	 \(E(u):=-\frac{1}{2} \innerprod{A u}{u} +\frac{1}{2} \int \eta |u|^2 dx + \frac{1}{4}\int |u|^{4} \)
	 such that $E(u(t)) = E(u(0))$ for  all $t \in [0,T].$
\end{theorem}

\begin{proof}
	We first obtain the continuity of the potential term and the nonlinearity, that is
	\begin{equation}\label{prop:contNLterms}
	\eta(x) u(t,x),  \  u(t,x)|u(t,x)|^2\in C^{1/2}([0,T] \eltwo).
	\end{equation}
		For the potential term $\eta u$, by using the $\mu>0$ in Proposition \ref{prop:cauchconvyInL2} we can write
	\begin{align*}
	\norm{\eta u(t_1) - \eta u(t_2)}{\eltwo} \leq \norm{\eta \japanbrac^{-\mu}}{\elinfty} \norm{\japanbrac^{\mu}(u(t_1) - u(t_2))}{\eltwo} \lesssim |t_1 - t_2|^{1/2}
	\end{align*}
	For the nonlinearity by using
	$$|u|^2 u - |v|^2 v = \bar u u^2 - u^2 \bar v + u^2 \bar v - \bar v v^2 = u^2 \bar (u-v) + \bar v (u+v) (u-v),$$
	we have
	\begin{align*}
	||u|u|^2(t_2)-u|u|^2(t_1)||_{L^2}\lesssim||u||_{L^\infty L^\infty}^2||u(t_2)-u(t_1)||_{L^2}  \lesssim |t_1 - t_2|^{1/2}.
	\end{align*}  
	Hence, \eqref{prop:contNLterms} follows.
	
	Next, we apply $A_\varepsilon^{-1/2}$ to both sides of \eqref{equ:nlsdecompapp}  to obtain
	\begin{align}\label{equ:wellposedthmRegEquRes}
	A_\varepsilon^{-1/2} \partial_t \ue &= A_\varepsilon^{1/2} \ue + A_\varepsilon^{-1/2} \eta_\varepsilon \ue + A_\varepsilon^{-1/2} u_\varepsilon |u_\varepsilon|^2.
	\end{align}
	By using Lemma \ref{lem:convNonlinearity} and Proposition \ref{prop:convEnergyU} and \eqref{prop:contNLterms},  together with Theorem \ref{thm:normResolventMain} we have that for every $t \in [0, T]$
	\begin{align*}
	A_\varepsilon^{1/2} \ue &\rightarrow A^{1/2} u \\
	A_\varepsilon^{-1/2} \eta_\varepsilon \ue &\rightarrow A^{-1/2} \eta u \\
	A_\varepsilon^{-1/2} u_\varepsilon |u_\varepsilon|^2 &\rightarrow A^{-1/2} u |u|^2
	\end{align*}
	strongly in $\eltwo$.  From this and \eqref{equ:wellposedthmRegEquRes} we obtain that 
	\[
	A_\varepsilon^{-1/2} \partial_t \ue \rightarrow w(t)
	\]
	for each $t \in [0, T]$, where $w = A^{1/2} u + A^{-1/2} \eta u + A^{-1/2} u |u|^2$.  
	
	That is to say, we have that the sequence $A_\varepsilon^{-1} \ue$ is bounded in $W^{1,\infty}([0, T], \eltwo)$ and for each $t$ we have $A_\varepsilon^{-1} \ue(t) \rightarrow A^{-1} u(t).$  By using Proposition \ref{prop:appendixWeakDiff} we obtain that $A^{-1} u \in W^{1,\infty}([0, T], \eltwo)$  and, by the uniqueness of the weak derivative, $A_\varepsilon^{-1/2} \partial_t \ue \rightarrow A^{-1/2} \partial_t u$ strongly for each $t$.

	Consequently, we obtain the limiting PDE
	\begin{align*}
	A^{-1/2} \partial_t u &= A^{1/2} u + A^{-1/2} \eta u - A^{-1/2} u |u|^2
	\end{align*}
	is satisfied in $\eltwo$ for each $t$.  Upon re-applying $A^{1/2}$ to both sides gives us the desired result
	\begin{align*}
	i\partial_t u &= A^{} u +  \eta u -  u |u|^2
	\end{align*}
	which is satisfied for each $t,$ in $ (\domain{\sqrt{-A}})^*$.   Furthermore, the continuity and strong existence essentially follows from  Proposition  \ref{prop:convEnergyU} and \eqref{prop:contNLterms}.  Hence, the result.
\end{proof}

\begin{remark}
	As a Corollary to  the poof of Theorem \ref{thm:mainwellposed} we obtain that in fact $t \rightarrow \left(i \partial_t u(t, x) - A u - \eta (x) u(t,x) + u(t,x)|u(t,x)|^2\right)$ is weakly continuous.  Hence, it follows that \eqref{equ:existenceWeak} in fact holds in $L^2$ for all $t \in [0, T].$ 
\end{remark}

We end this section with an observation regarding the weighted space regularity of the solution.
\begin{corollary}
	 Let  $\gamma<1$ be as in \eqref{def:gammalp} and $\mu$ as in Proposition   \ref{prop:cauchconvyInL2}. For $ 0 \leq \theta \leq 1$ define
	  \begin{align*}
	  \nu &= \theta \cdot \gamma\\
	  \delta &= (1-\theta)\cdot \mu.
	  \end{align*}
	 Then, the solution obtained in Theorem \ref{thm:mainwellposed} satisfies
	\[
	u \in  C([0,T], \sobolew{H}{\nu}{\delta}).
	\]
\end{corollary}
\begin{proof}
	By Theorem \ref{thm:mainwellposed} and Proposition \ref{prop:cauchconvyInL2} we readily have
	\[
	u \in C([0,T], \domain{\sqrt{-A}}) \cap C([0,T], \sobolew{L}{2}{\mu}).
	\]
	By interpolation, we obtain 
	\begin{align*}
	\norm{u}{\sobolew{H}{\nu}{\delta}} \leq \norm{u}{\sobolew{L}{2}{\mu}}^{1-\theta} \norm{u}{\sobolev{H}{\gamma}}^\theta
	\end{align*}
	where $\nu = \theta \cdot \gamma$ and  $\delta = (1-\theta)\cdot \mu.$  By using Propositions \ref{lem:formdom} and \ref{lem:gamma}  we get
	\[
	\norm{u}{\sobolev{H}{\gamma}}^\theta \leq \norm{u^\sharp}{\sobolev{H}{\gamma}}^\theta \leq \norm{A^{1/2}u}{\eltwo}^\theta.
	\]
	Hence, the result follows.	
\end{proof}

\subsection{Wave in full-space}\label{sec:nlw}

In this section, we study the well-posedness of the following stochastic cubic wave equation
\begin{align}\label{equ:waveEquMain}
\partial_t^2 u &= A u + \eta u -u^3 
\end{align}
with the initial data as in \eqref{equ:initialDataTimesPhi} below  which we explicitly construct by using our localization results from Section \ref{sec:extrapolatedOP} in the following.  Observe that, by reviewing the proofs of this section,  one can easily deduce by Lemma \ref{lem:estLpAH} that in fact our results also hold for a general nonlinearity with power in $[1, \infty).$ But for simplicity, we will write the proofs for the classical cubic nonlinearity.

This time the difficulty coming from the growth of the potential term $\eta$ is treated with the help of finite speed of propagation and the results we obtained in Section \ref{sec:extrapolatedOP} about localization with smooth functions in $\domain{A}$.  However, we will need to show in Theorem \ref{prop:convInitialData} that the important norms of this local data will converge. 

In accordance with this strategy, let us fix a localizer $\phi_R \in \holdercont{c}{\infty}(\R^2)$ with $\text{supp}(\phi_R) \subset B_R(0)$.  For this fixed $\phi_R$ we set
\begin{equation}\label{equ:constantEtainfty}
\begin{aligned}
C_R^\varepsilon &:= \norm{\mathbbm{1}_{\{\support{\phi_R}\}} \eta_\varepsilon }{\elinfty}\\
C_R &:= \norm{\mathbbm{1}_{\{\support{\phi_R}\}} \eta }{\elinfty}.
\end{aligned}
\end{equation}
Observe that, it directly follows  $C_R^\varepsilon \rightarrow C_R$  as $\varepsilon \rightarrow 0$.

 For a fixed $(u_0, u_1) \in \domain{A} \times \domain{\sqrt{-A}},$ by Theorem \ref{prop:localDomain} we have $\phi_R u_0 \in \domain{\sqrt{-A}}.$  So that,  for the localized initial data 
\begin{equation}\label{equ:initialDataTimesPhi}
(u(0), \partial_t u(0))= (\phi_R u_0, \phi_R u_1) \in \domain{\sqrt{-A}} \times \eltwo 
\end{equation}
we choose the following  regularizations
\begin{equation}\label{equ:initialDatWave}
\begin{aligned}
u_0^\varepsilon&:=\phi_R (-A_\varepsilon)^{-1}(-A) u_0\\
u_1^\varepsilon&:=\phi_R (-A_\varepsilon)^{-1/2}(-A)^{1/2} u_1
\end{aligned}
\end{equation}
and accordingly introduce the regularizations of \eqref{equ:waveEquMain} as
\begin{align}\label{equ:waveEquApp}
\partial_t^2 \ue &= A_\varepsilon \ue + \eta_\varepsilon\ue -\ue^3 \nonumber \\
u_\varepsilon(0)&=u_0^\varepsilon,\\
\partial_t{\ue}(0) &= u_1^\varepsilon.\nonumber
\end{align}
For each fixed $\varepsilon >0$, \eqref{equ:waveEquApp} has the conserved energy
\begin{align}\label{equ:conservedEnergywave}
E(\ue) = \frac{1}{2} \innerprod{\partial_{t} \ue}{\partial_{t} \ue} -\frac{1}{2}\innerprod{\ue}{A_\varepsilon\ue} -\frac{1}{2}\int \eta_\varepsilon|\ue|^2 + \frac{1}{4} \int |\ue|^4.
\end{align}
By classical results, for a fixed $\varepsilon>0$ and initial data  $(u^\varepsilon_0,u^\varepsilon_1)\in \ssp^2\times \ssp^1,$ \eqref{equ:waveEquApp}  admits a unique solution 
	\begin{equation}\label{thm:existanceWithVarep}
u_\varepsilon\in C([0,T];\ssp^2)\cap C^1([0,T];\ssp^1)\cap C^2([0,T];L^2).
\end{equation} 
As a first step to obtaining a priori bounds to \eqref{equ:waveEquApp}, in the following result we clarify the finite speed of propagation property of $\ue$ for a fixed $\varepsilon>0.$  This means that  the growth of the unbounded term $\eta$ can be suppressed by the compact support of $\ue.$
\begin{proposition}[Finite speed propagation]\label{prop:finiteSpeedPropAH}
	Let $r>0$ and $x_0 \in \mathbb{R}^2$ be fixed such that $B_{r}(x_0) \cap B_R(0) = \emptyset$. Then,  for $0 \leq t \leq r$ it follows that $u_{\varepsilon}(t,x) =0$ on $\left\{ (x,t)~|~ |x-x_0| \leq r-t \right\}$.  Consequently, for $R>0$ if the initial condition has support contained in $B_R(0)$ then after time $T>0$ the solution $\ue$  has support contained in $B_{R+T}(0)$.
\end{proposition}
\begin{proof}
	For a fixed $\varepsilon >0,$ we define $V_\varepsilon:= \eta_\varepsilon  + \xi_\varepsilon + c_\varepsilon(x)-K_\Xi$  and consider the following local energy 
	\[
	E_\varepsilon(t) = \int_{B_{r-t}(x_0)} \left( \left(\timeDer \ue \right)^2 + \ue^2 + \left| \nabla \ue \right|^2 + \frac{1}{2} |\ue|^4 \right)\, dx.
	\]
	We compute the time derivative $\timeDer E_\varepsilon$, denoting the surface measure by $dS(x)$ and using Gauss-Green we obtain:
	\begin{align*}
	&\timeDer E_\varepsilon= - \int_{\partial B_{r-t}(0)} \left( \ue^2+ (\timeDer \ue)^2+ \left| \nabla \ue \right|^2 + \frac{1}{2} |u|^4 \right)\, dS(x)\\
	&+ 2\int_{B_{r-t}(0)} \left( \ue^3  \timeDer \ue+  \ue \timeDer \ue+\timeDer \ue \timeDer^2 \ue + \nabla \ue \cdot \nabla \timeDer \ue\right)\, dx
	\end{align*}
	Through an integration by parts and the equation we can re-write the second one as
	\[
	2 \int_{B_{r-t}(0)} (2+V_\varepsilon)\ue \timeDer \ue \, dx + 2 \int_{\partial B_{r-t}(0)} \nu \cdot \nabla \ue \timeDer \ue  \, dS ,
	\]
	where $\nu$ denotes outer-normal. Since $|\nu| = 1,$ we have by the Young's inequality that
	\begin{align}\label{equ:finiteScauchschwars1}
	&2 \left|\int_{\partial B_{r-t}(0)} \nu \cdot \nabla \ue \timeDer \ue  \, dS\right| \\
	&\leq  \int_{\partial B_{r-t}(0)}  (\timeDer \ue)^2\, dS +  \int_{\partial B_{r-t}(0)} 
	\left| \nabla \ue \right|^2\, dS .
	\end{align}
	According to this we can re-arrange things as
	\begin{align*}
	&\timeDer E_\varepsilon + \int_{\partial B_{r-t}(0)}  \ue^2 dS + \frac{1}{2}\int_{\partial B_{r-t}(0)}  \ue^4 dS \\ 
	&+  \int_{\partial B_{r-t}(0)} \left( \left(\timeDer \ue \right)^2+ \left| \nabla \ue \right|^2- 2  \nu \cdot \nabla \ue \timeDer \ue  \right)\, dS(x)\\ 
	&=  \, 2 \int_{B_{r-t}(0)}(2+ V_\varepsilon) \ue \timeDer \ue , dx 
	\end{align*}
	By using \eqref{equ:finiteScauchschwars1} we have that all the additional terms on the left hand side are positive so we finally have
	\[
	\timeDer E_\varepsilon \leq  2 \int_{B_{r-t}(0)}(2+ |V_\varepsilon|)\left|\ue \timeDer \ue\right|\, dx 
	\]
	The crucial point is, on a finite ball $V_\varepsilon$ is bounded, though with a (log) blowing constant. If we put
	\[
	C_\varepsilon (x_0, r) = \underset{x \in B_{r-t}(x_0)}{\sup} \left\{2+ |V_\varepsilon(x)| \right\}
	\]
	By a Cauchy-Schwartz we readily obtain
	\[
	\timeDer E_\varepsilon \le C_\varepsilon (x_0, r) \left( \int_{B_{r-t}(0)} \ue^2\, dx \right)^{1/2}\left( \int_{B_{r-t}(0)}(\partial_t \ue)^2\, dx \right)^{1/2}
	\le C_\varepsilon E_\varepsilon .
	\]
	by Gronwall we have
	\[
	E_\varepsilon \leq E_\varepsilon(0) e^{tC_\varepsilon}
	\]
	but by assumption $E_\varepsilon(0) = 0$, hence the result.  
\end{proof}

In the sequel, the uniform bounds we will obtain  in the a priori estimates will involve the energy,  various Sobolev and $L^p$-norms of the initial data \eqref{equ:initialDataTimesPhi}.  In order to conclude that the bounds are really uniform, we must show the convergence of these norms to that of \eqref{equ:initialDatWave}.  Since, the operators $A_\varepsilon$ are changing, this convergence is not straightforward.  We prove these convergences in the following important result.
\begin{theorem}\label{prop:convInitialData}
	Let $M_R(x) := 2\phi_R(x) \nabla \phi_R(x)$ and  $0<\kappa < 1/2$ a positive constant with $1-\gamma \leq \kappa,$ where $0<\gamma<1$ is as defined in $\eqref{def:gammalp}.$  For the regularized initial data defined as in \eqref{equ:initialDatWave}, we obtain the following results:
	\begin{itemize}
		\item 	The following convergence of the energy of the regularized initial data holds
		\begin{align*}
		 \innerprod{A_\varepsilon \phi_R (-A_\varepsilon)^{-1}(-A) u_0}{\phi_R (-A_\varepsilon)^{-1}(-A) u_0} \rightarrow  \innerprod{A \phi_R  u_0}{\phi_R u_0} 
		\end{align*}
		as $\varepsilon \rightarrow 0.$
		\item  The following norm convergence holds
		\begin{align}\label{corr:strongAeuetoAu}
		\norm{A_\varepsilon^{1/2}u_0^\varepsilon - A^{1/2} \phi_R u_0}{\eltwo} \rightarrow 0.
		\end{align}
		\item 	The time initial data satisfies
		\begin{align}\label{equ:convTimeInit}
		\norm{\phi_R (-A_\varepsilon)^{-1/2}(-A)^{1/2} \partial_t u_1 - \phi_R \partial_t u_1}{\eltwo} \rightarrow 0.
		\end{align}
		\item The convergence of the energy at $t=0$ holds
		\[
		E(u_0^\varepsilon) \rightarrow E(u(0))
		\]
		as $\varepsilon \rightarrow 0$.
	\end{itemize}
\end{theorem}
\begin{proof}
We want to show the convergence 
	\begin{align*}
	\innerprod{A_\varepsilon \phi_R (-A_\varepsilon)^{-1}(-A)  u_0}{\phi_R (-A_\varepsilon)^{-1}(-A)   u_0} \rightarrow \innerprod{A \phi_R  u_0}{\phi_R u_0}
	\end{align*}
	as $\varepsilon \rightarrow 0.$
	Observe that we have 
	\begin{equation}\label{equ:convIniEqu1}
	\begin{aligned}
	&A_\varepsilon \phi_R (-A_\varepsilon)^{-1}(-A)  u_0 \\
	&= \phi_R A_\varepsilon(-A_\varepsilon)^{-1}(-A)  u_0  + 2 \nabla \phi_R \nabla \left((-A_\varepsilon)^{-1}(-A) u_0\right) + (-A_\varepsilon)^{-1}(-A) u_0 \Delta \phi_R
	\end{aligned}
	\end{equation}
	We do this term by term.  As the proof for the other terms is similar, we only give the details of the convergence
	\begin{align*}
	\innerprod{2 \nabla \phi_R \nabla \left((-A_\varepsilon)^{-1}(-A) u_0\right)}{\phi_R (-A_\varepsilon)^{-1}(-A)u_0}  \rightarrow \innerprod{2 \nabla \phi_R \nabla u_0}{\phi_R u_0}.
	\end{align*}
	We can write
	\begin{align*}
	&\innerprod{ \nabla \left((-A_\varepsilon)^{-1}(-A) u_0\right)}{M_R (-A_\varepsilon)^{-1}(-A) u_0}  - \innerprod{\nabla u_0}{M_R  u_0} \\
	&= \innerprod{\nabla \left((-A_\varepsilon)^{-1}(-A) u_0 - u_0\right)}{M_R (-A_\varepsilon)^{-1}(-A) u_0}\\ 
	&- \innerprod{\nabla u_0}{M_R(  u_0- (-A_\varepsilon)^{-1}(-A) u_0)}\\
	&= \innerprod{\nabla^{1-\kappa} \left((-A_\varepsilon)^{-1}(-A) u_0 - A^{-1} A u_0 \right)}{\nabla^{\kappa}M_R (-A_\varepsilon)^{-1}(-A) u_0}\\ 
	&- \innerprod{\nabla^{1-\kappa} u_0}{\nabla^{\kappa}M_R(  A^{-1}A u_0- (-A_\varepsilon)^{-1}(-A) u_0)}.
	\end{align*}
	By using Proposition \ref{prop:leibnizProductRule}, we estimate  as
	\begin{align*}
&|\innerprod{ \nabla \left((-A_\varepsilon)^{-1}(-A) u_0\right)}{M_R (-A_\varepsilon)^{-1}(-A) u_0}  - \innerprod{\nabla u_0}{M_R  u_0}|\\
& \leq \norm{\nabla^{1-\kappa} \left((-A_\varepsilon)^{-1}(-A) u_0 - A^{-1} A u_0 \right)}{\eltwo} \norm{\nabla^{\kappa}M_R (-A_\varepsilon)^{-1}(-A) u_0}{\eltwo} \\
&+ \norm{\nabla^{1-\kappa} u_0}{\eltwo}\norm{\nabla^{\kappa}M_R(  A^{-1}A u_0- (-A_\varepsilon)^{-1}(-A) u_0)}{\eltwo}\\
&\leq \norm{(-A_\varepsilon)^{-1}(-A) u_0 - A^{-1} A u_0}{\sobolev{H}{1-\kappa}} \norm{\nabla^{\kappa} M_R (-A_\varepsilon)^{-1}(-A) u_0}{\eltwo}\\ 
&+ \norm{u_0}{\sobolev{H}{1-\kappa}}\norm{\nabla^{\kappa}M_R (  A^{-1}A u_0- (-A_\varepsilon)^{-1}(-A) u_0)}{\eltwo}\\
&\leq  \norm{A_{\varepsilon}^{-1} - A^{-1}}{L^2\rightarrow \sobolev{H}{1-\kappa}} \norm{Au_0}{\eltwo} \\
& \times \left( \norm{\nabla^{\kappa} M_R }{\elinfty}  \norm{ (-A_\varepsilon)^{-1}(-A) u_0}{\eltwo} + \norm{M_R}{\elinfty} \norm{\nabla^{\kappa} (-A_\varepsilon)^{-1}(-A) u_0}{\eltwo} \right)\\
&+ \norm{u_0}{\sobolev{H}{1-\kappa}} \left( \norm{\nabla^{\kappa} M_R }{\elinfty}  \norm{(-A)^{-1}(-A) u_0 -  (-A_\varepsilon)^{-1}(-A) u_0}{\eltwo}\right.\\
&\left.  + \norm{M_R}{\elinfty} \norm{\nabla^{\kappa} (-A)^{-1}(-A) u_0 -  (-A_\varepsilon)^{-1}(-A) u_0}{\eltwo} \right)\\
&\lesssim \left(\norm{M_R}{\elinfty} + \norm{M_R}{\sobolev{H}{\kappa}}\right) \norm{A^{\varepsilon} - A^{-1}}{L^2\rightarrow \sobolev{H}{1-\kappa}} \norm{Au_0}{\eltwo}^2
\end{align*}
where we used Lemma \ref{lem:estLpAH} and $\kappa<1/2.$  As $\varepsilon\rightarrow 0$ the convergence follows from Theorem \ref{thm:normResolventMain}.   The other terms in \eqref{equ:convIniEqu1} are handled in a similar way and by Theorem \ref{thm:compactAHformula} the final result follows.

In order to show \eqref{corr:strongAeuetoAu}, 	observe that we have
\begin{align*}
&\norm{A_\varepsilon^{1/2}u_0^\varepsilon - A^{1/2} \phi_R u_0}{(\domain{\sqrt{-A}})^*} = \norm{A^{-1/2}\left(
	A_\varepsilon^{1/2}u_0^\varepsilon - A^{1/2} \phi_R u_0\right)}{\eltwo}\\
&  \norm{A^{-1/2} A_\varepsilon^{1/2} u_0^\varepsilon -A_\varepsilon^{-1/2}A_\varepsilon^{1/2} u_0^\varepsilon +  u_0^\varepsilon -  \phi_R u_0 }{\eltwo} \leq \norm{A^{-1/2} -A_\varepsilon^{-1/2} }{\eltwo \rightarrow \eltwo} \norm{A_\varepsilon^{1/2}u_0^\varepsilon}{\eltwo}  \norm{u_0^\varepsilon -  \phi_R u_0 }{\eltwo}
\end{align*}
We take the limit as $\varepsilon \rightarrow 0$ and use Theorem \ref{thm:normResolventMain} and Proposition \ref{thm:mainResolventHalfDerivative}.  This implies weak convergence in $L^2$ and above we have shown
\begin{align*}
\norm{A_\varepsilon^{1/2}u_0^\varepsilon}{\eltwo} \rightarrow \norm{A^{1/2} \phi_R u_0}{\eltwo}.
\end{align*}
Therefore, we obtain \eqref{corr:strongAeuetoAu}.  The proof of \eqref{equ:convTimeInit} follows similar arguments by using Theorem \ref{thm:normResolventMain}, we omit the details.

In virtue of what we have shown so far, for the convergence of the energy at $t=0$ the only non-trivial part is to show
\[
\frac{1}{4} \int |u_0^\varepsilon|^4 \rightarrow \frac{1}{4} \int |\phi_R u_0|^4.
\]
First observe by Propositions \ref{lem:formdom} and \ref{lem:gamma} we have
\begin{align*}
\norm{u_0^\varepsilon}{\sobolev{H}{\gamma}} \lesssim \norm{A_\varepsilon^{1/2} u_0^\varepsilon }{\eltwo} \leq \norm{A^{1/2} \phi_R u_0}{\eltwo} + \norm{A^{1/2} \phi_R u_0-A_\varepsilon^{1/2} u_0^\varepsilon}{\eltwo}
\end{align*}
which implies, by \eqref{corr:strongAeuetoAu}, that $u_0^\varepsilon$ is uniformly bounded  in $\sobolev{H}{\gamma}.$  Observe that both $u_0^\varepsilon$ and $\phi_R u_0$ are compactly supported.  Therefore, by weighted interpolation (Proposition \ref{prop:interpolationWeighted}) we can write the estimate
\begin{equation}\label{equ:waveInitConv1}
\begin{aligned}
&\norm{u_0^\varepsilon - \phi_R u_0}{L^4} \lesssim_{R,\delta} \norm{u_0^\varepsilon - \phi_R u_0}{\sobolew{L}{4}{\delta}}\\ 
&\leq \norm{u_0^\varepsilon - \phi_R u_0}{\sobolew{H}{\gamma'}{\delta'}} \leq \norm{u_0^\varepsilon - \phi_R u_0}{\sobolew{L}{2}{\kappa}}^{\theta}  \norm{u_0^\varepsilon - \phi_R u_0}{\sobolev{H}{\gamma}}^{1-\theta}\\
&\lesssim_{R,\kappa} \norm{u_0^\varepsilon - \phi_R u_0}{L^2}^{\theta}  \norm{u_0^\varepsilon - \phi_R u_0}{\sobolev{H}{\gamma}}^{1-\theta}
\end{aligned}
\end{equation}
where $\gamma' = (1-\theta)\cdot\gamma, \delta' = \theta\cdot\kappa$ and $0 \leq \theta \leq 1$ is chosen so that $\delta' > \delta>0$ and the embedding $\sobolew{H}{\gamma'}{\delta'} \subset \sobolew{L}{4}{\delta}$ holds.  Then, the result follows by Theorem \ref{thm:normResolventMain} and uniform boundedness in $\sobolev{H}{\gamma}$ when we take the limit in \eqref{equ:waveInitConv1} as $\varepsilon\rightarrow 0$.  Hence, all results follow.
\end{proof}

After these preparations, we are ready to show the main a priori bounds of this Section.  In addition to time regularity and energy bound, we also show the boundedness of the $\sobolev{H}{\gamma}$-norm (with $0<\gamma<1$ as in \eqref{def:gammalp}) which will then be useful to bound the $L^p$-norms. In the subsequent result, we  prove the convergence of a subsequence $\ue$ (that we denote by the same notation throughout this Section for simplicity) in $C([0, T], \eltwo) $.
\begin{proposition}[a priori bounds]\label{prop:aprioriBoundwave}
	For $t\in [0, T]$ and the constant $R>0$ in \eqref{equ:constantEtainfty} updated to $R+T,$  the unique solution $\ue$ of \eqref{equ:waveEquApp} satisfies the following estimates
\begin{equation}\label{prop:sobolevBoundWave}
	\begin{aligned}
	\norm{\partial_t \ue}{\eltwo}^2 + \norm{\sqrt{-A_\varepsilon} \ue}{\eltwo}^2 &\lesssim \left(\norm{\partial_t \ue(0)}{\eltwo}^2 + \norm{\sqrt{-A_\varepsilon} \ue(0)}{\eltwo}^2 + \norm{\ue(0)}{L^4}^4 \right) e^{(1+ C_R^\varepsilon)^2 t}\\
	\norm{\ue}{\sobolev{H}{\gamma}} &\leq \norm{\sqrt{-A_\varepsilon} \ue}{\eltwo}.
	\end{aligned}
	\end{equation}
	In particular, it follows that $\ue(t,x)$ is a bounded sequence in $C^1([0,T], \eltwo)$. 
	\end{proposition}
\begin{proof}
	By pairing both sides of the equation with $\partial_t\ue$ we obtain
	\begin{align*}
	\innerprod{\partial_t^2 \ue}{\partial_t\ue} + \innerprod{- A_\varepsilon \ue}{\partial_t\ue}  =  \innerprod{\eta_\varepsilon\ue}{\partial_t\ue}  -\innerprod{u^3}{\partial_t\ue}.
	\end{align*}
	Observe that we have
	\begin{align*}
	\innerprod{\partial_t^2 \ue}{\partial_t\ue} + \innerprod{- A_\varepsilon \ue}{\partial_t\ue} + \innerprod{u^3}{\partial_t\ue}= \frac{d}{dt}\left\{ \frac{1}{2} \norm{\partial_t \ue}{\eltwo}^2 +  \frac{1}{2} \norm{\sqrt{-A_\varepsilon} \ue}{\eltwo}^2 + \frac{1}{4} \int |\ue|^4 dx \right\}
	\end{align*}
	Then, by Young's inequality, we have that
	\begin{align*}
	&\frac{d}{dt}\left\{ \frac{1}{2} \norm{\partial_t \ue}{\eltwo}^2 +  \frac{1}{2}\norm{\sqrt{-A_\varepsilon} \ue}{\eltwo}^2 + \frac{1}{4} \int |\ue|^4 dx \right\} \leq \left| \innerprod{\eta_\varepsilon\ue}{\partial_t\ue} \right|\\
	&\leq \frac{1}{2} \norm{\eta_\varepsilon \ue}{\eltwo}^2 +  \frac{1}{2} \norm{\partial_t \ue}{\eltwo}^2\\
	& \leq (1+ C_R^\varepsilon)^2 \norm{\ue}{\eltwo}^2 + \frac{1}{2} \norm{\partial_t \ue}{\eltwo}^2
	\end{align*}
	where $C_R^\varepsilon$ is as defined in \eqref{equ:constantEtainfty}.  By Proposition \ref{lem:estLpAH} we have $\norm{\ue}{\eltwo}^2 \leq \norm{\sqrt{-A_\varepsilon} \ue}{\eltwo}^2$, in which case we can write 
	\begin{align*}
	&\frac{d}{dt} \left\{ \frac{1}{2} \norm{\partial_t \ue}{\eltwo}^2 + \norm{\sqrt{-A_\varepsilon} \ue}{\eltwo}^2 + \frac{1}{4} \int |\ue|^4 dx \right\}\\
	&\leq (1+ C_R^\varepsilon)^2  \left\{ \frac{1}{2} \norm{\partial_t \ue}{\eltwo}^2 + \norm{\sqrt{-A_\varepsilon} \ue}{\eltwo}^2 + \frac{1}{4} \int |\ue|^4 dx \right\}.
	\end{align*}
	By Gronwall this implies 
	\begin{align*}
	\norm{\partial_t \ue}{\eltwo}^2 + \norm{\sqrt{-A_\varepsilon} \ue}{\eltwo}^2 \lesssim \left(\norm{\partial_t \ue(0)}{\eltwo}^2 + \norm{\sqrt{-A_\varepsilon} \ue(0)}{\eltwo}^2 + \norm{\ue(0)}{L^4}^4 \right) e^{(1+ C_R^\varepsilon)^2 t}
	\end{align*}
		Observe that what is obtained on the right hand side is a uniform bound for any $t>0$ in virtue of Lemma \ref{lem:estLpAH}, Proposition \ref{prop:convInitialData} and the definition \eqref{equ:constantEtainfty}.
		
		By Lemma \ref{lem:estLpAH} and its proof we directly obtain
		\[
		\norm{\ue}{\sobolev{H}{\gamma}} \leq \norm{\sqrt{-A_\varepsilon} \ue}{\eltwo}.
		\]
		Hence, both estimates follow.
	\end{proof}
\begin{proposition}[convergence of $\ue$ in $L^2$]\label{prop:convwaveUeL2}
	There exists a subsequence $\ue$  and a strong limit point $u$  such that 
	\begin{align*}
	\norm{ \ue-  u}{\sobolevbochner{C}{}{[0,T]}{\eltwo}} \rightarrow 0
	\end{align*}
	Moreover the $\sobolev{H}{\gamma}$-norm of the limit $u$, namely $\norm{u}{\sobolev{H}{\gamma}},$ is uniformly bounded, which also implies the uniform boundedness of $\norm{u}{\elp}$ for $\frac{2}{1-\gamma}>p>0$.
\end{proposition}
\begin{proof}
	By Proposition \ref{prop:finiteSpeedPropAH} it follows  that $\ue \in \sobolew{L}{2}{\delta}$.  Indeed, we have
	\[
	\norm{\ue}{\sobolew{L}{2}{\delta}} \leq (1+(R+T)^2)^{\delta/2} \norm{\ue}{\eltwo}.
	\]
By interpolation (Proposition \ref{prop:interpolationWeighted}), we have the estimate
	\begin{align}\label{equ:weightSonUnif}
	\norm{\ue}{ \sobolew{H}{\gamma}{\delta'}} \leq \norm{\ue}{\sobolev{H}{\gamma'}}^\theta \norm{\ue}{\sobolew{L}{2}{\delta}}^{(1-\theta)} 
	\end{align}
	where we put
	\begin{align*}
	\gamma' = \theta \cdot \gamma + (1-\theta) \cdot 0\\
	\delta'  = (1-\theta) \cdot \delta.
	\end{align*}
	It follows by \eqref{equ:weightSonUnif} and Proposition \ref{prop:aprioriBoundwave} that $\ue$ is uniformly bounded in $\sobolew{H}{\gamma'}{\delta'}.$   Since $\partial_t \ue$ is uniformly bounded in $L^2$  and  the embedding $ \sobolew{H}{\gamma'}{\delta'} \subset L^2$ is compact by Proposition \ref{prop:compactweightEmdTri}, it follows that there exists a subsequence, that we still denote by $\ue,$ which converges to a limit point $u \in C([0,T], \eltwo)$.
	
	For the second part, we observe that $\ue \rightharpoonup u$ in $\sobolev{H}{\gamma}.$  By lower semicontinuity of the norms with respect to weak convergence,  \eqref{prop:sobolevBoundWave} and Lemma \ref{lem:estLpAH} we obtain
	\begin{equation}\label{corr:lpboundSol}
	\begin{aligned}
	\norm{u}{\elp} \leq \norm{u}{\sobolev{H}{\gamma}} \leq \left(\norm{\partial_t \ue(0)}{\eltwo}^2 + \norm{\sqrt{-A_\varepsilon}\ue(0) }{\eltwo}^2 + \norm{\ue(0)}{L^4}^4 \right) e^{(1+ C_R^\varepsilon)^2 t}
	\end{aligned}
	\end{equation}
	for $\frac{2}{1-\gamma}>p>0.$
\end{proof}

Observe that in the regularized equations \eqref{equ:waveEquApp} the linear parts of the equations, namely $A_\varepsilon,$ are also changing as $\varepsilon \rightarrow 0.$ Since $\domain{A_\varepsilon} \cap \domain{A} = \{0\},$ the operators $A$ and $A_\varepsilon$ cannot be applied to the same functions.  So, the convergence of the linear part of the equation needs to be handled carefully.  In the following result, we show the convergence of $\sqrt{-A_\varepsilon} \ue$ in $L^2.$
\begin{proposition}\label{prop:convwaveAe}
The limit point $u$ from Proposition \ref{prop:convwaveUeL2} satisfies $ u \in \elinfty([0,T], \domain{\sqrt{-A}})$ and 
	\[
	A_\varepsilon^{1/2}\ue(t) \rightharpoonup {A}^{1/2}u(t)
	\] 
	in $\eltwo$, for all times $t \in [0,T]$.
\end{proposition}

\begin{proof}
	Throughout the estimates we mainly use  Proposition \ref{thm:mainResolventHalfDerivative} and Proposition \ref{prop:aprioriBoundwave}.   We estimate
\begin{align}\label{equ:convAehalf1}
&\norm{A^{-1/2}A_\varepsilon^{1/2}\ue(t) - u(t)}{\eltwo}\\
&= \norm{A^{-1/2}A_\varepsilon^{1/2}\ue(t) - A_\varepsilon^{-1}A_\varepsilon^{1/2}\ue(t)  + A_\varepsilon^{-1}A_\varepsilon^{1/2}\ue(t)  - u(t)}{\eltwo} \\
&= \norm{A^{-1/2}A_\varepsilon^{1/2}\ue(t) - A_\varepsilon^{-1/2}A_\varepsilon^{1/2}\ue(t)  + \ue(t)  - u(t)}{\eltwo}\\
 &\leq \norm{A^{-1/2} - A_\varepsilon^{-1/2}}{\eltwo\rightarrow \eltwo} \norm{A_\varepsilon^{1/2}\ue(t) }{\eltwo} + \norm{\ue(t) - u(t)}{\eltwo} 
\end{align}
and obtain $\norm{A^{-1/2}A_\varepsilon^{1/2}\ue(t) - u(t)}{\eltwo} \rightarrow 0$ as $\varepsilon \rightarrow 0.$  By addition/subtraction of cross-terms we have
\begin{align*}
&A^{-1/2}A_\varepsilon^{1/2}\ue(t) - A^{-1/2}A_\varepsilon^{1/2}\ue(s) \\
&= A^{-1/2}A_\varepsilon^{1/2}\ue(t)  -A_\varepsilon^{-1/2}A_\varepsilon^{1/2}\ue(t)  + A_\varepsilon^{-1/2}A_\varepsilon^{1/2}\ue(t) \\
&-A_\varepsilon^{-1/2}A_\varepsilon^{1/2}\ue(s)  + A_\varepsilon^{-1/2}A_\varepsilon^{1/2}\ue(s) 
- A^{-1/2}A_\varepsilon^{1/2}\ue(s)\\
&= (A^{-1/2}-A_\varepsilon^{-1/2}) A_\varepsilon^{1/2}\ue(t)  + \ue(t) - \ue(s)  + (A_\varepsilon^{-1/2} - A^{-1/2}) A_\varepsilon^{1/2}\ue(s)
\end{align*}
and now putting the $L^2$ norms gives
\begin{align*}
&\norm{A^{-1/2}A_\varepsilon^{1/2}\ue(t) - A^{-1/2}A_\varepsilon^{1/2}\ue(s)}{\eltwo}\\
&\lesssim \norm{A^{-1/2}-A_\varepsilon^{-1/2}}{\eltwo \rightarrow\eltwo} \underset{t\in [0, T]}{\sup}\norm{A_\varepsilon^{1/2}\ue(t)}{\eltwo}  + |t-s|\underset{s\in [0, T]}{\sup}\norm{\partial_{t}\ue(s)}{\eltwo}
\end{align*}
This shows that $A_\varepsilon^{1/2}\ue(t)$ is an equicontinuous family in $(\domain{\sqrt{-A}})^{*}.$ By classical considerations there exists a $w(t) \in C([0, T], (\domain{\sqrt{-A}})^{*})$ such that $w: [0, T] \rightarrow L^2$ is weakly continuous and a subsequence (that we denote by the same notation) with $\norm{A_\varepsilon^{1/2}\ue(t) - w(t)}{(\domain{\sqrt{-A}})^{*}} \rightarrow 0.$ 
 
By \eqref{equ:convAehalf1} this implies that $u(t) = A^{-1/2} w(t),$ so $u(t) \in \domain{\sqrt{-A}}$ for all $t\in[0, T]$ and  the weak convergence is immediate. By Proposition \ref{prop:bochnerLpLqbdd}, it directly follows $u \in L^{\infty}([0, T],\domain{\sqrt{-A}} ).$
 \end{proof}

The results we have shown until now, together with the Sobolev and $L^p$-embeddings that $\domain{\sqrt{-A}}$ satisfies, are enough to show the convergence of other terms in \eqref{equ:waveEquApp}.  We gather these results in the following Proposition.
\begin{proposition}\label{pro:convEtaNL}
	For all times $t \in [0,T]$ we have the following convergences
	\begin{align*}
	&\norm{\ue^3(t) - u^3(t)}{\eltwo}  \rightarrow 0\\
	&\norm{\eta_\varepsilon \ue(t) \rightarrow \eta u(t)}{\eltwo}   \rightarrow 0.\\
	&\norm{	\partial_t\ue(t) \rightarrow \partial_{t}u(t)}{\eltwo} \rightarrow 0\\
	&A_\varepsilon^{-1/2}\partial_t^2\ue (t)  \rightharpoonup A^{-1/2}\partial_t^2 u(t) \text{\  in \  } L^2.
	\end{align*}
	Consequently, we obtain that $\partial_{t}u, u^3, \eta u \in C([0,T], L^2)$ and $\partial_t^2u \in L^{\infty} ([0,T], (\domain{\sqrt{-A}})^*).$
\end{proposition}
\begin{proof}
	We first show the convergence of the nonlinearity. Let $0<\gamma<1$ be as in Proposition \ref{prop:aprioriBoundwave}.  For $8 = \frac{4}{2-2\gamma'}$ and $\theta = \frac{\gamma'}{\gamma}$ by H\"older's inequality and interpolation we observe that
	\begin{align*}
	&\norm{\ue^3(t) - u^3(t)}{\eltwo}  =  \norm{\left( u(t) - \ue(t) \right) \left( u^2(t) + u(t)\ue(t) + \ue^2(t) \right)}{\eltwo}\\
	&\leq \norm{u(t) - \ue(t) }{L^4} \norm{u^2(t) + u(t)\ue(t) + \ue^2(t) }{L^4}\\
	&\leq  \norm{u(t) - \ue(t) }{\sobolev{H}{\gamma'}} \left[2 \norm{u}{\sobolev{H}{\gamma'}}^2+2 \norm{u}{\sobolev{H}{\gamma'}}^2\right]\\
	& \leq \left[\norm{u(t) - \ue(t) }{\eltwo}^{(1-\theta)} \norm{u(t) - \ue(t) }{\sobolev{H}{\gamma}}^\theta  \right] \left[2 \norm{u}{\sobolev{H}{\gamma'}}^2+2 \norm{u}{\sobolev{H}{\gamma'}}^2\right] \rightarrow 0
	\end{align*}
as $\varepsilon \rightarrow 0$ by Propositions \ref{prop:aprioriBoundwave} and \ref{prop:convwaveUeL2}.  The uniform boundedness of the $\sobolev{H}{\gamma}$-norms follows from Proposition \ref{prop:convwaveUeL2} and Lemma \ref{lem:estLpAH}.   In order to prove continuity one proceeds similar to above by replacing $\ue^3(t) \rightarrow u(s)^3$ for some $s\in [0, T].$  We omit the details of this similar computation.

The statements for the potential term is immediate by definition.   The convergence of the time derivative and continuity can be proved similarly to the Proposition \ref{prop:aprioriBoundwave} by using compact support and uniform boundedness of $\norm{\partial_t^2 \ue }{(\domain{\sqrt{-A_\varepsilon}})^*}$ as can be obtained from the equation, which will also be clear in the next paragraph.

For the weak convergence, we have  that $A_\varepsilon^{-1/2} \partial_t \ue $  is a bounded sequence in $W^{1,\infty}([0,T], \eltwo)$.  By using the regularized equation we also have 	
\begin{align}
A_\varepsilon^{-1/2} \partial_t^2 \ue = A_\varepsilon^{1/2} \ue + A_\varepsilon^{-1/2} \left( \eta_\varepsilon\ue \right) - A_\varepsilon^{-1/2} \ue^3.
\end{align}
Along with Proposition \ref{prop:appendixWeakDiff} and Theorem \ref{thm:normResolventMain},  this essentially implies that for all $t \in [0, T]$
$$A_\varepsilon^{-1/2}\partial_t^2 \ue(t)  \rightharpoonup A^{-1/2}\partial_t^2 u(t)$$
in $\eltwo$. Hence, all the results follow.
\end{proof}

Now, we are ready to prove the main existence and uniqueness Theorem for the cubic stochastic wave equation.  
\begin{theorem}[strong existence]\label{corr:convWaveSecDer}
	For an arbitrary time $T>0$ and initial data  as given in \eqref{equ:initialDataTimesPhi}, the stochastic NLW equation \eqref{equ:waveEquMain} has a unique compactly supported  solution 
	\begin{equation}\label{equ:existenceStrongsol}
	u \in C^{2} ([0,T], (\domain{\sqrt{-A}})^*) \cap C^{1}([0,T], L^2) \cap C([0,T], \domain{\sqrt{-A}}).
	\end{equation}
	Moreover, the equation has the conserved energy
	\(E(u):=\frac{1}{2}\left\langle\partial_{t} u, \partial_{t} u\right\rangle-\frac{1}{2}\langle u, A u\rangle+\frac{1}{2}\langle u, \eta u\rangle+\frac{1}{4} \int|u|^{4}\)
	such that $\frac{d}{dt}E(u(t)) = 0$ for all $t \in [0,T].$
	\end{theorem}
\begin{proof}
Recall that $u$ denotes the limit point from Proposition \ref{prop:convwaveUeL2}.  We first show 
\begin{equation}\label{prop:convStrongAe}
u \in C([0,T], \domain{\sqrt{-A}}).
\end{equation}	
We can write the following Duhamel formula for the regularized solution
\begin{align*}
\ue(t) &= \cos \left( t \sqrt{- A_\varepsilon} \right) \ue^0 + \frac{\sin \left( t
	\sqrt{- A_\varepsilon} \right)}{\sqrt{- A_\varepsilon}} \ue^1\\
& + \int_{0}^{t}
\frac{\sin \left( (t-s) \sqrt{- A_\varepsilon} \right)}{\sqrt{- A_\varepsilon}} \ue^3 ( s) \tmop{ds} + + \int_{0}^{t}
\frac{\sin \left( (t-s) \sqrt{- A_\varepsilon} \right)}{\sqrt{- A_\varepsilon}} \eta_\varepsilon \ue ( s) \tmop{ds}\\
&= \frac{\cos \left( t \sqrt{- A_\varepsilon} \right)}{\sqrt{- A_\varepsilon}} \sqrt{- A_\varepsilon}\ue^0 + \frac{\sin \left( t
	\sqrt{- A_\varepsilon} \right)}{\sqrt{- A_\varepsilon}} \ue^1\\
& + \int_{0}^{t}
\frac{\sin \left( (t-s) \sqrt{- A_\varepsilon} \right)}{\sqrt{- A_\varepsilon}} \ue^3 ( s) \tmop{ds} + + \int_{0}^{t}
\frac{\sin \left( (t-s) \sqrt{- A_\varepsilon} \right)}{\sqrt{- A_\varepsilon}} \eta_\varepsilon \ue ( s) \tmop{ds}.
\end{align*}
Observe that the functions $\frac{\sin \sqrt{x}}{\sqrt{x}}$ and $\frac{\cos \sqrt{x}}{\sqrt{x}}$ satisfy the assumptions for norm convergence in Corollary \ref{reedSimonCorrNorm} and we obtained the convergence $\sqrt{- A_\varepsilon}\ue^0 \rightarrow \sqrt{- A} u(0)$ in Theorem \ref{prop:convInitialData}.  Hence,  By  Corollary \ref{reedSimonCorrNorm} and Proposition \ref{pro:convEtaNL} we can take $\varepsilon\rightarrow0$ to obtain the following Duhamel formula for the solution
\begin{align*}
u(t) &= \cos \left( t \sqrt{- A} \right) u^0 + \frac{\sin \left( t
	\sqrt{- A} \right)}{\sqrt{- A}} u^1\\
& + \int_{0}^{t}
\frac{\sin \left( (t-s) \sqrt{- A} \right)}{\sqrt{- A}} u^3 ( s) \tmop{ds} + + \int_{0}^{t}
\frac{\sin \left( (t-s) \sqrt{- A} \right)}{\sqrt{- A}} \eta u ( s) \tmop{ds}.
\end{align*}

By Proposition \ref{prop:convwaveAe} we can apply $A^{1/2}$ to both sides to obtain
\begin{align*}
& A^{1/2}u(t) \\
&=  \left( \cos \left( t \sqrt{- A} \right)  \right) A^{1/2} u^0 + \left( \sin \left( t
\sqrt{- A} \right)  \right)  u^1\\
& + \int_{0}^{t} \left(  
\sin \left( (t-\tau) \sqrt{- A} \right) \right) u^3 ( \tau) \tmop{d\tau}+ \int_{0}^{t} \left(  
\sin \left( (t-\tau) \sqrt{- A} \right) \right) \eta u ( \tau) \tmop{d\tau}
\end{align*}
for all times $t \in [0, T].$ In order to show $u \in C([0,T], \domain{\sqrt{-A}}),$ for every $s\leq t \in [0,T]$ we calculate
\begin{equation}\label{equ:contAeDuh}
\begin{aligned}
& A^{1/2}u(t)  - A^{1/2}u(s)\\
&= A^{1/2} \left( \cos \left( t \sqrt{- A} \right) - \cos \left( s \sqrt{- A} \right)  \right) u^0 + \left( \sin \left( t
\sqrt{- A} \right) -\sin \left( s
\sqrt{- A} \right) \right) u^1\\
& + \int_{s}^{t}
\sin \left( (t-\tau) \sqrt{- A} \right) u^3 ( \tau) \tmop{d\tau} + \int_{0}^{s} \left(\sin \left( (t-\tau) \sqrt{- A} \right) -  
\sin \left( (s-\tau) \sqrt{- A} \right) \right) u^3 ( \tau) \tmop{d\tau}\\
& + \int_{s}^{t}
\sin \left( (t-\tau) \sqrt{- A} \right)  \eta u  ( \tau) \tmop{d\tau} + \int_{0}^{s} \left(\sin \left( (t-\tau) \sqrt{- A} \right) -  
\sin \left( (s-\tau) \sqrt{- A} \right) \right) \eta u ( \tau) \tmop{d\tau}.
\end{aligned}
\end{equation}
By using a priori bounds and the fact that $(s-\tau) \sqrt{- A} \rightarrow (t-\tau) \sqrt{- A}$ in the operator norm  as $|s-t| \rightarrow 0,$ it follows that for every $\epsilon>0$  there exists a $\delta>0$ such that  $|s-t| < \delta$ implies $\norm{A^{1/2}u(t)  - A^{1/2}u(s)}{\eltwo} < \epsilon.$  Hence, the continuity of $t \rightarrow A^{1/2}u(t)$ follows.

For existence,  by Proposition \ref{pro:convEtaNL} and its proof together with Proposition \ref{prop:convwaveAe} we obtain that $u$ is  a solution to \eqref{equ:waveEquMain}  for  all $t \in [0,T].$  We only need to clarify
\[
u \in C^2 ([0,T], (\domain{\sqrt{-A}})^*).
\]
For that, we apply $A_\varepsilon^{-1/2}$ to both sides of \eqref{equ:waveEquApp} to obtain
\begin{align}\label{equ:contSecDercont2}
A_\varepsilon^{-1/2} \partial_t^2 \ue = A_\varepsilon^{1/2} \ue + A_\varepsilon^{-1/2} \left( \eta_\varepsilon \ue \right) - A_\varepsilon^{-1/2} u_\varepsilon^3
\end{align}
and in the limit as $\varepsilon\rightarrow 0$ the continuity of $A^{-1/2} \partial_t^2u $ follows  from Theorem \ref{thm:normResolventMain}, Proposition \ref{pro:convEtaNL} and \eqref{prop:convStrongAe}. Hence, the strong existence  follows.  
	
We now show $\frac{d}{dt}E(u(t)) = 0$ for all $t \in [0,T].$ Let $\varphi_0$ be a smooth mollifier supported in the unit ball and define \(\varphi_{m}(t):=m \varphi_{0}(m t)\). For a fixed $\epsilon>0$ with  \(\frac{1}{m}<\epsilon\) let  \(I_\epsilon:=(\epsilon, T-\epsilon).
	\) We define \(u_{m}(t)=\) \(u \ast \varphi_{m}(t) \in C^{\infty}\left(I_\epsilon, \domain{\sqrt{-A}}\right).\)    Observe that we have
	\begin{align*}
	u'_m :=  \partial_t u_m = u \ast \varphi'_{m}(t)\\
	u''_m :=  \partial_t^2 u_m = u \ast \varphi''_{m}(t)
	\end{align*} 
	Applying the mollifier to both sides of the equation we obtain
	\begin{align*}
	u''_m  &= A u_m + \eta u_m -u^3\ast \varphi_{m} 
	\end{align*}
	Pairing both sides with $u'_m$ we obtain
	\begin{align*}
	&\frac{d}{dt} E_m(u(t)) =  \frac{d}{dt}\left[\frac{1}{2}\left\langle u'_m,  u'_m\right\rangle-\frac{1}{2}\langle u_m, A u_m\rangle+\frac{1}{2} \int u_m ^2\eta dx +\frac{1}{4} \int|u_m|^{4} dx\right] \\
	&=  - \langle u'_m, u^3\ast \varphi_{m} \rangle + \langle  u'_m, u_m^3 \rangle
	\end{align*}
	Taking the limit as $m\rightarrow \infty$ we obtain that
	\[
	\frac{d}{dt} E(u(t))  = 0
	\]
	for all $t \in (\epsilon, T-\epsilon).$ Since $\epsilon>0$ was arbitrary we obtained that energy is constant on $(0, T].$  By    \eqref{prop:convStrongAe} and arguments similar to those in Proposition \ref{pro:convEtaNL} it follows that $t \rightarrow E(u(t))$ is continuous.   Hence, we obtain that energy is constant on $[0, T]$ i.e. $\frac{d}{dt} E(u(t))  = 0$ for all $t \in [0, T].$
	
	Before we prove uniqueness, we need an estimate for a solution $u(x)$ as in \eqref{equ:existenceStrongsol}. We proceed similarly to our argument above. We apply the argument in  the proof of Proposition \ref{prop:aprioriBoundwave} to $u_m$ and then take the limit  as $m\rightarrow 0$ and $\epsilon\rightarrow 0$ to obtain
	\begin{equation}\label{equ:waveWeak2}
		\begin{aligned}
	&\norm{\partial_t u}{\eltwo}^2 + \norm{\sqrt{-A} u}{\eltwo}^2 \\
	&\lesssim \left(\norm{\partial_t u(0)}{\eltwo}^2 + \norm{\sqrt{-A} u(0)}{\eltwo}^2 + \norm{u(0)}{L^4}^4 \right) e^{(1+ \norm{\eta\phi}{\elinfty})^2 t}=: E_0(t, \eta, \phi, u(0))
	\end{aligned}
	\end{equation}
	where $\phi$ is a smooth compactly supported function with  $\support{u} \subset \support{\phi}$ and $\restrict{\phi}{\support{u}} = 1$.

For uniqueness, let $u_1 = u$ as above and $u_2$  be another compactly supported solution with the same initial data: $ (u_2(0), \partial_{t}u_2(0))= (u_1(0), \partial_{t}u_1(0))$. We want to show that $u_2 = u_1.$ We use the same technique as above with the mollifier $\varphi_m.$ For simplicity, we introduce the operator $B_m(f)(x) := f \ast \varphi_{m}(x).$ We put
\begin{align*}
u_1^m &= B_m u_1\\
u_2^m &= B_m u_2.
\end{align*}
With  this notation we have that
\begin{align*}
\partial_{t}^2  u_1^m - A u_1^m  - B_m (\eta u_1 )+ B_m (u_1^3(t)) &= 0 \\
\partial_{t}^2  u_2^m - A u_2^m  - B_m (\eta u_2 )+ B_m (u_2^3(t)) &= 0\end{align*}
We take the difference of the equations and pair both sides with $\partial_t(u_1^m - u_2^m)$ to obtain
\begin{align*}
&\innerprod{\partial_{t}^2 (u_1^m - u_2^m) }{\partial_t(u_1^m - u_2^m)} + \innerprod{ -A (u_1^m- u_2^m )}{\partial_t(u_1^m - u_2^m)}\\ 
&= \innerprod{ B_m(\eta( u_1- u_2 )) }{\partial_t(u_1^m - u_2^m)}+ \innerprod{B_m(u_1^3 - u_2^3) }{\partial_t(u_1^m - u_2^m)}\\
&= \innerprod{ B_m(\eta( u_1- u_2 )) }{\partial_t(u_1^m - u_2^m)}+ \innerprod{B_m(u_1^3 - u_1^2 u_2 +  u_1^2 u_2 -   u_2^3) }{\partial_t(u_1^m - u_2^m)}\\
&= \innerprod{B_m( \eta( u_1- u_2 )) }{\partial_t(u_1^m - u_2^m)}+ \innerprod{B_m(u_1^2 (u_1 - u_2) +  u_2 (u_1+ u_2) (u_1 - u_2))}{\partial_t(u_1^m - u_2^m)}
\end{align*}
We handle this term-by-term.  By using the fact that $B_m$ is a contraction on $L^2$ and Young's inequality we have
\begin{align*}
& \innerprod{B_m(u_1^2 (u_1 - u_2) +  u_2 (u_1+ u_2) (u_1 - u_2))}{\partial_t(u_1 - u_2)}\\
&= \innerprod{B_m(u_1^2 (u_1 - u_2))}{\partial_t(u_1 - u_2)} +  \innerprod{B_m(u_2 (u_1+ u_2) (u_1 - u_2))}{\partial_t(u_1 - u_2)}\\
&\leq \norm{B_m(u_1^2 (u_1 - u_2))}{L^2} \norm{\partial_t(u_1 - u_2)}{L^2} + \norm{B_m(u_2 (u_1+ u_2) (u_1 - u_2))}{L^2} \norm{\partial_t(u_1 - u_2)}{L^2}\\
&\lesssim \int u_1^4 (u_1 - u_2)^2 + 2 \int \partial_t(u_1 - u_2)^2 + \int u_2^2 (u_1+ u_2)^2 (u_1 - u_2)^2\\
& \leq \norm{ u_1^4}{\eltwo} \norm{(u_1 - u_2)^2}{\eltwo} + \norm{u_2^2 (u_1+ u_2)^2}{\eltwo}\norm{(u_1 - u_2)^2}{\eltwo} + 2 \norm{ \partial_t(u_1 - u_2)}{\eltwo}^2\\
& \leq \norm{u_1}{L^8}^{4}  \norm{u_1 - u_2}{L^4}^{2} + \norm{u_2^2 (u_1+ u_2)^2}{\eltwo}\norm{u_1 - u_2}{L^4}^2 + 2 \norm{ \partial_t(u_1 - u_2)}{\eltwo}^2\\
& \leq \norm{u_1}{L^8}^{4}  \norm{A^{1/2}(u_1 - u_2)}{L^2}^{2} + \norm{u_2^2 (u_1+ u_2)^2}{\eltwo}\norm{A^{1/2}(u_1 - u_2)}{L^2}^{2} + 2 \norm{ \partial_t(u_1 - u_2)}{\eltwo}^2\\
& \leq  \left( \norm{u_1}{L^8}^{4}   + \norm{u_2^2 (u_1+ u_2)^2}{\eltwo} +2 \right) \left( \norm{A^{1/2}(u_1 - u_2)}{L^2}^{2} +  \norm{ \partial_t(u_1 - u_2)}{\eltwo}^2 \right)
\end{align*}
where we have used Lemma \ref{lem:estLpAH}.
Similarly, for the potential term we have that
\begin{align*}
&| \innerprod{ B_m (\eta( u_1- u_2 )) }{\partial_t(u_1 - u_2)} |\\
& \leq \norm{B_m (\eta( u_1- u_2 ))}{\eltwo} \norm{\partial_t(u_1 - u_2)}{\eltwo} \\
& \lesssim \norm{\eta( u_1- u_2 )}{\eltwo} \norm{\partial_t(u_1 - u_2)}{\eltwo} \\
& \leq \norm{\eta\phi_{}}{\elinfty}^2 \norm{u_1 - u_2}{\eltwo}^2 + \norm{\partial_t(u_1 - u_2)}{\eltwo}^2\\
& \leq \left( \norm{\eta\phi_{}}{\elinfty}^2 + 1 \right) \left( \norm{A^{1/2}(u_1 - u_2)}{L^2}^{2} +  \norm{ \partial_t(u_1 - u_2)}{\eltwo}^2 \right)
\end{align*}
where $\phi$ is a smooth compactly supported function such that  for  $i = 1, 2,$ we have $\support{u_i} \subset \support{\phi}$ and $\restrict{\phi}{\support{u_i}} = 1$. By \eqref{equ:waveWeak2} and Lemma \ref{lem:estLpAH}, there exists a constant $C(E_0(t, \eta, \phi, u(0)))$ such that
\[
\norm{\eta\phi_{}}{\elinfty}^2 +  \norm{u_1}{L^8}^{4}   + \norm{u_2^2 (u_1+ u_2)^2}{\eltwo} +3 \leq C(E_0(t, \eta, \phi, u(0)))
\]
For $w_m= u_1^m - u_2^m,$ we put together the estimates to deduce
\begin{align*}
&\frac{d}{dt} \left\{ \norm{\partial_t w_m }{\eltwo}^2 + \norm{A^{1/2} w_m}{\eltwo}^2 \right\} \leq C(E_0(t, \eta, \phi, u(0))) \left\{ \norm{\partial_t w_m }{\eltwo}^2 + \norm{A^{1/2} w_m}{\eltwo}^2 \right\}.
\end{align*}
By Gronwall over $(\epsilon, T-\epsilon)$ we finally  obtain that
\begin{align*}
\norm{\partial_t w_m }{\eltwo}^2 + \norm{A^{1/2} w_m}{\eltwo}^2 \leq e^{\int_\epsilon^{T-\epsilon} C(E_0(\tau, \eta, \phi, u(0))) d\tau } \left\{ \norm{\partial_t w_m (\epsilon) }{\eltwo}^2 + \norm{A^{1/2} w_m(\epsilon)}{\eltwo}^2 \right\}.
\end{align*}
Observe that this estimate is true for arbitrarily large $m$ and arbitrarily small $\epsilon>0.$   So, we  take the limit as $m \rightarrow \infty$ and $\epsilon \rightarrow 0$ to obtain uniqueness.  Hence, all results follow.
\end{proof}

\begin{remark}\label{rem:nlwglobal}
We make a quick comment, that essentially follows from the classical arguments, for the possibly non-compactly supported initial data. Observe that our proofs can be naturally repeated when one fixes the time $T>0$ and starts with an initial data
$$
 (u_0, u_1) \in \domain{A} \times \domain{\sqrt{-A}},
 $$
 such that for all $\phi \in C_c^\infty$ we have $ (\phi u_0, \phi u_1) \in \domain{\sqrt{-A}} \times L^2.$   Then, for any fixed time $T>0$ our existence and uniqueness proofs can be done with the initial data $ (\phi_R u_0, \phi_R u_1)$ with increasing $R>0.$  This procedure leads to class of solutions $u_{R, T}(x)$ which are compactly supported but whose support sets are nested and increasing with $R>0.$  By uniqueness, it follows that $u_{R, T}(x)$ is a consistent family of solutions which extend each other on the increasing and nested family of balls.  Thereby, we have constructed a solution which is locally in 
 	\begin{equation}\label{equ:existenceStrongsol}
  C^{2} ([0,T], (\domain{\sqrt{-A}})^*) \cap C^{1}([0,T], L^2) \cap C([0,T], \domain{\sqrt{-A}}).
 \end{equation}
and extend the compactly supported solutions $u_{R, T}(x)$ for each $R>0.$ Of course, the solution satisfies the finite speed of propagation property by construction.
\end{remark}

\appendix\label{sec:appendix}

\section{APPENDIX}

\subsection{Paraproducts and Commutator estimates}\label{sec:appBasicHarm}

In this Section, we recall some notions from  Littelewood-Paley theory, paraproducts, Besov and Sobolev spaces 
and collect some results about products of distributions.

For $x = (x^1, \dots, x^d)$  and $dx = dx^1 \dots dx^d$  the Fourier transform is defined by 
\begin{align*}
\mathcal{{F}}f (k)= \hat{f}(k) := \langle f,\exp(-2\pi i \langle k, \cdot\rangle) = \int_{\mathbb{R}^d} f(x) \exp(- 2\pi i\langle k,x\rangle) dx
\end{align*}
on $\sobolev{L}{1}(\mathbb{R}^d).$

Let  $\mathscr {S} ( \mathbb{R}^d,\mathbb R)$ and  $\mathscr {S}' ( \mathbb{R}^d,\mathbb R)$ respectively denote the space of Schwartz functions and tempered distributions on $\mathbb{R}^d$.  The action  of the Fourier transform on the latter is extended in the standard way, that is for $\phi \in \mathscr {S}' (\mathbb{R}^d,\mathbb R)$ one has
\begin{align*}
\innerprod{\hat{f}}{\phi} = \innerprod{{f}}{\hat\phi} = \int_{\mathbb{R}^d \times \mathbb{R}^d} f(x)  \exp(- 2\pi i\langle \xi,x\rangle) \phi(\xi) dx d\xi . 
\end{align*}
For the inverse of the Fourier transform, that we denote by $\mathcal{F}^{-1},$ we note the following classical result.

\begin{theorem}
	The Fourier transform is an automorphism on $\mathscr {S} ( \mathbb \mathbb{R}^d,\mathbb R)$ and the inverse is equal to
	\[
	\mathcal{F}^{-1} (f) (\xi) = \mathcal{F}^{} (f) (-\xi).
	\]
\end{theorem}

The Sobolev space $\sobolev{H}{\alpha}(\mathbb R^d)$ with index $\alpha\in \mathbb{R}$ is defined as  
\begin{align*}
\sobolev{H}{\alpha}(\mathbb {R}^d) := \{f \in \mathscr {S}' ( \mathbb {R}^d;\mathbb R): 
\int_{\mathbb{R}^d} (1+|k|^2)^\alpha \, |\hat{f}(k)|^2 dk< +\infty\}\,. 
\end{align*}   
Recall that the fractional derivatives $D^\alpha$ is defined for all $\alpha>0$ as
\[
\FF (D^\alpha f) = |\xi|^\alpha \FF(f).
\]
Related to the fractional differentiation, we note the following result from \cite{muscaluVoltwoMult}. 
\begin{proposition}(Leibniz Inequality)\label{prop:leibnizProductRule}
	Let $1 < p_i, q_i \leq \infty$,  $1/r = 1/p_i+ 1/q_i$ for $i = 1, 2$.  Suppose that $\alpha >0$ and $ \frac{1}{1+\alpha}<r<\infty$.  Then, it follows that
	\[
	\norm{D^\alpha(fg) }{L^r}\leq \norm{D^\alpha f}{L^{p_1}} \norm{g}{L^{q_1}} + \norm{f}{L^{p_2}} \norm{D^\alpha g}{L^{q_2}}.
	\]
\end{proposition}

Below we recall the definition of
Littlewood-Paley blocks and introduce related notation and useful properties.

\begin{definition}\label{def:littlewoodPBl}
	Let $\chi$ and $\rho$ two non-negative smooth and compactly supported 
	radial functions $\mathbb{R}^d\to \mathbb{R}$ with the following properties
	\begin{enumerate}
		\item The support of $\chi$ is contained in a ball $\{x\in \mathbb{R}^d: |x| \le R\}$ 
		whereas the support of $\rho$ is contained in an annuli $\{x\in \mathbb{R}^d: a\le |x| \le b\}$;
		\item For all $\xi \in \mathbb{R}^d$, $\chi(\xi)+\sum_{j\ge0}\rho(2^{-j}\xi)=1$;
		\item For $j\ge 1$, $\chi  \rho(2^{-j}\cdot)\equiv 0$ and $\rho(2^{-i}\cdot) \rho(2^{-j}\cdot)\equiv 0$
		for $|i-j|> 1$. 
	\end{enumerate}
	The Littlewood-Paley blocks $(\Delta_j)_{j\geq-1}$ of a function $f\in\mathscr S'(\mathbb {T}^d)$ 
	are defined as 
	\begin{equation}\label{equ:LPblock}
	\mathscr F(\Delta_{-1} f) = \chi  \hat{f}\ \mbox{ and\ for }\ j \ge 0, \quad \mathscr F(\Delta_j f) = \rho(2^{-j}\cdot) \hat{f}.
	\end{equation}
		For $n \geq 0,$ we define $\rho_n := \rho(2^{-n}\cdot)$ and
		\[
		w_n = \begin{cases} \rho_{n-2} &\mbox{if } n \geq 2 \\
			\chi & \mbox{if } n = 1. \end{cases}
		\]
		For convenience, we use $\rho$ as a localizer on the Fourier side whereas $w_n$ is used as a localizer on the space side.  
		
 We also introduce the notation
	\begin{equation}\label{equ:rhoGreaterNot}
	\rho_{\geq L} (|x|) = \sum_{k = L}^\infty \rho_k (|x|),
	\end{equation}
where the other versions i.e. $\rho_{< L} (|x|)$ are defined accordingly.
\end{definition}

During the course of the paper we use the following property of $\rho_j$'s that we record as a remark.

\begin{remark}\label{remm:rhoFast}
	Observe that the Fourier transform of $\rho_j,$ namely $\hat\rho_j,$ is rapidly decreasing on $\{|x|>\frac{1}{2^j}\}$ in the sense that for any $N>0$ we can find a constant $C_N$ such that 
	\begin{equation}
	\mathbbm{1}_{\{|x|>\frac{1}{2^j}\}} |\hat\rho_0(x)| \leq C_N \frac{1}{1+ |x|^N}.
	\end{equation}
\end{remark}
Observe that for $f\in\mathscr S'(\mathbb {T}^d)$, the Littlewood-Paley blocks $(\Delta_j f)_{j\ge -1}$ are 
smooth functions as their Fourier transforms are compactly supported. For $f\in  \mathscr S'$ and  $j\ge 0$, we define
\begin{align*}
S_j f := \sum_{i=-1}^{j-1} \Delta_i f.
\end{align*} 
Observe that $S_j f$ converges in the sense of distributions to $f$
as $j\to \infty$. Before we give the definition of a Besov space, we introduce the following notation for Littlewood-Paley blocks and afterwards state the Bernstein's inequality.

\begin{definition}\label{def:deltaCutoffMap}
	We define
	\[
	\Delta_{> N} f := \sum_{k=N+1}^{\infty} \Delta_k f.
	\]
	Observe that by using the notation in \eqref{equ:rhoGreaterNot} this means $\Delta_{> N} f  = \mathcal{F}^{-1} \rho_{\geq N+1} \mathcal{F}f.$ The operators $\Delta_{\geq N}, \Delta_{\leq N}, \Delta_{< N}$ are defined in a similar manner.
\end{definition}

\begin{lemma}[Bernstein's inequality, \cite{gubinelli2015paracontrolled}]
	\label{lem:bernstein}Let $\mathscr{A}$ be an annulus and $\mathscr{B}$ be a
	ball. For any $k \in \mathbbm{N}, \lambda > 0,$ and $1 \le p \le
	q \le \infty$ the following statements hold.
	\begin{enumerate}
		\item if $u \in L^p (\mathbbm{R}^d) $ is such that $\tmop{supp}
		(\mathscr{F}u) \subset \lambda \mathscr{B}$ then
		\[ \underset{\mu \in \mathbbm{N}^d : | \mu | = k}{\max} \| \partial^{\mu}
		u \|_{L^q} \lesssim_k \lambda^{k + d \left( \frac{1}{p} - \frac{1}{q}
			\right)} \| u \|_{L^p} \]
		\item if $u \in L^p (\mathbbm{R}^d) $ is such that $\tmop{supp}
		(\mathscr{F}u) \subset \lambda \mathscr{A}$ then
		\[ \lambda^k \| u \|_{L^p} \lesssim_k \underset{\mu \in \mathbbm{N}^d : |
			\mu | = k}{\max} \| \partial^{\mu} u \|_{L^p} . \]
	\end{enumerate}
\end{lemma}

With the notation as we have introduced above, the Besov space with parameters $p,q \in [1,\infty),\alpha \in \mathbb{R}$ is defined as  
\begin{equation}\label{eq:Besov}
\sobolevb{\alpha}{p}{q}(\mathbb{R}^d, \mathbb{R}):=\left\{u\in \mathscr S'(\mathbb{R}^d); \quad \|u\|_{\sobolevb{\alpha}{p}{q}}=
\left(\sum_{j\geq-1}2^{jq\alpha}\|\Delta_ju\|^q_{L^p} \right)^{1/q} <+\infty\right\}.
\end{equation}
We introduce a specific notation for the Besov-H\"older spaces as
\begin{align*}
\mathscr {C}^{\alpha}:= B_{\infty,\infty}^{\alpha}
\end{align*}
which are  equipped with the norm 
$\|f\|_{\CC^\alpha}:=\|f\|_{B_{\infty,\infty}^{\alpha}}=\sup_{j\ge -1} 2^{j \alpha} 
\|\Delta_j f\|_{L^\infty}.$ In the interval $0<\alpha<1$ these spaces coincide with the classical H\"older spaces.\\

The product $fg$ of two distributions $f$ and $g$ can be decomposed into its ``paraproduct" and ``resonant product" components as  
\begin{align}\label{equ:paraproductdecomp}
fg=f\prec g+f\circ g+f\succ g
\end{align}
where
\begin{align*}
f\prec g :=\sum_{j\ge -1} \sum_{i=-1} ^{j-2} \Delta_i f \Delta_j g 
&&\text{and}&& f\succ g:= \sum_{j \ge -1} \sum_{i=-1}^{j-2} \Delta_i g \Delta_j f
\end{align*}
are often called the \textit{paraproducts} and the component 
\begin{align} \label{equ:resonantpro}
f\circ g:=\sum_{j\geq-1}\sum_{|i-j|\leq 1}\Delta_i f \Delta_j g 
\end{align} 
the \textit{resonant product}. We also introduce the notations 
\begin{align*}
f\preccurlyeq g&:=f\prec g+f\circ g\\
f\succcurlyeq g&:=f\succ g+f\circ g.
\end{align*}

For the regularities, a paraproduct  is always well defined irrespective of regularities of the functions $f, g$ whereas the resonant product is a priori only well defined if the sum of regularities is strictly greater than zero. In the following two results that we have gathered from \cite{gubinelliHofmanova2019global}, we first make the regularity properties precise and give precise estimates for paraproducts and resonant products, then we show a commutator estimate.  

\begin{proposition} \label{prop:bonyEst}
	 Let $\alpha, \beta \in \mathbb{R}$ be two arbitrary constants. The following product estimates hold:
	\begin{enumerate}
		\item If $f \in L^2$ and $g \in \CC^{\beta}$, then
		
		$\| f \prec g \|_{\ssp^{\beta - \delta}} \leq C_{\delta, \beta} \| f
		\|_{L^2} \| g \|_{\CC^{\beta}} 
		$ 
		for all $\delta > 0$.
		
		\item if $f \in \ssp^{\alpha}$ and $g \in L^{\infty}$ then
		
		$
		\| f \succ g \|_{\ssp^{\alpha}} \leq C_{\alpha, \beta} \| f \|_{\ssp^{\alpha}}
		\| g \|_{\CC^{\beta}} \hspace{0.17em} .
		$
		
		\item If $\alpha < 0$, $f \in \ssp^{\alpha}$ and $g \in \CC^{\beta}$,
		then
		
		$
		\| f \prec g \|_{\ssp^{\alpha + \beta}} \leq C_{\alpha, \beta} \| f
		\|_{\ssp^{\alpha}} \| g \|_{\CC^{\beta}} \hspace{0.17em} .
		$
		
		\item If $g \in \mathscr{C}^{\beta}$ and $f \in \ssp^{\alpha}$ for $\beta <
		0$ then
		
		$
		\|f \succ g\|_{\ssp^{\alpha + \beta}} \leq C_{\alpha, \beta} \| f
		\|_{\ssp^{\alpha}} \| g \|_{\CC^{\beta}}
		$
		
		\item If $\alpha + \beta > 0$ and $f \in \ssp^{\alpha}$ and $g \in
		\CC^{\beta}$, then
		
		$
		| | f \circ g \|_{\ssp^{\alpha + \beta}} \leq
		C_{\alpha, \beta} \| f \|_{\ssp^{\alpha}} \| g \|_{\CC^{\beta}}
		\hspace{0.17em} .
		$
	\end{enumerate}
	where $C_{\alpha, \beta}$ and $C_{\delta, \beta}$ are finite positive constants.
\end{proposition}
\begin{proposition} 
	\label{prop:commu} Let $\alpha \in (0, 1)$ and $\beta, \gamma \in \mathbb{R}$
	be constants such that $\beta + \gamma < 0$ and $\alpha + \beta + \gamma > 0.$  There exists a  trilinear operator $C$ that satisfies the following bound
	\begin{align*}
	\| C (f, g, h) \|_{\ssp^{\alpha + \beta + \gamma}} \lesssim \| f
	\|_{\ssp^{\alpha}} \| g \|_{\CC^{\beta}} \| h
	\|_{\CC^{\gamma}}
	\end{align*}
	for all $f \in \ssp^{\alpha}$, $g \in \CC^{\beta}$ and $h \in
	\CC^{\gamma}$.
	
	The restriction of the trilinear operator  $C$  to the smooth functions satisfies 
	\begin{align*}
	C (f, g, h) = (f \prec g) \circ h - f (g \circ h).
	\end{align*}
	\end{proposition}

We also need the following version of this, as we include a Fourier cut off $\Delta_{< N}$ in the ansatz \eqref{equ:ansatzmain}.  In the first result below, we formulate this modified commutator estimate and in the subsequent result, we show another commutator estimate.  As both results have similar proofs to the bounded domain case \cite{GUZ20}, we skip the proofs.
\begin{proposition}
	\label{prop:commu2} Let $\alpha \in (0, 1)$, $\beta, \gamma \in \mathbb{R}$
	such that $\beta + \gamma < 0$ and $\alpha + \beta + \gamma > 0$. 
	Then, there exists a  trilinear operator $C_N$ with the following bound
	\begin{align*}
	\| C_N (f, g, h) \|_{\ssp^{\alpha + \beta + \gamma}} \lesssim \| f
	\|_{\ssp^{\alpha}} \| g \|_{\CC^{\beta}} \| h
	\|_{\CC^{\gamma}}
	\end{align*}
	for all $f \in \ssp^{\alpha}$, $g \in \CC^{\beta}$ and $h \in
	\CC^{\gamma}$.
	
	The restriction of $C_N$  to the smooth functions satisfies 
	\begin{align*}
	C_N (f, g, h) &:= \left(\Delta_{> N} (f\prec g)\right))\circ h- f(g \circ h)
	\end{align*}
	\end{proposition}

\begin{lemma}\label{lem:circadj}
	
	Let $\alpha, \beta, \gamma \in \mathbbm{R}$ with $\alpha + \beta <0$, $\alpha + \beta + \gamma \geq
	0$, and $f \in \ssp^{\alpha}, g \in \CC^{\beta}, h \in \ssp^{\gamma},$
	then there exists a map $D (f, g, h)$ with the following bound
	\begin{equation} \label{equ:commLemma}
	| D (f, g, h) | \lesssim \| g \|_{\CC^{\beta}} \| f
	\|_{\ssp^{\alpha}} \| h \|_{\ssp^{\gamma}} . \end{equation}
	Moreover the restriction of $D (f, g, h)$  to the smooth functions $f, g, h$ is as follows:
	\begin{align*}
	D (f, g, h) = \langle f, h \circ g \rangle - \langle f \prec g, h
	\rangle.
	\end{align*}
	\end{lemma}

Here, we list some Lemmata which are used in the proof of Theorem \ref{thm:2dren}.
\begin{lemma}\label{lemm:renomemmaAbsEqui}
	Let $x, y$ be two vectors in $\mathbb{R}^2.$	If $\frac{|x|}{|y|}<a<1$ then it follows that $|x\pm y| \sim |y|.$ 
\end{lemma}
\begin{proof}
	We simply observe the chain of inequalities:
	\begin{align*}
		|y| \left|\frac{|x|}{|y|} - 1\right| = ||x| - |y|| \leq |x \pm y| \leq |x| + |y| \leq (1+ a) |y|
	\end{align*}
	then it follows that 
	\begin{align*}
		|y|( 1-a) \leq  |x \pm y| \leq (1+ a)|y|
	\end{align*}
	since  $1-a < 1- \frac{|x|}{|y|}$ by assumption.  Hence, the result follows.
\end{proof}

\begin{lemma}\label{lemm:sigmaLemma}
	Let $\sigma$ be a compactly supported smooth function. For every $0<m<1,$ the following estimate holds
	\begin{align*}
		|\sigma(\varepsilon|\tau_1|) \sigma(\varepsilon  |\tau_2|) - \sigma(\delta |\tau_1|) \sigma(\delta  |\tau_2|)| \lesssim |\varepsilon - \delta|^{m} \left(|\tau_2|^m + |\tau_1|^{m} \right).
	\end{align*}
\end{lemma}

\begin{proof}
	Since, $\sigma$ is a smooth function, for  $0<m<1$  we can write 
	\begin{align*}
		&|\sigma(\varepsilon|\tau_1|) \sigma(\varepsilon  |\tau_2|) - \sigma(\delta |\tau_1|) \sigma(\delta  |\tau_2|)| \\
		&= |\sigma(\varepsilon|\tau_1|) \sigma(\varepsilon  |\tau_2|) -\sigma(\varepsilon|\tau_1|) \sigma(\delta  |\tau_2|)+ \sigma(\varepsilon|\tau_1|) \sigma(\delta  |\tau_2|)- \sigma(\delta |\tau_1|) \sigma(\delta  |\tau_2|)|\\
		&=| \sigma(\varepsilon|\tau_1|) \left( \sigma(\varepsilon  |\tau_2|) - \sigma(\delta  |\tau_2|) \right) + \sigma(\delta  |\tau_2|) \left(\sigma(\varepsilon|\tau_1|) -\sigma(\delta |\tau_1|)\right)|\\
		&\lesssim |\varepsilon - \delta|^{m} \left(|\tau_2|^m + |\tau_1|^{m} \right).
	\end{align*}
\end{proof}

\begin{lemma}\label{lemm:tauZexchange}
	Let $\hat{w}: \mathbb{R}^2 \rightarrow \mathbb{R}$ be a function with rapid decay and let $|\tau| \geq 1$. For arbitrary $0< a <1$ and $N>0,$ there exists constants $C_{a}$ and $C_{N, a}$ such that, respectively  on the sets $\left\{(z,\tau) \in \mathbb{R}^2 \times \mathbb{R}^2~\left| ~\frac{|z|}{|\tau|} \leq a \right. \right\}$ and $\left\{(z,\tau) \in \mathbb{R}^2\times \mathbb{R}^2~\left|~\frac{|z|}{|\tau|} > a \right. \right\},$ the following estimates hold
	\begin{align*}
		\frac{1}{1+ |\tau \pm z|^2} &\leq C_a \frac{1}{1+ |\tau|^2} \\
		|\hat{w}(|z|)| \frac{1}{1+ |\tau \pm z|^2} &\leq C_{N, a} \frac{1}{1+ |\tau|^N} |\hat{w}(|z|)|^{1/2}.
	\end{align*}
\end{lemma}
\begin{proof}
	The first part directly follows from Lemma \ref{lemm:renomemmaAbsEqui}.   For the second part, since $\hat{w}$ has fast decay, for any $N>0$ and some constant $C_N$ we have 
	\[
	|\hat{w}(|z|)|^{1/2} \leq \frac{C_N}{1+ |z|^N} \leq \frac{C_{N, a}}{1+ |\tau|^N} 
	\]
	multiplying right hand side by $|\hat{w}(|z|)|^{1/2}$ and left handside by $|\hat{w}(|z|)|^{1/2} \frac{1}{1+ |\tau \pm z|^2} $ returns the result.
\end{proof}

\subsection{Theory of weighted spaces}

In this Section, we record several classical results about weighted functions spaces, our main references are \cite{triebelEntropy} and \cite{triebelFunctionthree}.   In this section we use the notation $\japanbrac^\gamma := (1+ |x|^2)^{\gamma/2}$. 

\begin{definition}\label{def:weightClass}
	Let $W^n$ denote the class of admissible weights with the following properties
	\begin{itemize}
		\item  Composed of $\holdercont{}{\infty}$ positive functions on $\mathbb{R}^n$.
		\item  For all multi-indices $\gamma$  the derivatives satisfy the bound
		\[
		|D^\gamma w(x)| \leq c_\gamma w(x).
		\]
		for all $x \in \mathbb{R}^n$.
		\item There exists constants $C>0$ and $\alpha  \geq 0$ such that
		\[
		0< w(x) < C  w(y)  \langle x-y \rangle^\alpha
		\]
		for all $x, y \in \mathbb{R}^n.$
	\end{itemize} 
\end{definition}

 The weighted $L^p$-norm is defined in the following way
\[
\norm{f}{\sobolew{L}{p}{\gamma}}:= \norm{\japanbrac^\gamma f}{\elp} = \left(\int_{\mathbb{R}^d} |f(x)|^p \japanbrac^{\gamma p} dx\right)^{1/p}.
\]
Similar to \eqref{eq:Besov}  the weighted Besov spaces are defined as
\begin{equation}\label{eq:Besovweight}
B_{p,q}^{\alpha}(\japanbrac^\gamma):=\left\{u\in \mathscr S'(\mathbb{R}^d); \quad \|u\|_{ B_{p,q}^{\alpha}}=
\left(\sum_{j\geq-1}2^{jq\alpha}\|\Delta_ju\|^q_{L^p(\japanbrac^{\gamma})} \right)^{1/q} <+\infty\right\}.
\end{equation}

One important property of these spaces is the following  
\[
\norm{f}{\sobolewb{\alpha}{p}{q}{\gamma}} \approx \norm{ \japanbrac^\gamma f}{\sobolevb{\alpha}{p}{q}}.
\]

We summarize below some fundamental results about embedding properties and interpolation for the weighted Besov spaces.

\begin{proposition}\label{prop:embeddingWeigt}
	The weighted Besov spaces as defined in \eqref{eq:Besovweight} have the following properties:
	\begin{enumerate}
		\item(Besov-Embedding)
		Let $\alpha_i, \gamma_i \in \mathbb{R}$ and $p_i, q_i \in [1,\infty]$ and suppose that $\gamma_2 \leq \gamma_1$.  If $\alpha_2 + \frac{d}{p_1} \leq \alpha_1 + \frac{d}{p_2} $ with $q_1 \leq q_2$ and $p_1 \leq p_2$ we have that
		\[
		\sobolewb{\alpha_1}{p_1}{q_1}{\gamma_1} \subseteq \sobolewb{\alpha_2}{p_2}{q_2}{\gamma_2}.
		\]
		\item($L^p$-embedding) 
		Let $\alpha>0$ and $p \in [2,\infty]$ satisfy $\frac{d}{2} - \frac{d}{p} \leq \alpha$.  The weight parameter satisfy $\gamma_i \in \mathbb{R} $ and $\gamma_2 \leq \gamma_1$. Then, we have that
		\[
		\sobolewb{\alpha}{2}{2}{\gamma_1}  \subseteq \sobolew{L}{p}{\gamma_2}
		\]
		\item(Duality)
		For $\alpha, \gamma \in \mathbb{R}$ and $p,q \in [1,\infty]$ we have that
		\[
		\left(\sobolewb{\alpha}{p}{q}{\gamma} \right)' = \sobolewb{-\alpha}{p^*}{q^*}{-\gamma}
		\]
		where $\frac{1}{p^*} + \frac{1}{p} = 1$ and $\frac{1}{q^*} + \frac{1}{q} = 1.$
	\end{enumerate}
\end{proposition}

We note the following results from \cite{triebelEntropy} regarding the compact and continuous embeddings in Besov spaces.
\begin{proposition}\label{prop:contweightEmdTri}
	Let $w_1, w_2 \in W^n$.  Let $-\infty < \alpha_2 < \alpha_1 < \infty, 0 < p_1 \leq p_2 \leq \infty, 0<q_1, q_2 \leq \infty$.  There exists a continuous embedding
	\[
	\mathcal{B}_{p_1, q_1}^{\alpha_1}({w_1}) \subset \mathcal{B}_{p_2, q_2}^{\alpha_2}({w_2})
	\]
	if and only iff 
	\begin{itemize}
		\item  $\frac{w_2(x)}{w_1(x)} \leq c <\infty.$ 
		\item $\alpha_1  - \frac{d}{p_1} > \alpha_2  - \frac{d}{p_2}.$
	\end{itemize}
	In case of the equality
	\[
	\alpha_1  - \frac{d}{p_1} = \alpha_2  - \frac{d}{p_2},
	\]
	the result also holds when $q_1 \leq q_2$.
\end{proposition}

\begin{proposition}\label{prop:compactweightEmdTri}
	Let $-\infty < \alpha_2 < \alpha_1 < \infty, 0 < p_1 \leq p_2 \leq \infty, 0<q_1, q_2 \leq \infty.$.  The embedding
	\[
	\sobolewb{\alpha_1}{p_1}{q_1}{\gamma_1} \subset \sobolewb{\alpha_2}{p_2}{q_2}{\gamma_2}
	\]
	is compact if and only iff $\gamma_1> \gamma_2 \geq 0$ (i.e. $\frac{\japanbrac^{\gamma_2}}{\japanbrac^{\gamma_1}} \rightarrow 0$ as $|x| \rightarrow \infty$) and
	\[
	\alpha_1  - \frac{d}{p_1} > \alpha_2  - \frac{d}{p_2}.
	\]
\end{proposition}
\begin{proposition}\label{prop:interpolationWeighted}
	Let $p_1, q_1, p_2, q_2 \in [1,\infty]$ and $\alpha_1, \alpha_2, \gamma_1, \gamma_2 \in \mathbb{R}$.  For a given $\theta \in [0,1]$, define the function $l: \mathbb{R}\times \mathbb{R} \rightarrow \mathbb{R}$ as $l(x,y) = (1-\theta)x + \theta y.$  Then, for $\frac{1}{p} = l(\frac{1}{p_1}, \frac{1}{p_2}), \frac{1}{q} = l(\frac{1}{q_1}, \frac{1}{q_2}), \alpha = l(\alpha_1, \alpha_2)$ and $\gamma = l(\gamma_1, \gamma_2)$, we have that
	\[
	\norm{f}{\sobolewb{\alpha}{p}{q}{\gamma}} \leq C \norm{f}{\sobolewb{\alpha_1}{p_1}{q_1}{\gamma_1}}^{1-\theta} \norm{f}{\sobolewb{\alpha_2}{p_2}{q_2}{\gamma_2}}^\theta.
	\]
\end{proposition}

We record the following result concerning the mapping properties of Littlewood-Paley blocks in the weighted Besov spaces setting, whose proof follows from \eqref{eq:Besovweight} and Definition \ref{def:deltaCutoffMap}.  

\begin{proposition}\label{prop:cutRegularity}
Let $s_2 >s_1$ and $\delta \geq 0, $   the following estimate holds
\begin{align*}
\norm{\Delta_{> N} f}{\sobolew{H}{s_1}{\delta}} \lesssim 2^{N(s_{1}- s_{2})} \norm{f}{\sobolew{H}{s_2}{\delta}}.
\end{align*}
\end{proposition}

Below, we record the Sobolev space versions of the product rule for para and resonant products from \cite{gubinelliHofmanova2019global}  which mirror their counter parts in the non-weighted case, as we have also recalled at the beginning of this Section.  In the subsequent result from \cite{DM19}, we give the general product rule for weighted Besov spaces.
\begin{proposition}\label{prop:paraproductEstWeight}(Paraproduct estimates)
	Let $\gamma_1 \gamma_2, \beta \in \mathbb{R}$ for any small $\varepsilon>0$ we have that
	\begin{equation*}
	\norm{f\prec g}{\sobolew{H}{\beta-\varepsilon}{\gamma_1+ \gamma_2}} \leq \norm{f}{\eltwo(\japanbrac^{\gamma_1})} \norm{g}{\sobolew{C}{\beta}{\gamma_2}}
	\end{equation*}
	and if $\alpha <0$ then
	\begin{equation*}
	\norm{f\prec g}{\sobolew{H}{\alpha+ \beta}{\gamma_1+ \gamma_2}} \leq \norm{f}{\sobolew{H}{\alpha}{\gamma_1}} \norm{g}{\sobolew{C}{\beta}{\gamma_2}}
	\end{equation*}
	for $\alpha + \beta>0$ we have that
	\begin{equation*}
	\norm{f\circ g}{\sobolew{H}{\alpha+ \beta}{\gamma_1+ \gamma_2}} \leq \norm{f}{\sobolew{H}{\alpha}{\gamma_1}} \norm{g}{\sobolew{C}{\beta}{\gamma_2}}
	\end{equation*}
\end{proposition}

\begin{proposition}\label{prop:productBesovWeight}
Let  \(\alpha_{1}, \alpha_{2} ,\gamma_{1}, \gamma_{2} \in \mathbb{R}\) with \(\alpha_{1}+\alpha_{2}>0\) and \( p_{1}, p_{2} \in[1, \infty]\).  For \(\alpha=\alpha_{1} \wedge \alpha_{2} \wedge\left(\alpha_{1}+ \alpha_{2}\right), \frac{1}{p}=\frac{1}{p_{1}}+\frac{1}{p_{2}}, q=p_{1} \vee p_{2}\) and \(\gamma=\gamma_{1}+\gamma_{2}\) the following estimate holds
$$
\left\|f_{1} \cdot f_{2}\right\|_{\mathcal{B}_{p, q}^{\alpha}\left(\langle x\rangle^{\gamma}\right)} \leqslant C\left\|f_{1}\right\|_{\mathcal{B}_{p_{1} p_{1}}^{\alpha_{1}}\left(\langle x\rangle^{\gamma_{1}}\right)}\left\|f_{2}\right\|_{\mathcal{B}_{p_{2}, p_{2}}^{\alpha_{2}}\left(\langle x\rangle^{\gamma_{2}}\right)}
$$
and for an arbitrary \(\kappa>0\) we have
$$
\left\|f_{1} \cdot f_{2}\right\|_{\mathcal{B}_{p, p}^{\alpha-\kappa}\left(\langle x\rangle^{\gamma}\right)} \leqslant C\left\|f_{1}\right\|_{\mathcal{B}_{p_{1}, p_{1}}^{\alpha_{1}}\left(\langle x\rangle^{\gamma_{1}}\right)}\left\|f_{2}\right\|_{\mathcal{B}_{p_{2}, p_{2}}^{\alpha_{2}}\left(\langle x\rangle^{\gamma_{2}}\right)}.
$$
\end{proposition}

Observe that the Young's inequality for convolutions also hold in the weighted setting when the weight is taken to be $\japanbrac^\kappa, \kappa>0.$  We show this in the following Lemma.

\begin{lemma}\label{lem:youngweighted}
 Let $f \in \sobolew{L}{p}{\kappa}$ and  $g \in \sobolew{L}{q}{\kappa}.$ For every $\kappa>0$ and \(p, q, r \in[1, \infty]\) satisfying
\(1+ \frac{1}{r}=\frac{1}{p}+\frac{1}{q}\)  the following estimate holds
	$$
	\|f \ast g \|_{\sobolew{L}{r}{\kappa}} \leqslant 2^{\kappa} \|f \|_{\sobolew{L}{p}{\kappa}}\|g \|_{\sobolew{L}{q}{\kappa}}.
	$$
\end{lemma}
\begin{proof}
For any $A, B \in \mathbb{R},$ observe that
\[
\frac{1+ |A+ B|^2}{2} \leq (1+ |A|^2)  (1+ |B|^2).
\]
By using this we obtain
\[
\langle x \rangle^\kappa \leq 2^{\kappa} \langle x - y \rangle^\kappa \langle y \rangle^\kappa.
\]
We estimate
\begin{align*}
& \japanbrac^{\kappa} f \ast g (x)\\
&\leq \int \japanbrac^{\kappa} |f(x-y)|  |g(y) | dy\\
&\leq 2^{\kappa}  \int \langle x - y \rangle^\kappa |f(x-y)| \langle y \rangle^\kappa |g(y) | dy = (|f|\cdot \langle \cdot \rangle^\kappa) \ast  (|g|\cdot \langle \cdot \rangle^\kappa)
\end{align*}

We obtain by the standard Young's inequality that
\begin{align*}
&\norm{\japanbrac^{\kappa} f \ast g (x)}{L^r} \leq \norm{(|f|\cdot \langle \cdot \rangle^\kappa) \ast  (|g|\cdot \langle \cdot \rangle^\kappa)}{L^r}\\
&\leq \norm{|f|\cdot \langle \cdot \rangle^\kappa}{L^p} \norm{|g|\cdot \langle \cdot \rangle^\kappa}{L^q}.
\end{align*}
Hence, the result follows.
\end{proof}

From \cite{HL15}, we note the following result regarding the weighted Besov regularity of the Gaussian white noise $Y$.
\begin{lemma}[\cite{HL15}]\label{lemm:hairerNoiseReg}
For any \( a>0\) and \(\beta \in(0,1)\) we have
$$
\mathbb{E}\left[\|Y\|_{\mathcal{C}^{\beta-2}\left(\langle x\rangle^{-a}\right)}\right]<\infty.
$$
\end{lemma}

\subsection{Some PDE facts}

Below, we list some classical facts from PDE theory and theory of operators that we need in the paper.  Our main references are \cite{GUZ20}, \cite{cazenave2003semilinear} and \cite{reedsimon2} where the reader can also consult for the proofs.
\begin{lemma}[logarithmic-Gronwall] \label{lem:loggronwall}
	Let $C_2,  \log C_1 \geq 1$ and $\theta(t)\ge1$ satisfy 
	\[
	\theta (t) \le C_1  + C_2 \int_0^t \theta (s) \log^{}
	(1+\theta (s)) \mathd s = h (t) .
	\]
	Then it follows that
	\[
	h (t) \le \exp (\log h (0) e^{C_2 t})-1 .
	\]
\end{lemma}
\begin{proposition}\label{prop:bochnerLpLqbdd}
	Let $I$ be an interval.  Let $X$ and $Y$ be two Banach spaces with $X \hookrightarrow Y$ and $1 < p, q \leq \infty$.  Let $f_n$ be a bounded sequence in $L^q(I, Y)$ and let $f : I \rightarrow Y$ be such that $f_n(t) \rightharpoonup f(t)$ for almost every $t\in I$. If $(f_n)_{n \geq 0}$ is bounded in $L^p(I, X)$ and if $X$ is reflexive, then $f \in L^p(I, X)$ and $\norm{f}{\elp(I, X)}  \leq \underset{n\rightarrow \infty}{\lim \inf}\norm{f_n}{\elp(I, X)}.$ 
\end{proposition}
\begin{proposition}\label{prop:appendixWeakDiff}
	Assume that $X$ is a reflexive space and that $1 < p \leq \infty$. Let $(f_n)_{n \in \mathbb{N}}$ bounded sequence in $W^{1, p}(I, X)$ and let $f_n(t) \rightharpoonup f(t)$ in $X$ as $n \rightarrow \infty$ for almost every $t \in I$.  It follows that $f \in W^{1, p}$ and $\norm{f}{W^{1,p}(I,X)} \leq \liminf_{n \rightarrow \infty} \norm{f_n}{W^{1,p}(I,X)}$.
\end{proposition}
\begin{proposition} \label{prop:selfadj}
	A closed symmetric operator on a Hilbert space $\mathcal{A}$ is self-adjoint if it has at least one real number in its resolvent set.
\end{proposition}

\bibliography{FullSpaceAndeson.bib}{}
\bibliographystyle{plain}

\end{document}